\documentclass[12pt]{report}

\usepackage{amsmath}    
\usepackage{amsfonts}   
\usepackage{amssymb}    
\usepackage{amsthm}     
\usepackage{mathrsfs}   
\usepackage{graphicx}   
\usepackage{color}
\usepackage{bm}
\usepackage[square,sort,comma,numbers]{natbib}
\usepackage{algorithm}
\usepackage{algpseudocode}
\usepackage{algpascal}
\usepackage[left=3cm,right=3cm,top=3cm,bottom=3cm]{geometry}
\usepackage[titletoc]{appendix}

\newlength{\vslength}
\newcommand{\iid}{{\it i.i.d.}}
\newcommand{\Frechet}{Fr\'{e}chet}
\newcommand{\Hajek}{H\'{a}jek}
\newcommand{\Holder}{H\"{o}lder}

\newtheorem{theorem}{Theorem}[section]

\newtheorem{lemma}{Lemma}[section]

\newtheorem{proposition}{Proposition}[section]

\newtheorem{corollary}{Corollary}[section]

\theoremstyle{remark}

\newtheorem{example}{\bf Example}[section]
\newenvironment{Exmp}{\begin{example}}{\hfill$\Box$\end{example}}

\renewenvironment{proof}{\noindent{\bf Proof}\;\;}{\qed}

\newcommand{\score}{\dot{\ell}}

\newcommand{\effscore}{\tilde{\ell}}
\newcommand{\effFI}{{\tilde{I}}}
\newcommand{\effI}{{\effFI}}

\newcommand{\scrP}{{\mathscr P}}

\def\cA{\mathcal A}
\def\cB{\mathcal B}

\def\cD{\mathcal D}
\def\cF{\mathcal F}
\def\cG{\mathcal G}
\def\cH{\mathcal H}

\def\cL{\mathcal L}
\def\cM{\mathcal M}

\def\cR{\mathcal R}

\def\cX{\mathcal X}

\newcommand{\bfc}{{\bf c}}

\newcommand{\bfe}{{\bf e}}

\newcommand{\bh}{{\bf h}}

\newcommand{\bfs}{{\bf s}}

\newcommand{\bx}{{\bf x}}
\newcommand{\bX}{{\bf X}}
\newcommand{\by}{{\bf y}}
\newcommand{\bY}{{\bf Y}}
\newcommand{\bz}{{\bf z}}

\newcommand{\bepsilon}{\mbox{\boldmath{$\epsilon$}}}

\newcommand{\btheta}{\bm{\theta}}

\newcommand{\bmu}{\mbox{\boldmath{$\mu$}}}

\newcommand{\bbG}{{\mathbb G}}

\newcommand{\bbP}{{\mathbb P}}
\newcommand{\bbR}{{\mathbb R}}
\newcommand{\bbX}{{\mathbb X}}
\newcommand{\bbZ}{{\mathbb Z}}

\newcommand{\bc}{\begin{center}}
\newcommand{\ec}{\end{center}}
\newcommand{\be}{\begin{equation}}
\newcommand{\ee}{\end{equation}}
\newcommand{\ba}{\begin{array}}
\newcommand{\ea}{\end{array}}
\newcommand{\bean}{\setlength\arraycolsep{2pt}\begin{eqnarray*}}
\newcommand{\eean}{\end{eqnarray*}}
\newcommand{\bea}{\setlength\arraycolsep{2pt}\begin{eqnarray}}
\newcommand{\eea}{\end{eqnarray}}
\newcommand{\ben}{\begin{enumerate}}
\newcommand{\een}{\end{enumerate}}
\newcommand{\bed}{\begin{itemize}}
\newcommand{\eed}{\end{itemize}}

\DeclareMathOperator*{\argmax}{argmax}
\DeclareMathOperator*{\argmin}{argmin}

\begin{document}

\thispagestyle{empty}
  \begin{center}
        \hyphenpenalty=10000
        \Large\bf\expandafter{The Semiparametric Bernstein--von Mises Theorem for Models with Symmetric Error}
  \end{center}

    \vskip1.5in

  \begin{center}
    \large\bf{By} \\ \vskip0.2in
    {\large\bf Minwoo Chae}
  \end{center}

    \vskip1.5in

  \begin{center}
    \bf A Thesis \\
        Submitted in fulfillment of the
        requirements\\ for the degree of \\
    \expandafter{Doctor of Philosophy} \\
        in
    \expandafter{Statistics} \\
  \end{center}

    \vskip0.5in
    \begin{center}
        \bf Department of \expandafter{Statistics}\\
       College of Natural Sciences \\
        \expandafter{Seoul National University}\\
        \expandafter{February, 2015}
    \end{center}

\vfil\eject

\setcounter{page}{1} %
\pagenumbering{roman}

\addcontentsline{toc}{chapter}{Abstract} 
\chapter*{Abstract} 

In a smooth semiparametric model, the marginal posterior distribution of the finite dimensional parameter of interest
is expected to be asymptotically equivalent to the sampling distribution of frequentist's efficient estimators.
This is the assertion of the so-called Bernstein-von Mises theorem, and recently,
it has been proved in many interesting semiparametric models.
In this thesis, we consider the semiparametric Bernstein-von Mises theorem in some models which have symmetric errors.
The simplest example of these models is the symmetric location model that has 1-dimensional location parameter
and unknown symmetric error. 
Also, the linear regression and random effects models are included provided the error distribution is symmetric.
The condition required for nonparametric priors on the error distribution is very mild, and 
the most well-known Dirichlet process mixture of normals works well.
As a consequence, Bayes estimators in these models satisfy frequentist criteria of optimality
such as \Hajek-Le Cam convolution theorem.
The proof of the main result requires that the expected log likelihood ratio has a certain quadratic expansion,
which is a special property of symmetric densities.
One of the main contribution of this thesis is to provide an efficient estimator of regression coefficients
in the random effects model, in which it is unknown to estimate the coefficients efficiently 
because the full likelihood inference is difficult.
Our theorems imply that the posterior mean or median is efficient, and the result from
numerical studies also shows the superiority of Bayes estimators.
For practical use of our main results,
efficient Gibbs sampler algorithms based on symmetrized Dirichlet process mixtures are provided.

\bigskip

\noindent{\bf Keywords:}
Semiparametric Bernstein-von Mises theorem,
Linear regression with symmetric error,
mixture of normal densities,
Dirichlet process mixture  \\

\tableofcontents
\listoftables\addcontentsline{toc}{chapter}{List of tables}
\listoffigures\addcontentsline{toc}{chapter}{List of figures}

\newpage

\pagenumbering{arabic}

\chapter{Introduction}

It is a fundamental problem in statistics to make an optimal decision for a given statistical problem.
Every statistical inference is based on the observed data, but
we rarely know about the sampling distribution of a given estimator with finite samples.
As a result, it is extremely restrictive in actual exercises to find an optimal estimator.
In many interesting examples, however, the sampling distribution of an estimator converges to
a specific distribution as the number of observations increases, and it is possible to estimate this limit.
Therefore statistical inferences and theories on optimality are usually based on these asymptotic properties.
For example, Fisher conjectured that the maximum likelihood estimator would be efficient,
and in the middle of the 20th century many statisticians solved this problem under different assumptions.

In this thesis, we prove that statistical inferences based on Bayesian posterior distributions
are efficient in some semiparametric problems.
More specifically, we prove the semiparametric Bernstein-von Mises (BvM) theorem in some models which have
symmetric errors.
In theses models, the observation $\bX = (X_{1}, \ldots, X_{n})^T$ can be represented by
\be \label{eq:general_model}
	\bX = \bmu + \bepsilon,
\ee
where $\bmu = (\mu_{1}, \ldots, \mu_{n})^T$ and $\bepsilon = (\epsilon_{1}, \ldots, \epsilon_{n})^T$.
Here $\bmu$ is non-random and can be parametrized by the location parameter $\theta \in \bbR$ or
the regression coefficient $\beta\in\bbR^p$ with explanatory variables.
The error distribution is assumed to be symmetric in the sense that $\bepsilon \stackrel{d}{=} -\bepsilon$,
where $\stackrel{d}{=}$ means that two distributions of both sides are the same.
Since the error distribution is completely unknown except its symmetricity,
these are semiparametric estimation problems.
Symmetric location model, linear regression with unknown error, and
random effects model are included in these models, all of them give very useful implication.
The assertion of the semiparametric BvM theorem is, roughly speaking, that
the marginal posterior distribution for the parameter of interest is asymptotically normal
centered on an efficient estimator with variance the inverse of Fisher information matrix.
As a result statistical inferences based on the posterior distribution satisfy
frequentist criteria of optimality.

Even before the 1970s, putting a prior, which is always a delicate and difficult problem in Bayesian analysis,
posed conceptual, mathematical, and practical difficulties in infinite dimensional models.
A discovery of Dirichlet processes by \citet{ferguson1973bayesian} was a breakthrough.
This prior is easy to elicit, has a large support, and the posterior distribution is analytically tractable.
After this discovery, there have been a growing interest on Bayesian nonparametric statistics,
and for the last few decades there was remarkable development in many fields science and industry.
Useful models, priors and efficient computational algorithms has been developed in broad areas,
and convenient statistical softwares have been provided to analyze data of various forms.
Especially the development of Markov chain Monte Carlo algorithms,
along with the improvement of computing technologies, boosts Bayesian methodologies
because they are very flexible and can be applied complex and highly structured data, while
frequentist methods may have some difficulties to analyze such data.
More recently, there was considerable progress on asymptotic behavior of posterior distributions.

While the BvM theorem for parametric Bayesian models is well established 
(e.g. \citet{le1986asymptotic, kleijn2012bernstein}),
the non- or semiparametric BvM theorem has been actively studied recently 
after \citet{cox1993analysis} and \citet{freedman1999wald} 
gave negative examples on the non- or semiparametric BvM theorem.
The BvM theorems for various models including survival models (\citet{kim2004bernstein, kim2006bernstein}), 
Gaussian regression models with increasing number of parameters 
(\citet{bontemps2011bernstein, johnstone2010high, ghosal1999asymptotic}), 
discrete probability measures (\citet{boucheron2009bernstein})
have been proved.
In addition, general sufficient conditions for non- or semiparametric BvM theorems are given by
\citet{shen2002asymptotic, castillo2012semiparametric, bickel2012semiparametric, castillo2013general}.
Those sufficient conditions, however, are rather abstract and not easy to verify.
In particular, it is difficult to apply these general theories to models with unknown errors
in which the quadratic expansion of the likelihood ratio is not straightforward.
More recently, \citet{castillo2013nonparametric, castillo2014bernstein} have established
fully infinite-dimmensional BvM theorems by considering weaker topologies than the classical $L_p$ spaces.

We consider the semiparametric BvM theorem in models of the form \eqref{eq:general_model}.
There is a vast amount of literature about the frequentist's efficient estimation in these models.
For example, for the symmetric location model, where $X_i$'s are i.i.d. with mean $\theta$,
we refer to \citet{beran1978efficient, stone1975adaptive, sacks1975asymptotically} and references therein.
More elegant and practical method using kernel density estimation can be found in \citet{park_lecturenote}.
This approach can be easily extended for estimating the regression coefficient in the linear regression model.
\citet{bickel1982adaptive} also provide an efficient estimator for the linear regression model.

Bayesian analysis of the symmetric location model has also received much attention since
\citet{diaconis1986inconsistent} showed that a careless choice of a prior on $P$ leads to an inconsistent posterior.
Posterior consistency of the symmetric location model with Polya tree prior is proved by \citet{ghosal1999consistent}, 
posterior consistency of more general regression model has been studied by \citet{amewou2003posterior, tokdar2006posterior},
and posterior convergence rate with Dirichlet process mixture prior has been derived by \citet{ghosal2007convergence}.
But the efficiency of the Bayes estimators, the semiparametric BvM theorem, in such models has not been proved yet.
We prove that this is true when the error distribution is endowed with a Dirichlet process mixture of normals prior.
Furthermore, we have shown that the Bayes estimators in random effect models,
where the error and random effects distributions are unknown except that they are symmetric about the origin,
are also efficient.
In the random effects model, it is known that the full likelihood inference is difficult because it can be obtained
by integrating out the random effects.

The remainder of the thesis is organized as follows.
In Chapter \ref{chap:review}, we review three topics in asymptotic statistics which are
prerequisites for our main results.
In Section \ref{sec:lan}, we introduce the local asymptotic normality and
associated frequentist's optimality theories.
Some empirical processes techniques are given in Section \ref{sec:empirical},
and the last section provides asymptotic theories on nonparametric Bayesian statistics.
The main results are given in Chapter \ref{chap:main}.
The first section proves a general semiparametric BvM theorem which requires two conditions:
the integral local asymptotic normality and convergence of the marginal posterior at parametric rate.
These two conditions are studied in more depth in following subsections.
In these two subsections, it is required that the expectation of the log likelihood ratio 
allows a certain quadratic expansion, and Section \ref{sec:quadratic} proves this condition
using the property of symmetric densities.
The last section of this chapter provides three examples mentioned above: the location,
linear regression and random intercept models.
Some numerical studies, which show the superiority of Bayes estimators in random effects models,
are provided in Chapter \ref{chap:numerical}.
A useful Gibbs sampler algorithm is given in the first section of this chapter.
A real dataset is also analyzed in Section \ref{sec:realdata}.
There are concluding remarks and future works in Chapter \ref{chap:conclusion},
and miscellanies that are required for main theorems and examples are given in Appendix.
Section \ref{sec:consistency} is devoted to prove posterior consistency when the model is slightly misspecified
and observations are independent but not identically distributed.
Some technical lemmas for semiparametric mixture models, such as bounded entropy and prior positivity conditions,
are given in Section \ref{sec:smixture}.
The last Section presents properties of symmetrized Dirichlet processes and
Gibbs sampler algorithms using symmetrized Dirichlet process mixtures.

Before going further, we introduce notations used in this thesis.
For a real-valued function $g$ defined on a subset of $\bbR$,
the first, second and third derivatives are denoted by $g^\prime$, $g^{\prime\prime}$
and $g^{\prime\prime\prime}$, respectively.
If the domain of $g$ is a subset of $\bbR^d$ for $d > 1$, then
$\nabla g$ and $\nabla^2 g$ denotes the $d\times 1$ gradient vector and $d\times d$ Hessian matrix.
Also, $\nabla_j g$ and $\nabla_{jk} g$ denote the first and second order partial derivatives of $g$
with respect to the corresponding indices.
The Euclidean norm is denoted by $|\cdot|$.
For a matrix $A$, $\|A\|$ represents the operator norm, defined as $\sup_{|x|\leq 1} |Ax|$, of $A$, and
if $A$ is a square matrix, $\rho_{\rm min} (A)$ and $\rho_{\rm max}(A)$ denotes 
the minimum and maximum eigenvalues of $A$.
The capital letters $P_\eta, P_{\theta,\eta},$ etc are the corresponding probability measures
of densities denoted by lower letters $p_\eta, p_{\theta,\eta}$, etc and vise versa.
The corresponding log densities are written by the letter $\ell_\eta, \ell_{\theta,\eta}$, etc.
The Hellinger and total variation metrics between two probability measures $P_1$ and $P_2$
are defined by 
$$
	h(P_1,P_2) =\left( \int \big(\sqrt{p_1} - \sqrt{p_2}\big)^2 d\mu \right)^{1/2}
$$
and $d_V(P_1, P_2) = 2\sup_A |P_1(A) - P_2(A)| = \int|p_1 - p_2| d\mu$, respectively,
where $\mu$ is a measure dominating both $P_1$ and $P_2$.
Let $K_{P_1}(P_2) = -\int\log (dP_2/dP_1) dP_1$ be the Kullback-Leibler divergence.
The metrics and Kullback-Leibler divergence are sometimes denoted like,
for example, $h(p_1,p_2)$ using the corresponding densities.
The expectation of a random variable $X$ under a probability measure $P$ is denoted by $P X$.
The notation $P_0$ always represents the true probability which generates the observation.
Finally, $N_{\theta,\Sigma}$ is the probability measure of the multivariate normal distribution 
with mean $\theta$ and variance $\Sigma$,
and $\phi_\sigma$ denotes the univariate normal density with mean 0 and variance $\sigma^2$.

\chapter{Literature reviews}
\label{chap:review}

This chapter briefly reviews three topics in asymptotic statistics.
Each topic is closely related to our main results and
essential techniques for the proofs in this thesis.
Section \ref{sec:lan} introduces some results derived from the local asymptotic normality
which is a key property of classical asymptotic theory.
In Section \ref{sec:empirical}, modern empirical processes theories are provided.
The last section is devoted to introduce Bayesian asymptotics including 
the parametric BvM theorem and theories for infinite dimensional models.

\section{Local asymptotic normality}
\label{sec:lan}

A sequence of statistical models is locally asymptotically normal if, roughly speaking,
the likelihood ratio behaves like that for a normal location parameter.
This implies that the likelihood ratio admits a certain quadratic expansion.
An important example is a smooth parametric model, so-called the regular parametric model.
If a model is locally asymptotically normal, estimating the model parameter can be
understood as a problem of estimating the normal mean in an asymptotic sense.
As a result, it satisfies some asymptotic optimality criteria such as
the convolution theorem and locally asymptotic minimax theorem.
There are much literature about the local asymptotic normality and related asymptotic theories.
Here we refer to two well-known books: \citet{bickel1998efficient} and \citet{van1998asymptotic}
which contain a lot of references and examples.

In this section, we only consider \iid~models because it contains all essentials
about the local asymptotic normality.
For \iid\; models, a sequence of statistical models can be represented as 
a collection of probability measures for a single observation.
An extension to non-\iid~models, including both finite and infinite dimensional models,
is well-established in \citet{mcneney2000application}.
Consider a statistical model $\{P_\theta: \theta \in \Theta\}$ parametrized by
finite dimensional parameter $\theta\in\Theta$ and assume that $\Theta$ is an open subset of $\bbR^p$.
The model is called \emph{locally asymptotic normal}, or simply \emph{LAN}, at $\theta$
if there exists a function $\score_{\theta}$
such that $\int|\score_\theta|^2 dP_\theta < \infty$ and
for every converging sequence $h_n \rightarrow h$ in $\bbR^p$, 
\be \label{eq:parametric_LAN}
  \sum_{i=1}^n \log\frac{p_{\theta+h_n/\sqrt{n}}}{p_{\theta}}(X_i) =
  \frac{h^T}{\sqrt{n}} \sum_{i=1}^n \score_{\theta}(X_i) - \frac{1}{2} h^T\, 
  I_{\theta}\,h + o_{P_{\theta}}(1)
\ee
as $n \rightarrow \infty$, where $I_\theta = P_\theta [\score_\theta \score_\theta^T]$.
The function $\score_\theta$ and matrix $I_\theta$ are called by the \emph{score function} and
\emph{Fisher information matrix}, respectively.
Le Cam formulated the first version of LAN property as early as 1953 in his thesis.
This original version can be found, for example, in \citet{le1990locally}.
Note that the likelihood ratio of the normal location model $\{N(h,\Sigma): h \in \bbR^p\}$
with single observation $X$ is given by
$$
	\frac{dN_{h,\Sigma}}{dN_{0,\Sigma}}(X) = h^T \Sigma^{-1} X - \frac{1}{2} h^T \Sigma^{-1}h
$$
where $dN_{h, \Sigma}$ is the multivariate normal density with mean $h$ and variance $\Sigma$.
Since the term $n^{-1/2} \sum_i \score_\theta(X_i)$ in \eqref{eq:parametric_LAN}
converges in distribution to the normal distribution $N(0, I_\theta)$,
it is clear that the local log likelihood ratio \eqref{eq:parametric_LAN} converges in distribution
to the log likelihood ratio of the normal location model in which $\Sigma^{-1} = I_\theta$.
The name LAN originated from this fact.

One important result is that every smooth parametric model is LAN.
Here the smoothness of a model can be expressed in quadratic mean differentiability.
A model $\{P_\theta: \theta \in \Theta\}$ is called \emph{differentiable in quadratic mean} at $\theta$
if it is dominated by a $\sigma$-finite measure $\mu$ and 
there exists an $L_2(P_\theta)$-function $\score_\theta$ such that
$$
	\int \left[ \sqrt{p_{\theta+h}} - \sqrt{p_\theta} - \frac{1}{2} h^T
	\score_\theta \sqrt{p_\theta} \right]^2 d\mu = o(|h|^2)
$$
as $h \rightarrow 0$.
This is actually the Hadamard (equivalently \Frechet) differentiability of the root density
$\theta \mapsto \sqrt{p_\theta}: \bbR^p \rightarrow L_2(\mu)$ which
can be established by pointwise differentiability plus a convergence theorem for integrals.
A proof of the following theorem can be found in Theorem 7.2 of \citet{van1998asymptotic}.

\bigskip

\begin{theorem}
Assume that $\Theta$ is open in $\bbR^p$ and $\{P_\theta: \theta\in\Theta\}$ is differentiable 
in quadratic mean at $\theta$.
Then, $P_\theta \score_\theta = 0$, $I_\theta$ exists,
and the LAN assertion \eqref{eq:parametric_LAN} holds.
\end{theorem}

\bigskip

More general statement of LAN can be found in \citet{strasser1985mathematical}.
With the help of the LAN property, Fisher's early concept of efficiency can be sharpened and elaborated upon.
We state three optimality theorems by Le Cam and \Hajek, which can be derived from the LAN property.
Besides the original reference, we refer to Chapter 8 of \citet{van1998asymptotic}
as a nice text.
An estimator sequence $\hat\theta_n$ is called \emph{regular} at $\theta$ if, for every $h$,
\be \label{eq:lim_dist}
	\cL\left(\sqrt{n}\Big(\hat\theta_n - \theta - \frac{h}{\sqrt{n}}\Big) \Big| P_{\theta+h/\sqrt{n}} \right)
	\stackrel{d}{\rightarrow} L_\theta
\ee
for some probability distribution $L_\theta$.
Here $\cL(T|P)$ denotes the distribution of $T=T(X)$ when $X$ follows the probability measure $P$
and $\stackrel{d}{\rightarrow}$ represents convergence in distribution.
Note that the limit distribution $L_\theta$ does not depend on $h$
and this is the key assumption for regularity of an estimator.
Let $*$ be the convolution operator.
The most important theorem about asymptotic optimality is definitely \Hajek-Le Cam convolution theorem
(\citet{hajek1970characterization, le1986asymptotic}) stated as follows.

\bigskip

\begin{theorem}[\bf Convolution] \label{thm:convolution}
Assume that $\Theta$ is open in $\bbR^p$ and $\{P_\theta: \theta\in\Theta\}$ is LAN at $\theta$
with the nonsingular Fisher information matrix $I_\theta$.
Then for any regular estimator sequence $\hat\theta_n$ for $\theta$,
there exist probability distributions $M_\theta$ such that
$$L_\theta = N_{0, I_\theta^{-1}} * M_\theta,$$
where $L_\theta$ is the limit distribution in \eqref{eq:lim_dist}.
\end{theorem}

\bigskip

Theorem \ref{thm:convolution} says that for a class of all regular estimators,
the normal distribution $N(0, I_\theta^{-1})$ is the best possible limit distribution.
However, some estimator sequences of interest, such as shrinkage estimators, are not regular.
A typical example is the Hodges superefficient estimator 
\bean
	\hat\theta_n = \left\{ \begin{array}{cc}
	\bar{X}_n & \textrm{if} ~ |\bar{X}_n | \geq n^{-1/4}\\
	\alpha \bar{X}_n & \textrm{o.w.}
	\end{array} \right.
\eean
for the normal location parameter.
Here $\alpha$ is an arbitrary positive constant which is strictly smaller than 1.
In this case, $\hat\theta_n$ is $n^{-1/2}$-consistent, that is $\sqrt{n}(\hat\theta_n - \theta) = O_{P_\theta}(1)$,
and asymptotically normal, but superefficient at 0 (variance is smaller than that of MLE).
Interestingly, the set of superefficiency is of Lebesgue measure zero and
this can be proved in general situations (\citet{le1953some}).

\bigskip

\begin{theorem} \label{thm:asconvolution}
Assume that $\Theta$ is open in $\bbR^p$ and $\{P_\theta: \theta\in\Theta\}$ is LAN at $\theta$
with the nonsingular Fisher information matrix $I_\theta$.
Let $\hat\theta_n$ be an estimator sequence such that $\sqrt{n}(\hat\theta_n-\theta)$ converges
to a limit distribution $L_\theta$ under every $\theta$.
Then, there exist probability distributions $M_\theta$ such that 
$$
	L_\theta = N_{0, I_\theta^{-1}} * M_\theta
$$
for Lebesgue almost every $\theta$.
\end{theorem}

\bigskip

Though the set of superefficiency is a null set, the above theorem may not be fully satisfactory
because there is no information about parameters which may be important as in the \Hajek's example.
Furthermore, an estimator sequence is required to be $n^{-1/2}$-consistent in Theorem \ref{thm:asconvolution}.
The following theorem, which can be found in Theorem 8.11 of \citet{van1998asymptotic}, 
is a refined version of the so-called local asymptotic minimax theorem
(\citet{hajek1972local, le1972limits}).
A function $l: \bbR^p \rightarrow [0,\infty)$ is called a \emph{bowl-shaped loss} if the sublevel sets
$\{x: l(x) \leq c\}$ are convex and symmetric about the origin.
It is called \emph{subconvex} if, moreover, these sets are closed.

\bigskip

\begin{theorem}[\bf Local asymptotic minimax]
Assume that $\Theta$ is open in $\bbR^p$ and $\{P_\theta: \theta\in\Theta\}$ is LAN at $\theta$
with the nonsingular Fisher information matrix $I_\theta$.
Then, for any estimator sequence $\hat\theta_n$ and bowl-shaped loss function $\ell$,
$$
	\sup_I \liminf_{n\rightarrow\infty}\sup_{h\in I} P_{\theta+h/\sqrt{n}}
	l\left( \sqrt{n}\Big(\hat\theta_n - \theta - \frac{h}{\sqrt{n}}\Big) \right)
	\geq \int l \; dN_{0, I_\theta^{-1}},
$$
where the first supremum is taken over all finite subsets $I$ of $\bbR^p$.
\end{theorem}

\bigskip

According to the three theorems above we conclude that the normal distribution $N(0, I_\theta^{-1})$
is the best possible limit distribution.
An estimator sequence $\hat\theta_n$ is called \emph{efficient} or \emph{best regular}
if it is regular and 
$$
	\cL(\sqrt{n}(\hat\theta_n - \theta)|P_\theta) \stackrel{d}{\rightarrow} N_{0, I_\theta^{-1}}
$$
as $n\rightarrow \infty$.
A well-known (see, for example, \citet{van1998asymptotic}) fact is that 
every efficient estimator is asymptotically linear estimator as stated in the following theorem.

\bigskip

\begin{theorem} \label{thm:efficient_parameteric}
An estimator sequence $\hat\theta_n$ is efficient if and only if
$$
	\sqrt{n}(\hat\theta_n - \theta) = \frac{1}{\sqrt{n}} \sum_{i=1}^n I_\theta^{-1}
	\score_\theta(X_i) + o_{P_\theta}(1).
$$
\end{theorem}

\bigskip

So far we have studied asymptotic optimality of an estimator sequence in a smooth parametric model.
The two theorems, the convolution theorem and local asymptotic minimax theorem,
have natural extensions in infinite dimensional models.
Typically an infinite dimensional parameter is not estimable at $n^{-1/2}$ rate (\citet{van1991differentiable}).
It is possible, however, to estimate some finite dimensional parameters at this rate 
even in an infinite dimensional model.
The central limit theorem, by which mean parameters are estimable at parametric rate,
is a representative example.
Under regularity conditions, moreover, some estimators can be shown to be asymptotically optimal 
in the sense of the convolution theorem and local asymptotic minimax theorem as in parametric models.

We first define the tangent set and tangent space.
For a given statistical model $\scrP$ containing $P_0$, consider a
one-dimensional submodel $t \mapsto P_t$ passing through $P_0$ at $t=0$ and
differentiable in quadratic mean.
By the differentiability we get the score function $g$ at $P_0$ from this submodel.
Letting $t \mapsto P_t$ range over the collection of all such submodels,
we obtain the collection of score functions,
which is called the \emph{tangent set} of the model $\scrP$ at $P_0$.
The closed linear span of the tangent set in $L_2(P_0)$, denoted by $\dot\scrP$, 
is called the \emph{tangent space} of $\scrP$ at $P_0$.

Since our main interest in Chapter \ref{chap:main} is to estimate a finite dimensional parameter in a semiparametric model,
we only consider the information bound for a semiparametric model
$\scrP = \{P_{\theta,\eta}: \theta\in\Theta,\eta\in\cH\}$,
$\theta$ is the finite dimensional parameter of interest and
$\eta$ is the infinite dimensional nuisance parameter.
For more general theory, readers are referred to two books: \citet{van1998asymptotic, bickel1998efficient}.
Fix $(\theta_0,\eta_0) \in \Theta \times \cH$, and define two submodels
$\scrP_1 = \{P_{\theta,\eta_0}: \theta \in\Theta\}$ and $\scrP_2 = \{P_{\theta_0,\eta}: \eta \in\cH\}$.
Assume that $\scrP_1$ is differentiable in quadratic mean and
let $\score_{\theta_0,\eta_0}$ be the score function at $\theta_0$.
Then it is easy to show that $\dot\scrP_1$ is equal to the set of all 
$h^T\score_{\theta_0,\eta_0}$, where $h$ ranges over $\bbR^p$.
The function defined by
$$
	\effscore_{\theta_0,\eta_0} = \score_{\theta_0,\eta_0} - \Pi_{\theta_0,\eta_0} \score_{\theta_0,\eta_0}
$$
is called the \emph{efficient score function} and the matrix 
$\effI_{\theta_0,\eta_0} = P_{\theta_0,\eta_0} \effscore_{\theta_0,\eta_0} \effscore_{\theta_0,\eta_0}^T$
is the \emph{efficient information matrix}, where
$\Pi_{\theta_0,\eta_0}$ is the orthogonal projection onto $\dot\scrP_2$ in $L_2(P_{\theta_0,\eta_0})$.
For defining the information for estimating $\theta$,
if $\dot\scrP = \dot\scrP_1 + \dot\scrP_2$, then 
it is enough to consider one-dimensional smooth (differentiable in quadratic mean) submodels of type 
\be \label{eq:submodel}
	t \mapsto P_{\theta_0 + th, \eta_t}
\ee
for $h \in \bbR^p$.
An estimator sequence $\hat\theta_n$ is \emph{regular} for estimating $\theta$
if it is regular in every such submodel, that is
$$
	\cL\left(\sqrt{n}\Big(\hat\theta_n - \theta_0 - \frac{t}{\sqrt{n}}h\Big)
	\Big| P_{\theta_0+th/\sqrt{n}, \eta_{t/\sqrt{n}}} \right)
	\stackrel{d}{\rightarrow} L_{\theta_0}
$$
for some $L_{\theta_0}$ which does not depend on $h$.
The following two theorems are extensions of the convolution theorem and local asymptotic minimax theorem,
respectively, to semiparametric models.

\bigskip

\begin{theorem}[\bf Convolution]
Assume that $\dot\scrP = \dot\scrP_1 + \dot\scrP_2$, $\dot\scrP$ is convex and
$\effI_{\theta_0,\eta_0}$ is nonsingular.
Then, every limit distribution $L_{\theta_0}$ of a regular sequence of estimators  can be written
$L_{\theta_0} = N_{0, \effI_{\theta_0,\eta_0}^{-1}}*M_{\theta_0}$ for some probability distribution $M_{\theta_0}$.
\end{theorem}

\bigskip

\begin{theorem}[\bf Local asymptotic minimax]
Assume that $\dot\scrP = \dot\scrP_1 + \dot\scrP_2$, $\dot\scrP$ is convex and
$\effI_{\theta_0,\eta_0}$ is nonsingular.
Then for any estimator sequence $\hat\theta_n$ and subconvex loss function $l$,
$$
	\sup_I \liminf_{n\rightarrow\infty}\sup_{i\in I} P_{1/\sqrt{n},i}
	l\left( \sqrt{n}\Big(\hat\theta_n - \theta - \frac{h}{\sqrt{n}}\Big) \right)
	\geq \int l \; dN_{0, I_{\theta_0,\eta_0}^{-1}},
$$
where the first supremum is taken over all finite index sets $I$ of one-dimensional
smooth submodels, denoted by $P_{t,i}$, of type \eqref{eq:submodel}.
\end{theorem}

\bigskip

As in parametric models, the normal distribution $N(0, \effI_{\theta_0,\eta_0}^{-1})$ 
can be considered as the best possible limit distribution.
A regular estimator sequence $\hat\theta_n$ is called \emph{efficient}
or \emph{best regular} if it is regular and its limit distribution is $N(0, \effI^{-1}_{\theta_0,\eta_0})$.
An efficient estimator is asymptotically linear as in Theorem \ref{thm:efficient_parameteric},
replacing the score function and information matrix by
the efficient score function and efficient information matrix.

\bigskip

\begin{theorem}
An estimator sequence $\hat\theta_n$ is efficient if and only if
$$
	\sqrt{n}(\hat\theta_n - \theta_0) = \frac{1}{\sqrt{n}} \sum_{i=1}^n \effI_{\theta_0,\eta_0}^{-1}
	\effscore_{\theta_0,\eta_0}(X_i) + o_{P_{\theta_0,\eta_0}}(1).
$$
\end{theorem}

\bigskip

Roughly speaking, the information bound $\effI_{\theta_0, \eta_0}$ of a semiparametric model
is equal to the infimum of information bounds of all smooth parametric submodels.
If there is a smooth parametric submodel whose information bound achieves this infimum,
it is the hardest submodel.
Formally in a smooth semiparametric model, if there exists a submodel 
$\{P_{\theta, \eta^*(\theta)}: \theta\in\Theta\}$
which has $\effscore_{\theta_0,\eta_0}$ as the score function at $\theta_0$,
it is called a \emph{least favorable submodel} at $(\theta_0,\eta_0)$.
There may be more than two least favorable submodels, or it may not exist.
Typically, if a maximizer of the map
$\eta \mapsto P_{\theta_0,\eta_0} \log \big(p_{\theta,\eta}/p_{\theta_0,\eta_0}\big)$
is smooth in $\theta$, it constitutes a least favorable submodel
(\citet{severini1992profile}; \citet{murphy2000profile}).

We finish this section with the notion of adaptiveness.
A smooth semiparametric model $\scrP$ is called (locally) \emph{adaptive} (at $P_{\theta_0,\eta_0}$)
if $\dot\scrP_1 \bot \dot\scrP_2$ in $L_2(P_{\theta_0,\eta_0})$.
By definition the efficient score function and information matrix is 
equal to the ordinary score function and information matrix in adaptive models.
Therefore the information bound for the semiparametric model $\scrP$ and the parametric model $\scrP_1$,
in which the true nuisance parameter $\eta_0$ is known, are the same.

\bigskip

\section{Empirical processes}
\label{sec:empirical}

In this section we review modern empirical process theories that play important roles
for the proofs given in Chapter \ref{chap:main}.
We assume that readers are familiar to weak convergence of probability measures in metric spaces.
Also, we do not state any measurability conditions,
because the formulation of these would require too many digressions.
For all details about this section and further reading including historical stories, examples and so on,
we refer to the monograph \citet{van1996weak}.

Consider a sample of random elements $X_1, \ldots, X_n$ in a measurable space $(\cX, \cA)$,
where $\cX$ is endowed with a semimetric\footnote{$d(x,y)$ can be equal to 0 when $x \neq y$.} $d$.
Let $\bbP_n = n^{-1} \sum_{i=1}^n \delta_{X_i}$ be the \emph{empirical measure} and 
$\bbG_n = \sqrt{n}(\bbP_n - P)$ be \emph{empirical process}, where $\delta_x$ denotes the
Dirac measure at point $x$.
Consider a collection $\cF$ of measurable functions $f:\cX \rightarrow \bbR$.
With the notation $\|P\|_\cF = \sup_{f\in\cF} |Pf|$, if 
$$
	\|\bbP_n - P\|_\cF \rightarrow 0
$$
in $P$-probability, $\cF$ is called a $P-$\emph{Glivenko-Cantelli class}, or simply
Glivenko-Cantelli class.
Under the condition $\sup_{f\in\cF} |f(x) - Pf| < \infty$ for every $x\in\cX$, the empirical process
$f \mapsto \bbG_n f$ can be viewed as an $\ell^\infty(\cF)$-valued random element.
If this map converges weakly to a tight Borel measurable element in $\ell^{\infty}(\cF)$,
it is called a \emph{Donsker class}, or $P$-Donsker to be more complete.

The Donsker property is very important and closely related to the notion of tightness.
Before going further, we introduce some definitions and theorems about
stochastic processes in spaces of bounded functions.
A sequence of $\cX$-valued stochastic processes $X_n$ is \emph{asymptotically tight} if for every $\epsilon > 0$
there exists a compact set $K$ such that
$$\liminf_{n\rightarrow\infty} P(X_n \in K^\delta) \geq 1-\epsilon$$
for every $\delta > 0$.
Here $K^\delta$ is defined by the set $\{x \in \cX: d(x, K) < \delta\}$.
This is slightly weaker than uniform tightness but enough to assure the weak convergence.
For an index set $T$, weak convergence in $\ell^\infty(T)$ is characterized as
asymptotic tightness plus convergence of marginals as stated in the following theorem.

\bigskip

\begin{theorem} \label{thm:weak_tight}
A sequence of $\ell^\infty(T)$-valued stochastic processes $X_n$ converges weakly to a tight limit 
if and only if $X_n$ is asymptotically tight and the marginals
$(X_n(t_1), \ldots, X_n(t_k))$ converge weakly to a limit for every finite subset $t_1, \ldots, t_k$ of $T$.
\end{theorem}

\bigskip

Asymptotic tightness is a quite complicate concept and it is closely related to
equicontinuity of sample paths of stochastic processes.
For a semimetric space $(T, \rho)$, 
a sequence of $\ell^\infty(T)$-valued stochastic process $Z_n$ is said to be \emph{asymptotically uniformly 
$\rho$-equicontinuous} in probability if for every $\epsilon_1, \epsilon_2 > 0$ there exists a
$\delta > 0$ such that
$$\limsup_n P \bigg( \sup_{\rho(s,t) < \delta} |Z_n(s)-Z_n(t)| > \epsilon_1 \bigg) < \epsilon_2.$$
The following theorem represents the relationship between asymptotic tightness and 
asymptotic unifomrly equicontinuity of sample paths.

\bigskip

\begin{theorem} \label{thm:tight_equiconti}
A sequence of stochastic processes $X_n$ indexed by $T$ is asymptotically tight if and only if
$X_n(t)$ is asymptotically tight in $\bbR$ for every $t\in T$ and there exists a semimetric
$\rho$ on $T$ such that $(T,\rho)$ is totally bounded and $X_n$ is asymptotically uniformly
$\rho$-equicontinuous in probability.
If, moreover, $X_n$ converges weakly to $X$, then almost all paths $t \mapsto X(t)$
are uniformly $\rho$-continuous and the semimetric $\rho$ can be taken equal to any semimetric
for which this is true and $(T,\rho)$ is totally bounded.
\end{theorem}

\bigskip

A stochastic process $X$ is called \emph{Gaussian} if each of its finite-dimensional marginals
$(X(t_1), \ldots, X(t_k))$ has a multivariate normal distribution on Euclidean space.
For a given stochastic process $X$, define a semimetric $\rho_2$ on $T$ by
$$
	\rho_2(s,t) = \Big( P|X(s) - X(t)|^2 \Big)^{1/2}.
$$
When the limit process $X$ in Theorem \ref{thm:tight_equiconti} is Gaussian,
$\rho_2$ can always be used to establish asymptotic equicontinuity of a sequence $X_n$.

\bigskip

\begin{theorem}
A Gaussian process $X$ in $\ell^\infty(T)$ is tight if and only if $(T, \rho_2)$
is totally bounded and almost all paths $t \mapsto X(t)$ are uniformly $\rho_2$-continuous.
\end{theorem}

\bigskip

Now we return to empirical processes $\bbG_n$ on $\cF$.
By the central limit theorem, a marginal distribution $(\bbG_n f_1, \ldots, \bbG_n f_k)$
converges weakly to a normal distribution.
Therefore if the stochastic process $f \mapsto \bbG_nf$ is asymptotically tight,
then $\cF$ is a Donsker class by Theorem \ref{thm:weak_tight}.
Since asymptotic tightness is conceptually equivalent to the uniform equicontinuity of sample paths
by Theorem \ref{thm:tight_equiconti},
we can expect from Arzel\`{a}-Ascoli theorem that the Donsker property can be determined
by the covering number.
The \emph{covering number} $N(\epsilon, \cF, \rho)$ of $\cF$ with respect to a semimetric $\rho$
is the minimal number of balls
$\{g:\rho(f,g) < \epsilon\}$ of radius $\epsilon$ needed to cover the set $\cF$.
For given two functions $l$ and $u$, the bracket $[l,u]$ is the set of all functions $f$
with $l \leq f \leq u$. An $\epsilon$-\emph{bracket} is a bracket $[l,u]$ with $\rho(u,l) < \epsilon$.
The \emph{bracketing number} $N_{[]}(\epsilon,\cF,\rho)$ is the minimum number of $\epsilon$-brackets
needed to cover $\cF$.
Then it is easy to show that 
$$N(\epsilon, \cF, \rho) \leq N_{[]}(2\epsilon,\cF,\rho)$$
for every $\epsilon > 0$.
Define
$$
	J_{[]}(\delta, \cF, L_2(P)) = \int_0^\delta \sqrt{\log N_{[]}(\epsilon, \cF, L_2(P))} d\epsilon
$$
for $\delta > 0$.
A collection $\cF$ of functions is a Donsker class if the covering number or bracketing number
is suitably bounded.
We only introduce conditions on bracketing numbers and refer to Section 2.6 of \citet{van1996weak}
for conditions on covering numbers.

\bigskip

\begin{theorem} \label{thm:donsker_bracketing}
$\cF$ is $P$-Donsker if $J_{[]}(1, \cF, L_2(P)) < \infty$.
\end{theorem}

\bigskip

The condition of Theorem \ref{thm:donsker_bracketing} is very simple and 
is satisfied for many interesting examples
For classes of smooth functions on a Euclidean space, we can find an upper bound for bracketing numbers.
To define such classes let, for a given function $f: I \subset \bbR^d \rightarrow \bbR$ and $\alpha >0$,
$$
  \|f\|_\alpha = \max_{k_\cdot \leq \lfloor\alpha\rfloor} \sup_x |D^k f(x)| \vee
  \max_{k_\cdot = \lfloor\alpha\rfloor} \sup_{x,y} \frac{|D^kf(x) - D^kf(y)|}{|x-y|^{\alpha - \lfloor\alpha\rfloor}}
$$
where the suprema are taken over all $x,y$ in the interior of $I$ with $x\neq y$,
the value $\lfloor\alpha\rfloor$ is the greatest integer strictly smaller than $\alpha$,
and for each vector $k$ of $d$ integers $D^k$ is the differential operator
$$
D^k = \frac{\partial^{k_\cdot}}{\partial x_1^{k_1} \cdots \partial x_k^{k_k}},
~~~~~ k_\cdot = \sum_i k_i.
$$
These are well-known $\alpha$-\emph{\Holder} norms.
Let $C_M^\alpha(I)$ be the set of all continuous functions $f: I \rightarrow \bbR$ with $\|f\|_\alpha \leq M$.

\bigskip

\begin{theorem} \label{thm:bracketing_bound}
Let $\bbR^d = \cup_{j=1}^\infty I_j$ be a partition into cubes of uniformly bounded size,
and $\cF$ be a class of functions $f:\bbR^d \rightarrow \bbR$ such that the restrictions of
$f$ onto $I_j$ belong to $C^\alpha_{M_j}$ for every $j$ and some fixed $\alpha > d/2$.
Then, there exists a constant $K$ depending only on $\alpha, V, r, d$
and the uniform bound on the diameter of the sets $I_j$ such that
\be \label{eq:bracket_bound}
  \log N_{[]}(\epsilon, \cF, L_r(P)) \leq \frac{K}{\epsilon^V} 
  \left(\sum_{j=1}^\infty M_j^{Vr/(V+r)} P(I_j)^{V/(V+r)}\right)^{(V+r)/r}
\ee
for $V \geq d/\alpha$
\end{theorem}

\bigskip

Theorem \ref{thm:donsker_bracketing} concern the empirical process for different $n$,
but each time with the same indexing class $\cF$.
This is enough for many applications, but sometimes it may be necessary to allow the class $\cF$
to change with $n$.
The following theorem is a modification of Theorem \ref{thm:donsker_bracketing} for this purpose.

\bigskip

\begin{theorem} \label{thm:donsker_varying_class}
Let $\cF_n=\{f_{n,t}: t\in T\}$ be a class of measurable functions 
indexed by a totally bounded semimetric space $(T, \rho)$ satisfying
$$
	\sup_{\rho(s,t) < \delta_n} P(f_{n,s} - f_{n,t})^2 \rightarrow 0, ~~\textrm{for all}~ \delta_n \downarrow 0
$$
and assume that there exists a function $F_n$ such that 
$\sup_{t\in T} |f_{n,t}| < F_n$, $PF_n^2 = O(1)$ and 
$PF_n^2 I(F_n > \epsilon \sqrt{n}) \rightarrow 0$ for all $\epsilon > 0$.
If $J_{[]}(\delta_n, \cF_n, L_2(P)) \rightarrow 0$ for every $\delta_n \downarrow 0$
and $Pf_{n,s}f_{n,t} - Pf_{n,s}Pf_{n,t}$ converges pointwise on $T\times T$,
then the sequence $\{\bbG_n f_{n,t}: t \in T\}$ converges weakly to a tight Gaussian process.
\end{theorem}

\bigskip

Theorems \ref{thm:donsker_bracketing} and \ref{thm:donsker_varying_class} only consider 
empirical processes of \iid\;observations.
We finish this section with an extension of Donsker theorem to the case of
independent but not identically distributed processes.
The following theorem is an extension of Jain-Marcus's central limit theorem (\citet{jain1975central}),
and the proof can be found in Theorems 2.11.9 and 2.11.11 of \citet{van1996weak}.

\bigskip

\begin{theorem} \label{thm:jain_marcus}
For each $n$, let $Z_{n1}, \ldots, Z_{nm_n}$ be independent stochastic processes
indexed by an arbitrary index set $\cF$.
Suppose that there exist independent random variables $M_{n1}, \ldots, M_{nm_n}$,
and a semimetric $\rho$ such that
$$
	|Z_{ni}(f) - Z_{ni}(g)| \leq M_{ni} \rho(f,g)
$$
for every $f,g\in\cF$, 
$$
	\int_0^\infty \sqrt{\log N(\epsilon, \cF, \rho)} d\epsilon < \infty,
$$
and
$$
	\sum_{i=1}^{m_n} P M_{ni}^2 = O(1).
$$
Furthermore assume that
$$
	\sum_{i=1}^{m_n} P \|Z_{ni}\|^2_\cF 1_{\{\|Z_{ni}\|_\cF > \delta\}} \rightarrow 0
$$
for every $\delta > 0$.
Then the sequence $\sum_{i=1}^{m_n} Z_{ni} - PZ_{ni}$ is asymptotically uniformly $\rho$-equicontinuous
in $P$-probability.
Moreover, it converges to a tight Gaussian process provided the sequence of covariance functions converges
pointwise on $\cF\times\cF$.
\end{theorem}

\bigskip

\section{Bayesian asymptotics}

For the last few decades, there were remarkable activities in the development of nonparametric Bayesian statistics.
This section reviews some frequentist properties of Bayesian procedures in infinite dimensional models.
There are books for nonparametric Bayesian statistics like \citet{ghosh2003bayesian} and \citet{hjort2010bayesian},
but they are not fully satisfactory because a lot of important theories and examples 
are developed quite recently.
Here we focus on asymptotic behaviors of posterior distributions when
\iid\;observations are given.

Let $X_1, \ldots, X_n$ be a random sample in a metric space $\cX$ with the Borel $\sigma$-algebra $\cB(\cX)$.
Consider a statistical model $\scrP = \{P_\theta: \theta \in \Theta\}$,
where the parameter space $\Theta$ is equipped with a metric $d$.
Let $\Pi$ be a \emph{prior} on $\Theta$, that is, a probability measure on
the Borel $\sigma$-algebra $\cB(\Theta)$ of $\Theta$.
Any version of the conditional distribution of $\theta$ given $X_1, \ldots, X_n$
is called a \emph{posterior distribution} and denoted by $\Pi(\cdot|X_1, \ldots, X_n)$.
We assume that there exists a $\sigma$-finite measure $\mu$ on $\cB(\cX)$ dominating
all $P_\theta$.
In this case, using Bayes' rule, the posterior distribution is given by
$$
	\Pi(A | X_1, \ldots, X_n) = \frac{\int_A \prod_{i=1}^n p_\theta(X_i) d\Pi(\theta)}
	{\int_\Theta \prod_{i=1}^n p_\theta(X_i) d\Pi(\theta)}
$$
for all $A \in \cB(\Theta)$.

A prior and data yield the posterior and the subjectiveness of this strategy
does not need the idea of what happens if further data arise.
However, one may be interested in asymptotic behavior of the posterior distribution
which can be seen as a frequentist viewpoint.
Frequentist typically assumes that there exists the true distribution $P_0$
which generates the observations $X_1, \ldots, X_n$.
Throughout this section, we assume that $P_0=P_{\theta_0}$ for some $\theta_0\in\Theta$, and under this assumption
the posterior distribution is expected to concentrate around the true parameter $\theta_0$.

Before going to infinite-dimensional models, we begin with parametric models.
In a smooth parametric, the posterior distribution is asymptotically normal centered on
a best regular estimator with the variance the inverse of Fisher information matrix.
This is the so-called BvM theorem which was proved by many authors.
The following theorem is considerably more elegant than the results by early authors
and proofs can be found, for example, in \citet{le1986asymptotic}, \citet{le2000asymptotics}.

\bigskip

\begin{theorem}[\bf Bernstein-von Mises] \label{thm:parametric_bvm}
Assume that a parametric model $\{P_\theta: \theta\in\Theta\}$ is differentiable in quadratic mean
at $\theta_0$ with nonsingular Fisher information matrix $I_{\theta_0}$.
Furthermore suppose that for every $\epsilon > 0$ there exists a sequence of tests $\varphi_n$ such that
$$
	P_{\theta_0}^n \varphi_n \rightarrow 0, ~~~~~ 
	\sup_{|\theta-\theta_0| > \epsilon} P_\theta^n (1-\varphi_n) \rightarrow 0.
$$
If the prior has continuous and positive density in a neighborhood of $\theta_0$,
then the corresponding posterior distributions satisfy
$$
	\sup_A \left| \Pi(\sqrt{n}(\theta-\theta_0) \in A | X_1, \ldots, X_n)
	- \int_A dN_{\Delta_{n}, I_{\theta_0}^{-1}}\right| \rightarrow 0	
$$
in $P_{\theta_0}^n$-probability, where $\Delta_n$ is a best regular estimator and 
the supremum is taken over all Borel sets.
\end{theorem}

\bigskip

Since best regular estimators are asymptotically equivalent up to $o_{P_0}(n^{-1/2})$ terms,
the centering sequence $\Delta_n$ in the BvM theorem can be any best regular estimator sequence.
An important application of the BvM theorem is that
the posterior mean is an efficient estimator and Bayesian credible sets are asymptotically equivalent
to frequentists' confidence intervals.
This implies that statistical inferences based on the posterior distribution is equally optimal
to that based on the maximum likelihood estimators.

A sequence of tests $\varphi_n$ is called \emph{uniformly consistent} for testing $H_0: \theta = \theta_0$
versus $H_1: \theta \in U$ if
\bean
	P_{\theta_0}^n \varphi_n \rightarrow 0 \\
	\sup_{\theta\in U} P_\theta^n(1-\varphi_n) \rightarrow 0
\eean
as $n \rightarrow \infty$.
Le Cam's version of the BvM theorem requires the existence of uniformly consistent tests
for testing $H_0: \theta=\theta_0$ versus $H_1:|\theta-\theta_0| > \epsilon$ for every $\epsilon > 0$.
Such tests certainly exist if there exist estimators $\hat\theta_n$ that are uniformly consistent, that is,
$$\sup_\theta P_\theta (|\hat\theta_n - \theta| > \epsilon) \rightarrow 0$$
for every $\epsilon > 0$.

Theorem \ref{thm:parametric_bvm} is quite general so it can be applied for most smooth parametric models.
As frequentist theory, however, Theorem \ref{thm:parametric_bvm} does not 
generalize fully to nonparametric estimation problems.
Actually many nonparametric priors do not work well in the sense that
the posterior mass does not concentrate around the true parameter.
An important counterexample was found by \citet{diaconis1986inconsistent, diaconis1986consistency}
which proves that the posterior distribution may be inconsistent even if a very natural nonparametric prior is used.
\citet{doss1984bayesian, doss1985bayesian1, doss1985bayesian2} found similar phenomena for median estimation problem.
Before introducing this example, we define the posterior consistency rigorously and
state an important theorem about consistency proved by \citet{doob1949application}.
The sequence of posteriors $\Pi(\cdot|X_1, \ldots, X_n)$ is said to be \emph{consistent} at $\theta_0$
(with respect to a metric $d$) if for every $\epsilon > 0$
$$
	P_{\theta_0}^n \Pi\big(\{\theta\in\Theta: d(\theta,\theta_0) > \epsilon | X_1, \ldots, X_n\}\big)
	\rightarrow 0
$$
as $n \rightarrow \infty$.
The definition of consistency may be different in some texts in which consistency is defined
using almost-sure convergence, not convergence in probability.
More precisely, we call the posterior is \emph{almost-surely consistent} at $\theta_0$, if for every $\epsilon > 0$
$$
	\Pi\big(\{\theta\in\Theta: d(\theta,\theta_0) > \epsilon | X_1, \ldots, X_n\}\big)
	\rightarrow 0
$$
$P_{\theta_0}^\infty$-almost-surely.
Furthermore we say that a sequence $\epsilon_n \rightarrow 0$ is the \emph{convergence rate of the posterior distribution}
at $\theta_0$ (with respect to a metric $d$) if for any $M_n \rightarrow \infty$, we have that
$$
	\Pi\Big(\{\theta\in\Theta: d(\theta,\theta_0) \geq M_n \epsilon_n\} | X_1, \ldots, X_n\Big) \rightarrow 0
$$
in $P_{\theta_0}^n$-probability.
As the definition of posterior consistency, the convergence rate of the posterior also can be defined
using almost-sure convergence.
Now we state the theorem by \citet{doob1949application}.

\bigskip

\begin{theorem}
Suppose that $\Theta$ and $\cX$ are both complete and separable metric spaces, 
and the model is identifiable.
Then there exists $\Theta_0 \subset \Theta$, with $\Pi(\Theta_0)=1$ such that
$\Pi(\cdot|X_1, \ldots, X_n)$ is consistent at every $\theta \in \Theta_0$.
\end{theorem}

\bigskip

Doob's theorem looks very useful bet it does not tell about the posterior consistency at a specific $\theta_0$.
Although the set of inconsistency is a $\Pi$-null set, it may not be ignorable when $\Theta$ is
an infinite-dimensional parameter space.
As mentioned above the Diaconis-Freedman's counterexample was a surprising discovery
in Bayesian nonparametric society as the case of Hodges supperefficient estimator.
Before the discovery of this counterexample, it was believed that most prior works well
except some abnormal examples.
To explain the Diaconis-Freedman example, we need to mention the Dirichlet process (\citet{ferguson1973bayesian}) prior
which is often considered as a starting point of Bayesian nonparametrics.
Dirichlet processes are widely used in many fields of science and industry 
for the prior of unknown probability distributions.
The definition of Dirichlet processes and its symmetrized version is given in Section \ref{sec:sdp}.
In the statement of the following theorem, we slightly abuse notations for $\theta$
which is used for the location parameter, not the whole parameter, in a semiparametric location problem.

\bigskip

\begin{theorem} \label{thm:diaconis}
Consider an \iid\;observations $X_1, \ldots, X_n$ from well-specified model
$$X_i = \theta + \epsilon_i,$$
where $\epsilon_i$ follows an unknown distribution $P$.
For the prior, $\theta$ has the standard normal density, and $P$ is independently drawn
from the symmetrized Dirichlet process with mean the standard Cauchy distribution.
Then the posterior is inconsistent at $\theta=0$ and $P=P_0$ for some $P_0$
which has infinitely differentiable density $p_0$, which is compactly supported and symmetric about 0,
with a strict maximum at $0$.
\end{theorem}

\bigskip

An example of inconsistent $p_0$ in Theorem \ref{thm:diaconis} is illustrated in Figure \ref{fig:diaconis}.
With this $p_0$, the posterior mass for $\theta$ concentrate around two distinct points $\pm\gamma$
for some $\gamma>0$.
To prove the posterior consistency at a specific point $\theta_0$,
the condition by \citet{schwartz1965bayes} can be a very useful tool.
It requires that the prior mass of every Kullback-Leibler neighborhood of the true parameter is positive.
Furthermore a uniformly consistent sequence of tests are required.

\begin{figure}
\begin{center}
\includegraphics[width=80 mm, height=50 mm]{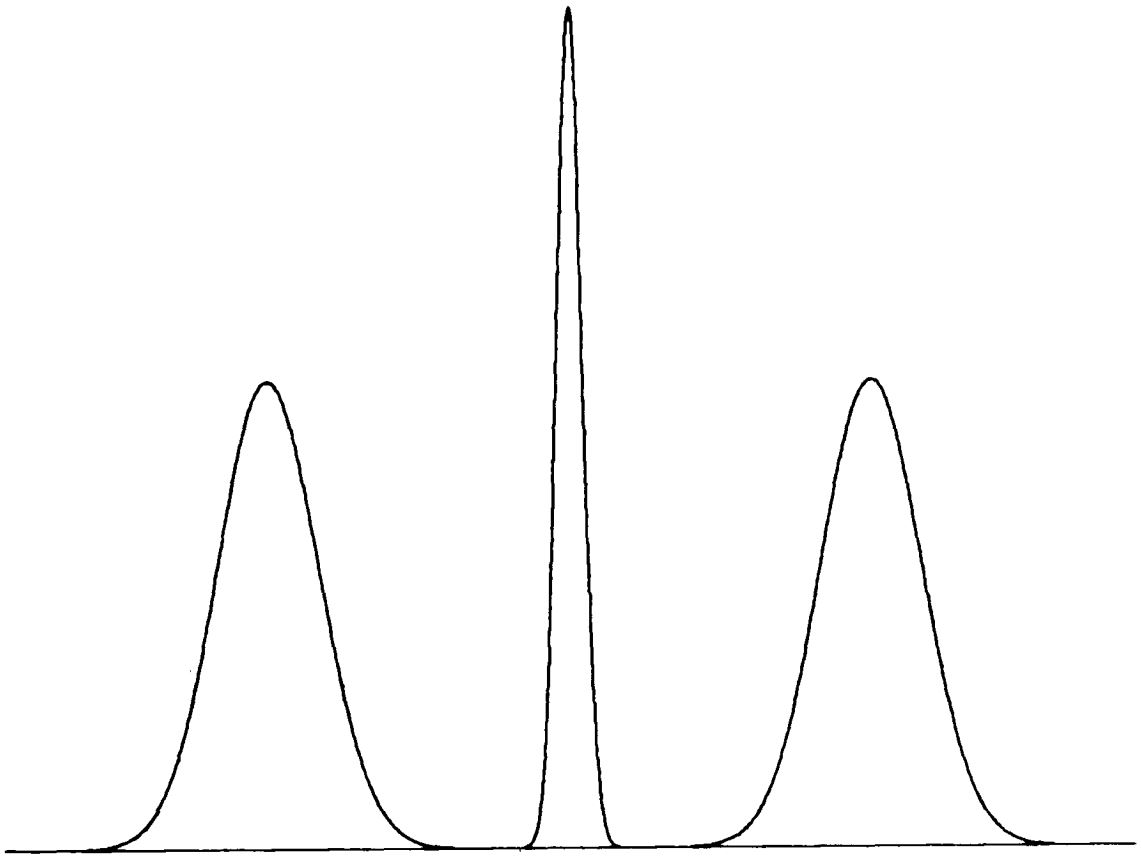}
\end{center}
\caption{Diaconis-Freedman's counterexample \label{fig:diaconis}}
\end{figure}

\bigskip

\begin{theorem}
Let $\Pi$ be a prior on $\Theta$, and assume that the model is dominated by a common $\sigma$-finite measure.
If for every $\epsilon > 0$,
\be \label{eq:kl_positive}
	\Pi\left( \big\{\theta\in\Theta: K_{P_{\theta_0}}(P_\theta) < \epsilon \big\} \right) > 0
\ee
and there exists a uniformly consistent sequence of tests 
for testing $H_0: \theta = \theta_0$ versus $H_1: d(\theta,\theta_0) > \epsilon$, then
the posterior is almost-surely consistent.
\end{theorem}

\bigskip

There are many interesting examples satisfying the Scwartz's condition.
\citet{barron1999consistency} founds a sufficient condition using bracketing number
for consistency with respect to Hellinger distance..
Some extensions to semiparametric models and non-\iid\; models can be found,
for example, in \citet{amewou2003posterior} and \citet{wu2008posterior}.
More recently \citet{walker2004new} founds a new sufficient condition for posterior consistency.

Many statisticians do not fully satisfy posterior consistency and they want to know
how fast it converges to the true parameter.
As an extension of Scwartz's theorem, \citet{ghosal2000convergence} found sufficient conditions
which assures a certain rate of posterior consistency.
Let $D(\epsilon, \Theta, d)$ denote the $\epsilon$-\emph{packing number} of $\Theta$, that is,
the maximal number of points in $\Theta$ such that the distance between every pair is at least $\epsilon$.
This is related to the covering number by the inequalities
$$N(\epsilon, \Theta, d) \leq D(\epsilon, \Theta, d) \leq N(\epsilon/2, \Theta, d).$$
The following general theorem given in \citet{ghosal2000convergence} is 
very intuitive and interpretable.

\bigskip

\begin{theorem} \label{thm:ghosal2000}
Let $d$ be the metric on $\Theta$ defined by $d(\theta_1, \theta_2) = h(p_{\theta_1}, p_{\theta_2})$
or $d(\theta_1, \theta_2) = d_V(p_{\theta_1}, p_{\theta_2})$.
Suppose that for a sequence $\epsilon_n$ with $\epsilon_n \rightarrow 0$ and $n\epsilon_n^2 \rightarrow \infty$,
a constant $C > 0$ and sets $\Theta_n \subset \Theta$, we have
\bea
	\log D(\epsilon_n, \Theta_n, d) \leq n\epsilon_n^2, \label{eq:rate_entropy}
	\\
	\Pi(\Theta \backslash \Theta_n) \leq \exp ( -n\epsilon_n^2 (C+4) ) \nonumber,
\eea
and
\be
	\Pi \bigg( \Big\{\theta: -P_{\theta_0} \log\Big(\frac{p_\theta}{p_{\theta_0}}\Big) \leq \epsilon_n^2, ~~
	\log\Big(\frac{p_\theta}{p_{\theta_0}}\Big)^2 \leq \epsilon_n^2 \Big\} \; \Big| \; X_1, \ldots, X_n \bigg) \geq
	\exp(-n\epsilon_n^2 C). \label{eq:kl_rate}
\ee
Then for sufficiently large $M$, we have that
$$
	\Pi\Big(\{\theta\in\Theta: d(\theta,\theta_0) \geq M \epsilon_n\} | X_1, \ldots, X_n\Big) \rightarrow 0
$$
in $P_0^n$-probability.
\end{theorem}

A sequence $\Theta_n$ is a sieve for $\Theta$.
Condition \eqref{eq:rate_entropy} requires that the model is not too big.
The log of covering number is called \emph{entropy} and this is often interpreted
as the complexity of the model (\citet{birge1983approximation, lecam1973convergence}).
Under certain conditions a rate satisfying \eqref{eq:rate_entropy} gives the optimal rate of convergence
relative to the Hellinger metric.
Condition \eqref{eq:rate_entropy} ensures the existence of certain tests and could be replaced by a testing condition.
Condition \eqref{eq:kl_rate} requires that the prior mass around the true parameter is not too small,
and this is a refined version of condition \eqref{eq:kl_positive}.
Roughly speaking condition \eqref{eq:kl_rate} tells that the prior mass should be uniformly spread
on the support of the prior.

An important application of Theorem \ref{thm:ghosal2000} is Dirichlet process mixture priors
for density estimation problems.
\citet{ghosal2001entropies} found a tight entropy bound for classes of mixtures of normal densities
and got Hellinger convergence rate $\epsilon_n = n^{-1/2} \log n$ when the true density is a mixture of normals.
Note that this is nearly parametric rate.
Although the true density is not a mixture of normal densities, a Dirichlet process mixture of normals prior
works well if the prior mass for standard deviance of normal is concentrated around zero as $n \rightarrow \infty$.
When the true density is twice continuously differentiable,
\citet{ghosal2007posterior} proved that a Dirichlet process mixture of normals prior gives
Hellinger convergence rate $\epsilon_n = n^{-2/5} \log n$ which is almost same to the optimal 
rate $n^{-2/5}$ of kernel density estimation.

Conditions in Theorem \ref{thm:ghosal2000} may be slightly strong than required,
and more refined versions are given in \citet{ghosal2000convergence}.
\citet{shen2001rates} independently found similar sufficient conditions 
for posterior convergence rate around the same time.
More recently \citet{walker2007rates} developed new conditions as an extension of \citet{walker2004new}
and provided an example which gives a better convergence rate than previous works.
When the model is misspecified, \citet{kleijn2006misspecification} proved that
the posterior converges to the parameter in the support at minimal Kullback-Leibler divergence to the true parameter,
at rate as if it were in the support.

\chapter{Main results}
\label{chap:main}

\section{Semiparametric Bernstein-von Mises theorem}

Consider a sequence of statistical models 
$\scrP^{(n)} = \{P_{\theta,\eta}^{(n)}: \theta \in \Theta, \eta \in \cH\}$
parametrized by finite dimensional $\theta$ of interest and 
infinite dimensional $\eta$ which is usually considered as a nuisance parameter.
Assume that $\Theta$ is an open subset of $\bbR^p$
and $P_{\theta,\eta}^{(n)}$ has the density $x \mapsto p_{\theta,\eta}^{(n)}(x)$
with respect to a $\sigma$-finite measure $\mu_n$.
Let $X^{(n)}$ be a random element which follows $P_0^{(n)}$ and assume that
$P_0^{(n)} = P_{\theta_0,\eta_0}^{(n)}$ for some $\theta_0 \in \Theta$ and $\eta_0 \in \cH$.
We consider a product prior $\Pi_\Theta \times \Pi_\cH$ on $\Theta \times \cH$ and
denote the posterior distribution by $\Pi(\cdot|X^{(n)})$.
Assume that $\Pi_\Theta$ is \emph{thick} at $\theta_0$, that is, 
it has a positive continuous Lebesgue density in a neighborhood of $\theta_0$.
Also $\Pi_\cH$ is allowed to depend on $n$, but we abbreviate the notation $n$
in $\Pi_\cH$ for notational simplicity.
For a given prior distribution $\Pi_\cH$ on $\cH$, let 
\be \label{eq:i_lik}
	s_n(h) = \int_\cH \frac{p^{(n)}_{\theta_n(h),\eta}}{p^{(n)}_{\theta_0,\eta_0}}(X^{(n)}) d\Pi_\cH(\eta)
\ee
be the \emph{integrated likelihood}, where $\theta_n(h) = \theta_0 + h/\sqrt{n}$.
We begin this section with the statement of general BvM theorem.
The proof is almost identical to that of Theorem 2.1 in \citet{kleijn2012bernstein}
upon replacement of parametric likelihoods with integrated likelihoods.
Hereafter, some quantities in proofs may not be measurable, and in this case
the expectation can be understood by the outer integral and measurable majorants.
We refer to Part I of \citet{van1996weak} for details about this.

\bigskip

\begin{theorem} \label{thm:BvM_general}
Assume that the model $\{P^{(n)}_{\theta,\eta}: \theta \in \Theta, \eta \in \cH\}$
is endowed with the product prior $\Pi=\Pi_\Theta \times \Pi_\cH$,
where $\Pi_\Theta$ is thick at $\theta_0$, and
\be \label{eq:sqrtn_general}
	P_0^{(n)} \Pi\big(\sqrt{n}|\theta-\theta_0| > M_n | X^{(n)}\big) \rightarrow 0
\ee
for every real sequence $(M_n)$ with $M_n \rightarrow \infty$.
Furthermore, suppose that for given sequences of uniformly tight random vectors $(\Delta_n)$
and non-random positive definite matrices $(V_n)$ satisfying
$\liminf_{n\rightarrow\infty}\rho_{\min} (V_n) > 0$,
the integrated likelihood \eqref{eq:i_lik} satisfies
\be \label{eq:ilan_general}
	\sup_{h \in K} \bigg| \log \frac{s_n(h)}{s_n(0)}
	- h^T V_n \Delta_n + \frac{1}{2} h^T V_n h \bigg| = o_{P_0}(1)
\ee
for any compact $K\subset\bbR^p$.
Then,
\be \label{eq:bvm_general}
	\sup_B \left|\Pi\big(\sqrt{n}(\theta-\theta_0) \in B | X^{(n)}\big) - N_{\Delta_n, V_n^{-1}}(B)
	\right| \rightarrow 0
\ee
in $P_0^{(n)}$-probability.
\end{theorem}
\begin{proof}
We first prove the assertion conditional on an arbitrary compact set $K \subset \bbR^p$
and then we use this to prove \eqref{eq:bvm_general}.
Let $\Psi_n$ be the normal distribution $N_{\Delta_n,V_n^{-1}}$
and $\Pi_n(B) = \Pi (\sqrt{n}(\theta-\theta_0) \in B | X^{(n)})$.
For any set $K \subset \bbR^p$ with $\Pi_n(K) > 0$, 
we define a conditional version $\Pi_n^K$ by
$\Pi_n^K(B) = \Pi_n(B\cap K)/\Pi_n(K)$.
Similarly we define a conditional measure $\Psi_n^K$ corresponding to $\Psi_n$

Let $K \subset \bbR^p$ be a compact set containing a neighborhood of $0$,
and $\Xi_n$ be the event that $\Pi_n(K) > 0$.
Then, for any open neighborhood $U \subset \Theta$ of $\theta_0$,
$\theta_0 +K/\sqrt{n} \subset U$ for large enough $n$.
Since $\theta_0$ is an interior point of $\Theta$, for large enough $n$, the random functions
$f_n:K\times K\rightarrow \bbR$,
$$f_n(g,h) = \left( 1-\frac{\psi_n(h)s_n(g)\pi_n(g)}{\psi_n(g)s_n(h)\pi_n(h)} \right)_+,$$
are well defined, where $\psi_n$ is the density of $\Psi_n$, $\pi_n$ is the density of the prior
for the centered and rescaled parameter $h=\sqrt{n}(\theta-\theta_0)$, and $x_+ = x \vee 0$.
Note that $\sup_{h,g \in K}\pi_n(g)/\pi_n(h) \rightarrow 1$ as $n\rightarrow\infty$.
Therefore,
$$
	\sup_{h,g\in K}\left|\log\frac{\psi_n(h)s_n(g)\pi_n(g)}{\psi_n(g)s_n(h)\pi_n(h)}\right| = o_{P_0}(1)
$$
by \eqref{eq:ilan_general} and we conclude that
$$\sup_{g,h\in K} f_n(g,h) \stackrel{P_0}{\rightarrow} 0$$
as $n \rightarrow \infty$.

Let $\epsilon > 0$ be given and define $\Omega_n = \left\{ \sup_{g,h\in K} f_n(g,h) \leq \epsilon\right\}$.
Since the total variation is bounded by 2,
$$
	P_0^{(n)} d_V\left(\Pi_n^K, \Psi_n^K\right) 1_{\Xi_n} \leq 
	P_0^{(n)} d_V\left(\Pi_n^K, \Psi_n^K\right) 1_{\Omega_n\cap\Xi_n}
	+2P_0^{(n)}(\Xi_n\backslash\Omega_n).
$$
Note that
\bean
	\frac{1}{2} P_0^{(n)} d_V\left(\Pi_n^K, \Psi_n^K\right) 1_{\Omega_n\cap\Xi_n}
	&=& P_0^{(n)} \int \left(1 - \frac{d\Psi_n^K}{d\Pi_n^K}\right)_+ d\Pi_n^K 1_{\Omega_n\cap\Xi_n}\\
	&=& P_0^{(n)} \int \left(1 - \int \frac{\psi^K_n(h)s_n(g)\pi_n(g)}
	{\psi_n^K(g)s_n(h)\pi_n(h)} d\Psi_n^K(g)\right)_+ d\Pi_n^K(h) 1_{\Omega_n\cap\Xi_n}
\eean
and $\psi_n^K(h) / \psi_n^K(g) = \psi_n(h)/\psi_n(g)$ for all $g,h\in K$.
Therefore, by the Jensen's inequality on the function $x \mapsto (1-x)_+$, we have
\bean
	\frac{1}{2} P_0^{(n)} d_V\left(\Pi_n^K, \Psi_n^K\right) 1_{\Omega_n\cap\Xi_n}
	&\leq& P_0^{(n)} \int\int \left(1 - \frac{\psi_n(h)s_n(g)\pi_n(g)}
	{\psi_n(g)s_n(h)\pi_n(h)}\right)_+ d\Psi_n^K(g) d\Pi_n^K(h) 1_{\Omega_n\cap\Xi_n} \\
	&\leq& P_0^{(n)} \int\int \sup_{g,h\in K} f_n(g,h) 1_{\Omega_n\cap\Xi_n} d\Psi_n^K(g) d\Pi_n^K(h) \leq \epsilon.
\eean
Since $P_0^{(n)}(\Omega_n^c) \rightarrow 0$, we conclude
$P_0^{(n)} d_V\left(\Pi_n^K,\Psi_n^K\right) 1_{\Xi_n}\rightarrow 0$.

Now, we can choose a sequence of balls $(K_n)$ centered at 0 with radii $M_n \rightarrow \infty$
and satisfying $P_0^{(n)} d_V\left(\Pi_n^{K_n},\Psi_n^{K_n}\right) 
1_{\Xi_n}\rightarrow 0$, where $\Xi_n$ is redefined by the event $\{\Pi_n(K_n) > 0\}$.
Note that
$$
	d_V\left(\Pi_n,\Psi_n\right) \leq 2 \cdot 1_{\Xi_n^c}
	+ d_V\left(\Pi_n,\Pi_n^{K_n}\right) 1_{\Xi_n}
	+ d_V\left(\Pi_n^{K_n},\Psi_n^{K_n}\right) 1_{\Xi_n}
	+ d_V\left(\Psi_n,\Psi_n^{K_n}\right)
$$
and $P_0^{(n)}(\Xi_n^c) \rightarrow 0$ by \eqref{eq:sqrtn_general}.
We also have $P_0^{(n)} d_V\left(\Pi_n,\Pi_n^{K_n}\right)1_{\Xi_n} \leq 
2 P_0^{(n)} \Pi_n(K_n^c) \rightarrow 0$ by \eqref{eq:sqrtn_general}.

It only remains to prove $P_0^{(n)} d_V\left(\Psi_n,\Psi_n^{K_n}\right) \rightarrow 0$.
For that, it is sufficient to show that 
$\Psi_n(K_n^c) = N_{\Delta_n, V_n^{-1}}(K_n^c) \rightarrow 0$ converges in $P_0^{(n)}$-probability.
This follows by the fact that $(\Delta_n)$ is uniformly tight and
$\liminf_{n\rightarrow \infty} \rho_{\min}(V_n) > 0$.
\end{proof}

\bigskip

Note that if $h \mapsto s_n(h)$ is continuous $P_0^{(n)}$-almost-surely, then
\eqref{eq:ilan_general} is equivalent to
$$
	\log \frac{s_n(h_n)}{s_n(0)}
	- h_n^T V_n \Delta_n + \frac{1}{2} h_n^T V_n h_n = o_{P_0}(1)
$$
for every bounded random sequence $(h_n)$.

Conditions in Theorem \ref{thm:BvM_general} are quite intuitive, but not easy to prove.
In the following two subsections, we provide sufficient conditions for the conditions
\eqref{eq:sqrtn_general} and \eqref{eq:ilan_general} for models
in which there is no information loss.
These conditions are given as follows.

There exist a positive number $\epsilon_0 > 0$,
$L_2(P_0^{(n)})$-functions $\score^{(n)}_{\theta,\eta}$,
a sequence $(\cH_n)$ of subsets of $\cH$ containing $\eta_0$, and  
matrices $V_{n,\eta}$ satisfying
\bea \label{eq:zero_score_condition}
	\sup_{\eta\in\cH_n} |P_0^{(n)} \score^{(n)}_{\theta_0,\eta}| &=& o(n^{1/2})
	\\
	\label{eq:score_conti_condition}
	\sup_{\eta\in\cH_n} \left| \score_{\theta_0,\eta}^{(n)}(X^{(n)})
	- \score_{\theta_0,\eta_0}^{(n)}(X^{(n)}) \right| &=& o_{P_0}(n^{1/2})
	\\
	\label{eq:donsker_condition}
	\sup_{|\theta-\theta_0| < \epsilon_0} \sup_{\eta\in\cH_n} \left|\score_{\theta,\eta}^{(n)}(X^{(n)})
	- P_0^{(n)}\score_{\theta,\eta}^{(n)}\right| &=& O_{P_0}(n^{1/2})
	\\
	\label{eq:V_conti_condition}
	\sup_{\eta\in\cH_n} \|V_{n,\eta} - V_{n,\eta_0}\| &=& o(1)
	\\
	\label{eq:V_positive}
	0 < \liminf_{n\rightarrow\infty}\rho_{\min}(V_{n,\eta_0}) < 
	\limsup_{n\rightarrow\infty}\rho_{\max}(V_{n,\eta_0}) &<& \infty
\eea
and for large enough $N$
\be
	\label{eq:quadratic_condition}
	\sup_{n\geq N} \sup_{\eta\in\cH_n} \left| \frac{1}{n}P_0^{(n)}
	\Big( \ell^{(n)}_{\theta,\eta} - \ell^{(n)}_{\theta_0,\eta}\Big)
	+ \frac{1}{2} (\theta-\theta_0)^T V_{n,\eta} (\theta-\theta_0) \right| = o(|\theta-\theta_0|^2)
\ee
as $\theta\rightarrow\theta_0$.
Furthermore,
\be \label{eq:lr_apprx_small_condition}
	\sup_{|h| \leq M_n} \sup_{\eta\in\cH_n}
	\left|\left( \ell^{(n)}_{\theta_n(h),\eta}(X^{(n)}) - \ell^{(n)}_{\theta_0,\eta}(X^{(n)})
	- \frac{h^T}{\sqrt{n}} \score_{\theta_0,\eta}^{(n)}(X^{(n)}) \right)^o\right|
	\cdot (1 \vee |h|^2)^{-1} =  o_{P_0}(1)
\ee
and
\be \label{eq:lr_apprx_large_condition}
	\sup_{M_n < |h| < \epsilon_0 \sqrt{n}} \sup_{\eta\in\cH_n}
	\left|\left( \ell^{(n)}_{\theta_n(h),\eta}(X^{(n)}) - \ell^{(n)}_{\theta_0,\eta}(X^{(n)})
	\right)^o\right| \cdot |h|^{-2}
	= o_{P_0}(1)
\ee
for every $(M_n)$, $M_n \rightarrow \infty$ and $M_n/\sqrt{n}\rightarrow 0$,
where $X^o = X - P_0^{(n)} X$ is the centered random variable of $X$.

These conditions are highly related to those of \citet{van1996efficient}
which prove the efficiency of maximum likelihood estimators in semiparametric models.
The most important condition in \citet{van1996efficient} is that 
a class of score functions is Donsker,
which implies uniformly asymptotic equicontinuity or asymptotic tightness of the stochastic processes.
This corresponds to conditions \eqref{eq:score_conti_condition} and \eqref{eq:donsker_condition}.
Condition \eqref{eq:score_conti_condition} is related to the asymptotic equicontinuity
of the stochastic process and \eqref{eq:donsker_condition} is a direct result of asymptotic tightness.
Both properties can be proved by showing that the stochastic process
$$
	(\theta,\eta) \mapsto \frac{1}{\sqrt{n}} \left(
	\score^{(n)}_{\theta,\eta}(X^{(n)}) - P_0^{(n)} \score^{(n)}_{\theta,\eta}\right)
$$
indexed by a neighborhood of $(\theta_0,\eta_0)$ is asymptotically tight.
Modern empirical process theory is an useful tool for proving this property.
Once \eqref{eq:donsker_condition} is shown to be true,
\eqref{eq:lr_apprx_small_condition} and \eqref{eq:lr_apprx_large_condition}
can be easily checked by Taylor expansion of $\theta \mapsto \ell_{\theta,\eta}^{(n)}(x)$
provided it is smooth.
Condition \eqref{eq:zero_score_condition} implies that
the expectation of the ordinary score function vanishes near $\eta_0$ at order $n^{-1/2}$
and this is similar to condition (2.9) of \citet{van1996efficient}.
Condition \eqref{eq:quadratic_condition} is that the expectation of the log likelihood ratio
is approximated by a quadratic function near $\theta_0$.
Therefore if the model is smooth, \eqref{eq:V_conti_condition}, \eqref{eq:V_positive} and
\eqref{eq:quadratic_condition} imply \eqref{eq:zero_score_condition}.
Note that conditions \eqref{eq:V_conti_condition} and \eqref{eq:V_positive} are natural,
so \eqref{eq:quadratic_condition} is the most stringent to prove.
For models considered in this thesis,
the symmetricity of densities make an important role to prove \eqref{eq:quadratic_condition}.

\subsection{Integral local asymptotic normality}

In this subsection, we prove the integral LAN condition \eqref{eq:ilan_general}
using conditions mentioned above.
A key requirement is the uniform LAN \eqref{eq:uLAN} which can be proved by
the quadratic expansion \eqref{eq:quadratic_condition} and application of the empirical process theory.
Another important condition is \eqref{eq:conv_perturb_general}
which is the consistency of nuisance posterior under $n^{-1/2}$-perturbation of $\theta$.
For \iid~models, a well-established theory is given in Theorem 3.1 of \citet{bickel2012semiparametric}.
An extension to non-\iid~independent models can be found in Theorem \ref{thm:consist_perturb_general} of
Section \ref{sec:consistency}.

\bigskip

\begin{theorem}[Integral LAN] \label{thm:ILAN_general}
Suppose that \eqref{eq:zero_score_condition}, \eqref{eq:score_conti_condition}, 
\eqref{eq:V_conti_condition}, \eqref{eq:quadratic_condition} 
and \eqref{eq:lr_apprx_small_condition} hold for some $\score_{\theta,\eta}^{(n)}$, $(\cH_n)$ and $(V_{n,\eta})$.
Furthermore, assume that
\be \label{eq:conv_perturb_general}
	P_0^{(n)} \Pi\big(\cH_n | \theta=\theta_n(h_n), X^{(n)}\big) \rightarrow 1
\ee
for every bounded random sequence $(h_n)$.
Then,
$$
	\log \frac{s_n(h_n)}{s_n(0)} = 
	\frac{h_n^T}{\sqrt{n}} \score_{\theta_0,\eta_0}^{(n)} (X^{(n)}) 
	- \frac{1}{2} h_n^T V_{n,\eta_0}h_n + o_{P_0}(1)
$$
holds for every bounded random sequence $(h_n)$.
\end{theorem}
\begin{proof}
For a given compact set $K \subset \bbR^p$, let
\bean
	R_{1,n}(h) &=& \sup_{h \in K}\sup_{\eta \in \cH_n} \bigg|
	\log \frac{p^{(n)}_{\theta_n(h), \eta}}{p^{(n)}_{\theta_0, \eta}}(X^{(n)})
	-G_n(h, \eta) \bigg| \\
	R_{2,n}(h) &=& \sup_{h \in K} \sup_{\eta \in \cH_n} \big| G_n(h, \eta) 
	- G_n(h, \eta_0) \big|,
\eean
where 
$$
	G_n(h,\eta) = \frac{h^T}{\sqrt{n}} \score_{\theta_0,\eta}^{(n)} (X^{(n)}) - \frac{1}{2} h^T V_{n,\eta}h.
$$
Then, $R_{1,n}(h) = o_{P_0}(1)$ by Lemma \ref{lem:uLAN} and
$R_{2,n}(h) = o_{P_0}(1)$ by \eqref{eq:score_conti_condition} and \eqref{eq:V_conti_condition}.
Let $\epsilon, \delta > 0$ and a random sequence $(h_n)$ in $K$ be given,
and let $M_n$ be the maximum of $-\log \Pi\big[\cH_n| \theta =\theta_n(h_n), X^{(n)}\big]$, 
$-\log \Pi\big[\cH_n| \theta =\theta_0, X^{(n)}\big]$, $R_{1,n}(h_n)$ and $R_{2,n}(h_n)$.
If we define $A_n$ by the event $\{ M_n < \epsilon / 3\}$,
then $P_0^{(n)}(A_n) \geq 1-\delta$ for large enough $n$.
On $A_n$, we have
\bea
	\int_\cH \frac{p^{(n)}_{\theta_n(h_n), \eta}}{p^{(n)}_{\theta_0, \eta_0}}(X^{(n)}) d\Pi_\cH(\eta)
	&\leq& e^{\epsilon/3} \int_{\cH_n} 
	\frac{p^{(n)}_{\theta_n(h_n), \eta}}{p^{(n)}_{\theta_0, \eta_0}}(X^{(n)}) d\Pi_\cH(\eta) \nonumber \\
	&\leq& e^{2\epsilon/3}\int_{\cH_n} e^{G_n(h_n, \eta)} 
	\frac{p^{(n)}_{\theta_0, \eta}}{p^{(n)}_{\theta_0, \eta_0}}(X^{(n)}) d\Pi_\cH(\eta) \nonumber \\
	&\leq& e^{\epsilon + G_n(h_n, \eta_0)} \int_{\cH_n}  
	\frac{p^{(n)}_{\theta_0, \eta}}{p^{(n)}_{\theta_0, \eta_0}}(X^{(n)}) d\Pi_\cH(\eta) \nonumber \\
	&\leq& e^{\epsilon + G_n(h_n, \eta_0)} 
	\int_\cH \frac{p^{(n)}_{\theta_0, \eta}}{p^{(n)}_{\theta_0, \eta_0}}(X^{(n)}) d\Pi_\cH(\eta) \label{eq:upper_bound}
\eea
and
\bea
	\int_\cH \frac{p^{(n)}_{\theta_n(h_n), \eta}}{p^{(n)}_{\theta_0, \eta_0}}(X^{(n)}) d\Pi_\cH(\eta)
	&\geq& \int_{\cH_n} 
	\frac{p^{(n)}_{\theta_n(h_n), \eta}}{p^{(n)}_{\theta_0, \eta_0}}(X^{(n)}) d\Pi_\cH(\eta) \nonumber \\	
	&\geq& e^{-\epsilon/3}\int_{\cH_n} e^{G_n(h_n, \eta)} 
	\frac{p^{(n)}_{\theta_0, \eta}}{p^{(n)}_{\theta_0, \eta_0}}(X^{(n)}) d\Pi_\cH(\eta) \nonumber \\
	&\geq& e^{-2\epsilon/3 + G_n(h_n, \eta_0)} \int_{\cH_n}
	\frac{p^{(n)}_{\theta_0, \eta}}{p^{(n)}_{\theta_0, \eta_0}}(X^{(n)}) d\Pi_\cH(\eta) \nonumber \\
	&\geq& e^{-\epsilon + G_n(h_n, \eta_0)} 
	\int_\cH \frac{p^{(n)}_{\theta_0, \eta}}{p^{(n)}_{\theta_0, \eta_0}}(X^{(n)}) d\Pi_\cH(\eta) \label{eq:lower_bound},
\eea
where the last inequality of \eqref{eq:lower_bound} holds
by the consistency of the posterior of $\eta$ given $\theta = \theta_0$.
The inequalities \eqref{eq:upper_bound} and \eqref{eq:lower_bound} can be summarized by
$$
	P_0^n \left[ \left| \log \frac{s_n(h_n)}{s_n(0)} - 
	G_n (h_n, \eta_0) \right| \geq \epsilon \right] < \delta
	$$
and this yields the desired result.
\end{proof}

\bigskip

\begin{lemma}[Uniform LAN] \label{lem:uLAN}
Assume that \eqref{eq:zero_score_condition}, \eqref{eq:quadratic_condition} 
and \eqref{eq:lr_apprx_small_condition} hold for some $\score_{\theta,\eta}^{(n)}$, $(\cH_n)$ and $(V_{n,\eta})$.
Then, the uniform LAN assertion 
\be \label{eq:uLAN}
	\sup_{h \in K}\sup_{\eta \in \cH_n} \bigg|
	\log \frac{p^{(n)}_{\theta_n(h), \eta}}{p^{(n)}_{\theta_0, \eta}}(X^{(n)})
	- \frac{h^T}{\sqrt{n}} \score_{\theta_0,\eta}^{(n)} (X^{(n)}) + \frac{1}{2} h^T V_{n,\eta}h
	\bigg| = o_{P_0}(1)
\ee
holds.
\end{lemma}
\begin{proof}
We can rewrite the left hand side of \eqref{eq:uLAN} by
\bean
	&& \left(\ell_{\theta_n(h),\eta}^{(n)} (X^{(n)}) - \ell_{\theta_0,\eta}^{(n)} (X^{(n)})
	- \frac{h^T}{\sqrt{n}} \score_{\theta_0,\eta}^{(n)} (X^{(n)})\right)^o \\
	&& ~~~~ - \frac{h^T}{\sqrt{n}} P_0^{(n)} \score^{(n)}_{\theta_0,\eta}
	+ P_0^{(n)} \Big( \ell^{(n)}_{\theta_n(h),\eta} - \ell^{(n)}_{\theta_0,\eta}\Big) 
	+ \frac{1}{2} h^T V_{n,\eta}h,
\eean
and for $h$ in a compact set $K\subset \bbR$ and $\eta\in\cH_n$,
the supremum of the first term converges to 0 in $P_0^{(n)}$-probability
by \eqref{eq:lr_apprx_small_condition}.
The last three terms also converges uniformly to 0 by
\eqref{eq:zero_score_condition} and \eqref{eq:quadratic_condition}.
\end{proof}

\bigskip

\subsection{Parametric convergence rate of the marginal posterior}

In this subsection, the marginal posterior of $\theta$ is shown to converge at parametric rate $n^{-1/2}$.
It looks very natural but the proof is not easy as mentioned in \citet{bickel2012semiparametric}.
We apply the second approach given in Section 6 of \citet{bickel2012semiparametric}.
The proof is quite technical and we motivated from the proofs of
Theorem 2.4 in \citet{ghosal2000convergence} and Theorem 3.1 in \citet{kleijn2012bernstein}.

There are extensive literatures about posterior consistency, condition \eqref{eq:consist}.
The version that adapts to our examples is given in Theorem \ref{thm:consist_general}.

\bigskip

\begin{theorem} \label{thm:sqrt_n_general}
Suppose that
\eqref{eq:donsker_condition}--\eqref{eq:lr_apprx_large_condition} hold for some 
$\score_{\theta,\eta}^{(n)}$, $(\cH_n)$, $(V_{n,\eta})$
and sufficiently small $\epsilon_0 > 0$.
Also, the posterior is consistent in the sense that
\be \label{eq:consist}
	\Pi\left(|\theta-\theta_0| < \epsilon, \eta\in\cH_n | X^{(n)}\right) \rightarrow 1
\ee
in $P_0^{(n)}$-probability for every $\epsilon > 0$.
Then, \eqref{eq:sqrtn_general} holds for every $M_n \rightarrow \infty$
provided $\Pi_\Theta$ is thick at $\theta_0$.
\end{theorem}
\begin{proof}
It is sufficient to show that \eqref{eq:sqrtn_general} holds 
for sufficiently slowly increasing $(M_n)$ so that $M_n/\sqrt{n} \rightarrow 0$.
For given such $(M_n)$, we can choose $C > C_1 > 0$ and $C_2 > 0$ satisfying the assertions
of Lemmas \ref{lem:sqrtn_lower_bound} and \ref{lem:sqrtn_upper_bound}.
Let $\Omega_n$ be the intersection of two events whose probabilities are tending to 1 in the both Lemmas.
For a given $\epsilon_0 > 0$ (see below),
let $\Theta_n = \{\theta\in\Theta: M_n/\sqrt{n} < |\theta-\theta_0| \leq \epsilon_0\}$,
$\Theta_{n,j} = \{\theta_n(h) \in \Theta_n: j M_n \leq |h| < (j+1) M_n\}$
and $J$ be the minimum among $j$'s satisfying $(j+1)M_n/\sqrt{n} > \epsilon_0$.
Since $\Pi_\Theta$ is thick at $\theta_0$, $\epsilon_0$ can be chosen sufficiently small so that
$\Pi_\Theta(\Theta_{n,j}) \leq D \cdot \big((j+1)M_n/\sqrt{n}\big)^p$ for some constant $D > 0$.
Then on $\Omega_n$, 
\bean
	\sup_{\eta\in\cH_n} \Pi(\theta \in \Theta_n | \eta, X^{(n)})
	&\leq& \frac{e^{C_1 M_n^2}}{C_2 (M_n/\sqrt{n})^p} 
	\sup_{\eta\in\cH_n} \int_{\Theta_n} 
	\frac{p^{(n)}_{\theta, \eta}}{p^{(n)}_{\theta_0, \eta}}(X^{(n)}) d\Pi_\Theta(\theta) \\
	&\leq& \frac{e^{C_1 M_n^2}}{C_2 (M_n/\sqrt{n})^p}
	\sum_{j=1}^J \Pi_\Theta(\Theta_{n,j}) \sup_{\theta \in \Theta_{n,j}} \sup_{\eta\in\cH_n} 
	\frac{p^{(n)}_{\theta, \eta}}{p^{(n)}_{\theta_0, \eta}}(X^{(n)}).
\eean
Since
\bean
	\sup_{\theta \in \Theta_{n,j}} \sup_{\eta\in\cH_n} 
	\frac{p^{(n)}_{\theta, \eta}}{p^{(n)}_{\theta_0, \eta}}(X^{(n)})	\leq \exp(-C j^2 M_n^2)
\eean
on $\Omega_n$, we have on this set
\bean
	\sup_{\eta\in\cH_n} \Pi(\theta \in \Theta_n | \eta, X^{(n)}) 
	&\leq& C_2^{-1} D e^{C_1M_n^2} \sum_{j=1}^J (j+1)^p e^{-Cj^2M_n^2} \rightarrow 0
\eean
as $n \rightarrow \infty$, by the choice of $C > C_1$.
We conclude that
$$\sup_{\eta\in\cH_n} \Pi\big(\theta\in\Theta_n | \eta,X^{(n)}\big) \rightarrow 0$$
in $P_0^{(n)}$-probability because $P_0^{(n)}(\Omega_n) \rightarrow 1$.
Now, we can write
\bean
	&& ~ \Pi\left(\sqrt{n}|\theta-\theta_0| > M_n | X^{(n)}\right)\\
	&&= \Pi\left(|\theta-\theta_0| > \epsilon_0 | X^{(n)}\right)
	+ \Pi\left(\theta \in \Theta_n | X^{(n)}\right) \\
	&&= \Pi\left(|\theta-\theta_0| > \epsilon_0 | X^{(n)}\right)
	+ \int \Pi\left(\theta \in \Theta_n | \eta, X^{(n)}\right) d\Pi(\eta|X^{(n)}) \\
	&&\leq \Pi\left(|\theta-\theta_0| > \epsilon_0 | X^{(n)}\right)
	+ \sup_{\eta\in\cH_n}\Pi\left(\theta \in \Theta_n | \eta, X^{(n)}\right)
	+ \Pi(\eta\in\cH_n^c| X^{(n)})
\eean
and each term converges in $P_0^{(n)}$-probability to 0.
\end{proof}

\bigskip

\begin{lemma} \label{lem:sqrtn_lower_bound}
For given $(M_n)$, $M_n \rightarrow\infty$ and $M_n /\sqrt{n}\rightarrow 0$,
suppose that \eqref{eq:donsker_condition}--\eqref{eq:lr_apprx_small_condition}
hold for some $\epsilon_0$, $\score_{\theta,\eta}^{(n)}$, $(\cH_n)$ and $(V_{n,\eta})$.
Then, for every $C_1 > 0$, there exists $C_2 > 0$ such that
$$
	P_0^{(n)} \bigg[\inf_{\eta\in\cH_n} \int_\Theta 
	\frac{p^{(n)}_{\theta, \eta}}{p^{(n)}_{\theta_0, \eta}}(X^{(n)})
	d\Pi_\Theta(\theta) \geq C_2 \left(\frac{M_n}{\sqrt{n}}\right)^p e^{-C_1 M_n^2} \bigg]
	\rightarrow 1
$$
provided $\Pi_\Theta$ is thick at $\theta_0$.
\end{lemma}
\begin{proof}
For given $C > 0$ and $(M_n)$ with $M_n \rightarrow \infty$ and $M_n/\sqrt{n} \rightarrow 0$,
we have
\bean
	\int_\Theta \frac{p^{(n)}_{\theta, \eta}}{p^{(n)}_{\theta_0, \eta_0}}(X^{(n)}) d\Pi_\Theta(\theta) 
	&\geq& \int_{|h| \leq CM_n}
	\exp\bigg[ \left( \ell^{(n)}_{\theta_n(h),\eta}(X^{(n)}) - \ell^{(n)}_{\theta_0,\eta}(X^{(n)})
	- \frac{h^T}{\sqrt{n}}\score_{\theta_0,\eta}^{(n)}(X^{(n)}) \right)^o \\
	&& ~~~~~~~~~~~~~~~~~~~~ + \frac{h^T}{\sqrt{n}} \left(\score_{\theta_0,\eta}^{(n)}(X^{(n)})
	- P_0^{(n)} \score_{\theta_0,\eta}^{(n)}\right) \\
	&& ~~~~~~~~~~~~~~~~~~~~ - \frac{1}{2} h^T V_{n,\eta_0} h \\
	&& ~~~~~~~~~~~~~~~~~~~~ + \frac{1}{2}h^T (V_{n,\eta_0} - V_{n,\eta})h \\
	&& ~~~~~~~~~~~~~~~~~~~~ + P_0^{(n)} \Big( \ell^{(n)}_{\theta_n(h),\eta}	
	- \ell^{(n)}_{\theta_0,\eta}\Big) + \frac{1}{2} h^T V_{n,\eta} h \bigg] d\Pi_n(h),
\eean
where $\Pi_n$ is the prior for the centered and rescaled parameter $h = \sqrt{n}(\theta-\theta_0)$.
For $|h| \leq C M_n$ and $\eta\in\cH_n$,
the exponent is uniformly bounded below by
$$ M_n^2 \left(-\frac{C^2}{2} \cdot\|V_{n,\eta_0}\| + o_{P_0}(1)\right) $$
by \eqref{eq:donsker_condition}, \eqref{eq:V_conti_condition}, \eqref{eq:quadratic_condition}
and \eqref{eq:lr_apprx_small_condition}.
Since $\limsup_{n\rightarrow\infty}\rho_{\max}(V_{n,\eta_0}) < \infty$ by \eqref{eq:V_positive},
$\Pi_\Theta$ is thick at $\theta_0$, and $C$ is arbitrary, we have the desired result.
\end{proof}

\bigskip

\begin{lemma} \label{lem:sqrtn_upper_bound}
For given $(M_n)$, $M_n \rightarrow\infty$ and $M_n /\sqrt{n}\rightarrow 0$,
suppose that \eqref{eq:V_conti_condition}--\eqref{eq:quadratic_condition}
and \eqref{eq:lr_apprx_large_condition} holds for some $\score_{\theta,\eta}^{(n)}$, $(\cH_n)$, $(V_{n,\eta})$
and sufficiently small $\epsilon_0 > 0$.
Then, there exists a constant $C > 0$ such that
$$
	P_0^{(n)} \bigg[\sup_{\eta\in\cH_n} \log
	\frac{p^{(n)}_{\theta_n(h), \eta}}{p^{(n)}_{\theta_0, \eta}}(X^{(n)})
	\leq -C |h|^2, ~ \textrm{for} ~ M_n < |h| < \epsilon_0 \sqrt{n} \bigg]
	\rightarrow 1
$$
as $n \rightarrow \infty$.
\end{lemma}
\begin{proof}
Let a real sequence $(M_n)$, $M_n \rightarrow \infty$ and $M_n/\sqrt{n} \rightarrow 0$, be given.
For given $\delta > 0$, if $\epsilon_0 > 0$ is sufficiently small, then
$$
	\left|P_0^{(n)} \Big( \ell^{(n)}_{\theta_n(h),\eta}	
	- \ell^{(n)}_{\theta_0,\eta}\Big) + \frac{1}{2} h^T V_{n,\eta} h\right| < \delta \cdot |h|^2
$$
for large enough $n$ and every $h$ with $|h| < \sqrt{n}\epsilon_0$ by \eqref{eq:quadratic_condition}.
Write
\bean
	\log\frac{p^{(n)}_{\theta_n(h), \eta}}{p^{(n)}_{\theta_0, \eta_0}}(X^{(n)})
	&=&  \left( \ell^{(n)}_{\theta_n(h),\eta}(X^{(n)}) - \ell^{(n)}_{\theta_0,\eta}(X^{(n)}) \right)^o \\
	&& ~~ + P_0^{(n)} \Big( \ell^{(n)}_{\theta_n(h),\eta}	
	- \ell^{(n)}_{\theta_0,\eta}\Big) + \frac{1}{2} h^T V_{n,\eta} h \\
	&& ~~ + \frac{1}{2}h^T (V_{n,\eta_0} - V_{n,\eta})h - \frac{1}{2} h^T V_{n,\eta_0} h.
\eean
Then, for $M_n < |h| < \epsilon_0\sqrt{n}$ and $\eta\in\cH_n$, the right hand side is uniformly bounded above by
$$ |h|^2 \cdot \left( -\frac{1}{2} \rho_{\min}(V_{n,\eta_0}) + \delta + o_{P_0}(1) \right) $$
by \eqref{eq:V_conti_condition} and \eqref{eq:lr_apprx_large_condition}.
Since $\delta > 0$ can be arbitrarily small and
$\liminf_{n\rightarrow\infty}\rho_{\min}(V_{n,\eta_0}) > 0$ by \eqref{eq:V_positive},
we have the desired result.
\end{proof}

\bigskip

\section{Quadratic expansion of the expected log likelihood ratio}
\label{sec:quadratic}

This section is devoted to study about uniform quadratic expansion of the expected log likelihood ratio 
\eqref{eq:quadratic_condition} in models with symmetric densities.
Typically in a smooth parametric model it is expected that
$$
	P_{\theta_0} \log\frac{p_\theta}{p_{\theta_0}} = -\frac{1}{2} (\theta-\theta_0)^T 
	I_{\theta_0} (\theta-\theta_0) + o(|\theta-\theta_0|^2)
$$
as $\theta\rightarrow\theta_0$ by use of Taylor expansion.
Here $I_{\theta_0}$ is the
Fisher information matrix at $\theta_0$.
In this expansion, the linear term is equal to zero because the model is well specified so
$\theta \mapsto P_{\theta_0} \ell_\theta$ is maximized at $\theta_0$.
To satisfy the condition \eqref{eq:quadratic_condition}, this quadratic expansion 
should be satisfied when the nuisance parameter is slightly misspecified.
This is not generally true, even in models without information loss.
Consider, for example, the Gaussian model $N(\mu, \sigma^2)$.
When $\sigma^2$ is misspecified the log likelihood ratio satisfies
\bean
	P_{\mu_0,\sigma_0} \log\frac{p_{\mu,\sigma}}{p_{\mu_0,\sigma}}
	= -\frac{(\mu-\mu_0)^2}{2\sigma^2}
\eean
so the quadratic expansion \eqref{eq:quadratic_condition} is satisfied.
In contrast, when $\mu$ is misspecified, the expected log likelihood ratio is given by
\bean
	P_{\mu_0,\sigma_0} \log\frac{p_{\mu,\sigma}}{p_{\mu,\sigma_0}}
	= -\frac{1}{2} \log \frac{\sigma^2}{\sigma_0^2} - \frac{1}{2} \Big(\sigma_0^2 + (\mu-\mu_0)^2 \Big)
	\left(\frac{1}{\sigma^2} - \frac{1}{\sigma_0^2} \right)
\eean
so it does not allow the desired quadratic expansion.
Note that the linear term of the Taylor expansion with respect to $\sigma^2$ is given by
$$
	\frac{(\mu-\mu_0)^2}{2\sigma_0^2} (\sigma^2 - \sigma_0^2),
$$
and the map $\sigma^2 \mapsto P_{\mu_0,\sigma_0}\ell_{\mu,\sigma}$
is maximized at $\sigma^2 = \sigma_0^2 + (\mu - \mu_0)^2$.
This implies that the condition \eqref{eq:quadratic_condition} may be difficult to be satisfied in general.
Fortunately, many interesting models satisfy this condition,
and we establish a sufficient condition for condition \eqref{eq:quadratic_condition}
in models with symmetric error.

We consider univariate and multivariate models with symmetric errors
in the following two subsections, respectively.
These models, like the Gaussian location model in which $\sigma^2$ is considered as nuisance parameter,
allow the desired quadratic expansion when nuisance parameter is misspecified.
Condition \eqref{eq:quadratic_condition} requires that this quadratic expansion happens uniformly
around the true parameter $\eta_0$.
We will provide sufficient conditions for uniform quadratic expansions
and prove a class of mixtures of normal densities satisfies these conditions.

\subsection{Univariate symmetric densities}

We first consider one-dimensional location problem
$$X = \theta + \epsilon,$$
where the error distribution is parametrized by $\eta\in\cH$ for some infinite dimensional $\cH$.
Write the density of error distribution by $p_\eta$ and let $p_{\theta,\eta}(x) = p_\eta(x-\theta)$.
A density $p_\eta$ is assumed to be symmetric about 0 and continuously differentiable for every $\eta \in \cH$.
Fix $(\theta_0,\eta_0)\in\bbR\times \cH$ which can be considered as the true parameter.
Define
$$
	V_\eta = P_{\theta_0,\eta_0} \big[ \ell^\prime_{\theta_0,\eta} \ell^\prime_{\theta_0,\eta_0}\big]
	= P_{\eta_0} \big[ \ell^\prime_{\eta} \ell^\prime_{\eta_0}\big]
$$
if it exists.
The following lemma is the key identity for our result
so we mention it before stating the main theorem.

\bigskip

\begin{lemma} \label{lem:hf}
If $f > 0$ and $f(x+\theta/2) \cdot f(-x + \theta/2) = 1$ for all $x,\theta\in\bbR$, then
\be \label{eq:hf}
	\int_{-\infty}^\infty h(x) \log f(x) dx =  - \int_0^\infty
	\left[h\left(-x+ \frac{\theta}{2} \right) - h\left(x+\frac{\theta}{2} \right)\right]
	\log f\left(x+\frac{\theta}{2} \right) dx
\ee
for any suitably integrable function $h$.
\end{lemma}
\begin{proof}
The left hand side of \eqref{eq:hf} is equal to
\bean
&&\int_{-\infty}^0 h\left(x + \frac{\theta}{2}\right) \log f\left(x + \frac{\theta}{2}\right) dx + 
\int_0^\infty h\left(x + \frac{\theta}{2}\right) \log f\left(x + \frac{\theta}{2}\right) dx \\
&=& \int_{-\infty}^0 h\left(x + \frac{\theta}{2}\right) \log f\left(x + \frac{\theta}{2}\right) dx - 
\int_{-\infty}^0 h\left(-x+\frac{\theta}{2} \right) \log f\left(x+\frac{\theta}{2}\right) dx \\
&=& -\int_0^\infty \left[h\left(-x+ \frac{\theta}{2} \right) - h\left(x+\frac{\theta}{2} \right)\right]
\log f\left(x+\frac{\theta}{2} \right) dx,
\eean
so the proof is complete.
\end{proof}

\bigskip

\begin{theorem} \label{thm:quad_genaral}
Suppose that for a subset $\cH_0\subset\cH$ there exist $\epsilon > 0$ and a function $Q$ such that
$P_{\eta_0} Q^2 < \infty$, $\sup_{|x|<\epsilon}Q(x) < \infty$, and 
\be \label{eq:location_Q}
	\sup_{|\theta|<\epsilon}\sup_{\eta\in\cH_0}\left|\frac{\ell_\eta(x+\theta)
	- \ell_\eta(x)}{\theta}\right| \leq  Q(x)
\ee
for all $x$. 
Furthermore, assume that
\be \label{eq:tech_cond1}
	\int \sup_{|\theta|<\epsilon} \left|\frac{p_{\eta_0}(x+\theta) 
	- p_{\eta_0}(x)}{\theta}\right| \cdot Q(x) \; dx < \infty
\ee
and
\be \label{eq:tech_cond2}
	\int \sup_{\eta\in\cH_0} \left|\frac{\ell_{\eta}(x+\theta) 
	- \ell_{\eta}(x)}{\theta} - \ell^\prime_\eta(x) \right| 
	\cdot |p_{\eta_0}^\prime(x)| \; dx = o(1)
\ee
as $\theta \rightarrow 0$.
Then
\be \label{eq:uql}
	\sup_{\eta \in \cH_0} \Big| P_{\theta_0,\eta_0} 
	\log \frac{p_{\theta, \eta}}{p_{\theta_0, \eta}} + 
	\frac{(\theta-\theta_0)^2}{2}  V_\eta \Big| = o((\theta-\theta_0)^2)
\ee
as $\theta \rightarrow \theta_0$.
\end{theorem}
\begin{proof}
Without loss of generality we assume that $\theta_0=0$ and $\theta > \theta_0$.
Then, applying Lemma \ref{lem:hf} with $f = p_{\theta,\eta}/p_{\theta_0,\eta}$, we get
\bean
	P_{\theta_0,\eta_0} \log \frac{p_{\theta, \eta}}{p_{\theta_0, \eta}} 
	&=& -\int_0^\infty \log \frac{p_{\theta, \eta}}{p_{\theta_0, \eta}} \left(x+\frac{\theta}{2}\right) \cdot 
	\left[ p_{\theta_0,\eta_0}\left(x-\frac{\theta}{2} \right) - 
	p_{\theta_0,\eta_0}\left(x+\frac{\theta}{2} \right) \right] dx 
	\\
	&=& -\int_0^\infty \left[ \ell_{\eta}\Big(x - \frac{\theta}{2}\Big) - 
	\ell_{\eta}\Big(x + \frac{\theta}{2}\Big) \right] \cdot
	\left[ p_{\eta_0}\Big(x-\frac{\theta}{2} \Big) - 
	p_{\eta_0}\Big(x+\frac{\theta}{2} \Big) \right] dx 
	\\
	&=& -\int_0^\infty \Big[\ell_\eta(x+\theta) - \ell_\eta(x)\Big] \cdot
	\Big[p_{\eta_0}(x+\theta) - p_{\eta_0}(x)\Big] dx + o(|\theta|^2)
\eean
as $\theta \rightarrow 0$, where the $o(|\theta|^2)$ term converges to 0 uniformly in $\eta\in\cH_0$.
Since 
$$V_\eta = 2\int_0^\infty \ell_\eta^\prime(x) \; p_{\eta_0}^\prime(x) dx,$$
the left hand side of \eqref{eq:uql} is bounded by
\bean
	&&\theta^2 \times \bigg[ \int_0^\infty \sup_{\eta\in\cH_0}\left|
	\frac{\ell_\eta(x+\theta)-\ell_\eta(x)}{\theta}\right|
	\cdot \left|\frac{p_{\eta_0}(x+\theta)-p_{\eta_0}(x)}{\theta} 
	- p_{\eta_0}^\prime(x) \right| dx
	\\
	&& ~~~~~~~~ + \int_0^\infty \sup_{\eta\in\cH_0} \left|\frac{\ell_\eta(x+\theta)-\ell_\eta(x)}{\theta}
	- \ell_\eta^\prime(x) \right| \cdot |p_{\eta_0}^\prime(x)| \;dx \bigg] + o(|\theta|^2)
	\\
	&& \leq \theta^2 \times \int_0^\infty Q(x) \cdot \left|\frac{p_{\eta_0}(x+\theta)-p_{\eta_0}(x)}{\theta} 
	- p_{\eta_0}^\prime(x) \right| dx + o(|\theta|^2)
	\\
	&& = o(|\theta|^2)
\eean
as $\theta \rightarrow 0$, where the last equality holds by the dominated convergence theorem.
\end{proof}

\bigskip

\begin{Exmp}[Mixtures of normal densities] \label{ex:mixnorm}
Let positive constants $\sigma_1, \sigma_2$ and $M$, with $\sigma_1 < \sigma_2$ be given.
Let $\cH$ be the set of all Borel probability measures $\eta$ supported on 
$[-M,M]\times[\sigma_1,\sigma_2]$ and satisfying $d\eta(z,\sigma) = d\eta(-z,\sigma)$.
Define mixtures of normal densities
$$p_\eta(x) = \int\phi_\sigma(x-z)d\eta(z,\sigma)$$
for every $\eta \in \cH$ and let $\cH_0=\cH$.
Then, by Lemma \ref{lem:quad_bound}, there exists a function $Q$ satisfying $P_{\eta_0} Q^2 < \infty$, 
$\sup_{|x| < \epsilon} Q(x) < \infty$ for every $\epsilon > 0$, and \eqref{eq:location_Q}.
Furthermore, we can choose $Q(x) = O(x)$ as $|x| \rightarrow \infty$.
Conditions \eqref{eq:tech_cond1} and \eqref{eq:tech_cond2} are satisfied by
(v) and (vi) of Lemma and \ref{lem:density}, respectively.
We conclude that the assertion of Theorem \ref{thm:quad_genaral} holds.
\end{Exmp}

\bigskip

A class of mixtures of normal densities is large enough to approximate every twice continuously differentiable density.
If $p$ is a twice continuously differentiable density,
it is well-known (see, for example \citet{ghosal2007posterior}) that 
$h(p, p*\phi_\sigma) = O(\sigma^2)$ as $\sigma \rightarrow 0$,
where $*$ denotes the convolution.
When $p$ is symmetric, it can be similarly approximated by symmetric normal mixtures, but in this case, 
there should be a restriction on mixing distribution to make a density symmetric.
In Example \ref{ex:mixnorm} we impose an assumption that $\eta$ is supported on
$[-M,M] \times [\sigma_1, \sigma_2]$ for some $M>0$ and $0 < \sigma_1 < \sigma_2 < \infty$.
This assumption is required just for technical convenience, and with some additional efforts,
the results could be extended to symmetric mixtures supported on $\bbR \times [0,\infty)$
as in \citet{ghosal2001entropies} and \citet{tokdar2006posterior}.
Furthermore, even when mixtures of normal densities are used for
modeling a smooth density (not necessarily a mixture of normal densities) as in \citet{ghosal2007posterior},
we believe that the results in this thesis could be fully generalized.

In the remainder of this subsection, we prove some elementary properties of mixtures of normal densities
required in Example \ref{ex:mixnorm}.
We follow the notations presented in Example \ref{ex:mixnorm}.
Note first that
\be \label{eq:fractional_bound}
	\inf_z \frac{g(z)}{f(z)} \leq \frac{\int g(z) d\eta(z)}{\int f(z) d\eta(z)}
	~~ \textrm{and} ~~~
	\frac{\int h(z) d\eta(z)}{\int f(z) d\eta(z)} \leq \sup_z \frac{|h(z)|}{f(z)}	
\ee
for any probability measure $\eta$ and integrable real-valued functions $f, g > 0$ and $h$.
These inequalities are very useful to bound the ratio of two mixtures of densities.

\bigskip

\begin{lemma} \label{lem:quad_bound}
There exists a function $x \mapsto Q_{\theta_0}(x)$ with $P_{\theta_0,\eta_0} Q_{\theta_0}^2 < \infty$
and an open neighborhood $U$ of $\theta_0$
such that
$$ \sup_{\eta\in\cH}|\ell_{\theta_1,\eta}(x) - \ell_{\theta_2,\eta}(x)| \leq Q_{\theta_0}(x) \cdot |\theta_1 - \theta_2| $$
for all $\theta_1, \theta_2 \in U$ and $x\in\bbR$.
Furthermore, $Q_{\theta_0}(x) = O(x)$ as $x \rightarrow \infty$.
\end{lemma}
\begin{proof}
By \eqref{eq:fractional_bound}, we have
\bean
	\left| \log \frac{p_{\theta_1, \eta}}{p_{\theta_2, \eta}}(x) \right|
	&=& \left| \log \frac{\int \phi_\sigma(x-\theta_1-z)d\eta(z,\sigma)}{\int \phi_\sigma(x-\theta_2-z)d\eta(z,\sigma)} \right|
	\leq \sup_{z,\sigma} \left| \log\frac{\phi_\sigma(x-\theta_1-z)}{\phi_\sigma(x-\theta_2-z)}\right| \\
	&\leq& \sup_{z,\sigma} \frac{|\theta_1 - \theta_2|}{\sigma^2} \cdot \left| x-\frac{\theta_1+\theta_2}{2}-z\right| \\
	&\leq& |\theta_1-\theta_2| \cdot \Big(|x|+|\theta_1+\theta_2|/2+M\Big) \big/ \sigma_1^2
\eean
for every $\theta_1, \theta_2 \in \Theta$ and this assures
the existence of $Q_{\theta_0}$ and $\sup_{\eta\in\cH}|\score_{\theta_0, \eta}(x)| = O(x)$
as $x \rightarrow \infty$.
\end{proof}

\bigskip

\begin{lemma} \label{lem:density}
For $\theta_0=0$ and $0 < h < 1$, the density function satisfies
\bean
\begin{array}{ll}
(i)~~~p_{\theta_0,\eta_0}^\prime(x) = O\big(x \phi_{\sigma_2}(x-M)\big) & 
(ii)~\sup_{\eta \in \cH} | \ell_{\theta_0, \eta}^\prime(x) | = O(x) \\
(iii)~\sup_{\eta \in \cH}|\ell_{\theta_0,\eta}^{\prime\prime}(x)| = O\big(x^2\big) &
(iv)~\sup_{\eta \in \cH} | \ell_{\theta_0, \eta}^{\prime\prime\prime}(x) | = O(x^3) \\
(v)~~~|p_{\theta_0,\eta_0}^\prime(x+h) - p_{\theta_0,\eta_0}^\prime(x)|
	= h O\big(x^2 \phi_{\sigma_2}(x-M)\big)~~ & \\
(vi)~~\sup_{\eta \in \cH} | \ell_{\theta_0, \eta}^\prime(x+h) - \ell_{\theta_0, \eta}^\prime(x)| = h O(x^2 ) &
\end{array}
\eean
as $x \rightarrow \infty$.
\end{lemma}
\begin{proof}
Note that
\bean
	\phi_\sigma^\prime(x) &=& -\frac{x}{\sigma^2} \phi_\sigma(x) \\
	\phi_\sigma^{\prime\prime}(x) &=& \left( \frac{x^2}{\sigma^4} - \frac{1}{\sigma^2} \right) \phi_\sigma(x) \\
	\phi_\sigma^{\prime\prime\prime}(x) &=& \left( \frac{3x}{\sigma^4} - \frac{x^3}{\sigma^6} \right) \phi_\sigma(x)
\eean
and $\phi_\sigma(x) \leq \sigma_1^{-1} \sigma_2 \phi_{\sigma_2}(x)$ for every $\sigma_1 \leq \sigma \leq \sigma_2$.
First, (i) holds by
\bean
	p_{\theta_0,\eta_0}^\prime(x) &=& \int \phi_\sigma^\prime(x-z) d\eta_0(z,\sigma) \\
	&=& \int \frac{z \phi_\sigma(x-z)}{\sigma^2} d\eta_0(z,\sigma) 
		- x\int \frac{\phi_\sigma(x-z)}{\sigma^2} d\eta_0(z,\sigma) \\
	&\leq& \int \frac{\sigma_2 M \phi_{\sigma_2}(x-z)}{\sigma_1^3} d\eta_0(z,\sigma) 
		- x\int \frac{\sigma_2 \phi_{\sigma_2}(x-z)}{\sigma_1^3} d\eta_0(z,\sigma) \\
	&\leq& O\big(\phi_{\sigma_2}(x-M)\big) + O\big(x \phi_{\sigma_2}(x-M)\big) 
	= O\big(x \phi_{\sigma_2}(x-M)\big)
\eean
as $x \rightarrow \infty$.
Also, since we have
\bean
	\ell_{\theta_0,\eta}^\prime(x) &=& \frac{p_{\theta_0,\eta}^\prime}{p_{\theta_0,\eta}}(x)
	= \frac{\int \phi_\sigma^\prime(x-z)d\eta(z,\sigma)}{\int\phi_\sigma(x-z)d\eta(z,\sigma)}
	\leq \sup_{z,\sigma} \left|\frac{\phi_\sigma^\prime}{\phi_\sigma}(x-z)\right| \leq (x+M)/\sigma_1^2, \\
	\ell_{\theta_0,\eta}^{\prime\prime}(x)
	&=& \frac{p_{\theta_0,\eta} p_{\theta_0,\eta}^{\prime\prime} - 
	(p_{\theta_0,\eta}^\prime)^2}{(p_{\theta_0,\eta})^2}(x)
	\leq \sup_{z,\sigma} \left|\frac{\phi_\sigma^{\prime\prime}}{\phi_\sigma}(x-z)\right|
	+ \sup_{z,\sigma} \left|\frac{\phi_\sigma^\prime}{\phi_\sigma}(x-z)\right|^2 \\
	&\leq& \frac{1}{\sigma_1^2} + \frac{2(x+M)^2}{\sigma_1^4}, \\
	\ell_{\theta_0,\eta}^{\prime\prime\prime}(x) &=& 
	\frac{p_{\theta_0,\eta}^{\prime\prime\prime}}{p_{\theta_0,\eta}}(x) 
	- 3\frac{p_{\theta_0,\eta}^\prime p_{\theta_0,\eta}^{\prime\prime}}{(p_{\theta_0,\eta})^2}(x)
	+ 2\left(\frac{p_{\theta_0,\eta}^\prime}{p_{\theta_0,\eta}}\right)^3(x)
	\leq \frac{6(x+M)}{\sigma_1^4} + \frac{6(x+M)^3}{\sigma_1^6}
\eean
for $x > 0$, (ii), (iii) and (iv) are proved.
Next, (v) can be proved by 
\bean
	p_{\theta_0,\eta_0}^\prime(x+h) - p_{\theta_0,\eta_0}^\prime(x) &=& 
	h \int_0^1 p_{\theta_0,\eta_0}^{\prime\prime}(x+th) dt \\
	&=& h \int_0^1 \int \phi_\sigma^{\prime\prime}(x+th-z) d\eta_0(z,\sigma) dt \\
	&\leq& h \sigma_1^{-5} \sigma_2 (x+1+M)^2 \phi_{\sigma_2}(x-M) 
\eean
for $0 < h < 1$ and large enough $x$.
In the same way, combining
\bean
	\ell_{\theta_0,\eta}^\prime(x+h) - \ell_{\theta_0,\eta}^\prime(x) &=& 
	h \int_0^1 \ell_{\theta_0,\eta}^{\prime\prime}(x+th) dt	
\eean
and (iii), we can prove (vi).
\end{proof}

\bigskip

\subsection{Multivariate symmetric densities}

For $\bX = (X_1, \ldots, X_m)^T$, $\btheta = (\theta_1, \ldots, \theta_m)^T$
and $\bepsilon=(\epsilon_1, \ldots, \epsilon_m)^T$, consider a multivariate location problem
$$\bX = \btheta + \bepsilon,$$
where the distribution of $\bepsilon$ is parametrized by $\eta\in\cH$ for some infinite dimensional $\cH$.
Write the density of error distribution by $p_\eta$ and let $p_{\btheta,\eta}(\bx) = p_\eta(\bx-\btheta)$.
A density $p_\eta$ is assumed to be symmetric about the origin in the sense that 
$p_\eta(\bx) = p_\eta(-\bx)$, and continuously differentiable for every $\eta \in \cH$.
Fix $\btheta_0 \in \bbR^m$ and $\eta_0\in\cH$ which can be considered as the true parameter.
For a set $A \subset \bbR^m$ let $A^- = \{\by\in\bbR^m: -\by\in A\}$.
The following lemma and theorem are the multivariate correspondences of Lemma \ref{lem:hf}
and Theorem \ref{thm:quad_genaral}.

\bigskip

\begin{lemma} \label{lem:imp}
Assume that a positive function $f: \bbR^m \rightarrow (0,\infty)$
satisfies $f(\bx+\btheta/2) \cdot f(-\bx + \btheta/2) = 1$.
Then, for any suitably integrable function $h$,
\be \label{eq:hf_re} 
	\int h(\bx) \log f(\bx) d\bx = 
	- \int_A \left[h\left(-\bx+ \frac{\btheta}{2} \right) 
	- h\left(\bx+\frac{\btheta}{2} \right)\right]
	\log f\left(\bx+\frac{\btheta}{2} \right) d\bx,
\ee
where $A$ is any subset of $\bbR^m$ such that $A$ and $A^-$ are disjoint, and
$(A \cup A^-)^c$ has Lebesgue measure zero.
\end{lemma}
\begin{proof}
The left hand side of \eqref{eq:hf_re} is equal to
\bean
	&&\int_A h\left(\bx + \frac{\btheta}{2}\right) \log f\left(\bx + \frac{\btheta}{2}\right) d\bx + 
	\int_{A^-} h\left(\bx + \frac{\btheta}{2}\right) \log f\left(\bx + \frac{\btheta}{2}\right) d\bx \\
	&=& \int_A h\left(\bx + \frac{\btheta}{2}\right) \log f\left(\bx + \frac{\btheta}{2}\right) d\bx - 
	\int_A h\left(-\bx+\frac{\btheta}{2} \right) \log f\left(\bx+\frac{\btheta}{2}\right) d\bx \\
	&=& -\int_A \left[h\left(-\bx+ \frac{\btheta}{2} \right) - h\left(\bx+\frac{\btheta}{2} \right)\right]
	\log f\left(\bx+\frac{\btheta}{2} \right) d\bx
\eean
which is the desired result.
\end{proof}

\bigskip

\begin{theorem} \label{thm:uniform_quadratic_lme}
Suppose that for a subset $\cH_0\subset\cH$ there exist $\epsilon >0$ and a function $Q$
such that $P_{\eta_0}Q^2 < \infty$, $\sup_{|\bx|<\epsilon} Q(\bx) < \infty$, and
\be \label{eq:vector_Q}
	\sup_{|\theta|<\epsilon}\sup_{\eta\in\cH_0} \frac{|\ell_\eta(\bx+\btheta)
	- \ell_\eta(\bx)|}{|\btheta|} \leq  Q(\bx)
\ee
for all $\bx$. 
Furthermore, assume that
\be \label{eq:vector_tech1}
	\int \sup_{|\btheta|<\epsilon} \frac{|p_{\eta_0}(\bx+\btheta) 
	- p_{\eta_0}(\bx)|}{|\btheta|} \cdot Q(\bx) \; d\bx < \infty
\ee
and
\be\label{eq:vector_tech2}
	\int \sup_{\eta\in\cH_0} \Big|\ell_{\eta}(\bx+\btheta) 
	- \ell_{\eta}(\bx) - \btheta^T \nabla\ell_\eta(\bx)\Big|
	\cdot |\nabla p_{\eta_0}(\bx)| \; d\bx = o(|\btheta|)
\ee
as $\btheta$ converges to the zero vector.
Then
\be \label{eq:mul_quad}
	\sup_{\eta\in\cH_0} \left|
	P_{\btheta_0,\eta_0} \log \frac{p_{\btheta,\eta}}{p_{\btheta_0,\eta}} 
	+ \frac{1}{2} (\btheta-\btheta_0)^T V_{\eta} (\btheta-\btheta_0) \right|
	= o(|\btheta-\btheta_0|^2)
\ee
as $\btheta \rightarrow \btheta_0$.
\end{theorem}
\begin{proof}
Without loss of generality, we may assume that $\btheta_0$ is the zero vector.
Define $A = \{\bx=(x_1,\ldots,x_m): x_1 > 0\}$, then $A$ satisfies the condition of Lemma \ref{lem:imp}.
Applying Lemma \ref{lem:imp} with $f = p_{\btheta,\eta}/p_{\btheta_0,\eta}$
\bean
	&P_{\btheta_0,\eta_0}& \log \frac{p_{\btheta, \eta}}{p_{\btheta_0,\eta}} 
	= -\int_A \log \frac{p_{\btheta, \eta}}{p_{\btheta_0,\eta}} 
	\Big(\bx+\frac{\btheta}{2}\Big) \cdot 
	\left[ p_{\btheta_0,\eta_0} \Big(\bx-\frac{\btheta}{2} \Big) - 
	p_{\btheta_0,\eta_0} \Big(\bx+\frac{\btheta}{2} \Big) \right] d\bx
	\\
	&=& -\int_A \left[ \ell_{\eta}\Big(\bx - \frac{\btheta}{2}\Big) - 
	\ell_{\eta}\Big(\bx + \frac{\btheta}{2}\Big) \right]\cdot 
	\left[ p_{\eta_0} \Big(\bx-\frac{\btheta}{2} \Big) - 
	p_{\eta_0}\Big(\bx+\frac{\btheta}{2} \Big) \right] d\bx
	\\
	&=& -\int_A \Big( \ell_\eta(\bx+\btheta) - \ell_\eta(\bx)\Big) \cdot 
	\Big(p_{\eta_0}(\bx+\btheta) - p_{\eta_0}(\bx)\Big) d\bx
	+ o(|\btheta|^2)
\eean
as $\btheta \rightarrow \btheta_0$, where the $o(|\btheta|^2)$ term converges to the zero vector
uniformly in $\eta\in\cH_0$.
Since 
$$V_\eta = 2\int_A \nabla\ell_\eta(\bx) \nabla p_{\eta_0}^T(\bx) d\bx,$$
the left hand side of \eqref{eq:mul_quad} is bounded by
\bean
	&& \int_0^\infty \sup_{\eta\in\cH_0}\big|
	\ell_\eta(\bx+\btheta)-\ell_\eta(\bx) \big| \cdot
	\Big| p_{\eta_0}(\bx+\btheta)-p_{\eta_0}(\bx) 
	- \btheta^T \nabla p_{\eta_0}(x) \Big| d\bx
	\\
	&& ~~~~~~~~ + \int_0^\infty \sup_{\eta\in\cH_0} 
	\big| \ell_\eta(\bx+\btheta)-\ell_\eta(\bx) 
	- \btheta^T \nabla \ell_\eta(\bx) \big| \cdot |\btheta^T \nabla p_{\eta_0}(\bx)| \;d\bx + o(|\btheta|^2)
	\\
	&\leq& |\btheta|^2 \times \int_0^\infty Q(x) \cdot 
	\left|\frac{p_{\eta_0}(\bx+\btheta)-p_{\eta_0}(\bx)}{|\btheta|}
	- \frac{\btheta^T}{|\btheta|} p_{\eta_0}(\bx) \right| d\bx + o(|\btheta|^2)
	\\
	&=& o(|\btheta|^2)
\eean
as $\btheta$ converges to zero vector, where the last equality holds by the dominated convergence theorem.
\end{proof}

\bigskip

An important application of Theorem \ref{thm:uniform_quadratic_lme} is to the random effects models
in which $\bepsilon$ is the sum of the random effects and errors.
Therefore the distribution of random effects, as well as the error distribution, should
be assumed to be symmetric about the origin.
In Section \ref{sec:lme}, two distributions $F$ and $G$ in Example \ref{ex:mixmnorm} will be considered as
the distributions of errors and of random effects, respectively.

\bigskip

\begin{Exmp} \label{ex:mixmnorm}
For $0 < \sigma_1 < \sigma_2 < \infty$ and $M > 0$,
we define $\cF$ by the set of all Borel probability measures $F$ 
supported on $[-M,M] \times [\sigma_1,\sigma_2]$ and satisfying $dF(z,\sigma) = dF(-z,\sigma)$.
Let $\cG$ be the set of all Borel probability measures on $[-M,M]$ with $dG(b) = dG(-b)$
and define $\cH = \cH_0 = \cF\times\cG$.
For $\bx = (x_1, \ldots, x_m)^T$, $\btheta = (\theta_1, \ldots, \theta_m)^T$,
and $\eta = (F,G) \in \cH$, let
$$
	p_{\eta}(\bx) = \int \prod_{j=1}^m p_{F} (x_j-b) dG(b),
$$
where $p_{F}(x) = \int \phi_\sigma(x-z) dF(z)$.
Then $p_{\eta}(\bx) = p_{\eta}(-\bx)$ for every $\bx\in\bbR^m$ and $\eta \in \cH$.
Also,
\bean
	|\ell_\eta(\bx+\btheta) - \ell_\eta(\bx)| &=&
	\left|\log \frac{\int \prod_j p_F(x_j+\theta_j-b) dG(b)}{\int \prod_j p_F(x_j-b) dG(b)}\right|
	\\
	&\leq& \sup_{b \in [-M,M]} \left| \sum_{j=1}^m\log\frac{p_F(x_j+\theta_j-b)}{p_F(x_j-b)} \right|
	\\
	&=& O(|\bx|)
\eean
as $|\bx| \rightarrow \infty$.
Therefore there exists $Q$ such that $P_{\eta_0} Q^2 < \infty$, 
$\sup_{|\bx| < \epsilon} Q(\bx) < \infty$ for every $\epsilon > 0$ and satisfies \eqref{eq:vector_Q}.
The fact that
$$
	p_{\eta_0}(\bx+\btheta) - p_{\eta_0}(\bx) = \btheta^T \int_0^1 \nabla p_{\eta_0} (\bx + t\btheta) dt
$$
and (ii) of Lemma \ref{lem:den_property} yield \eqref{eq:vector_tech1}.
In a similar way, \eqref{eq:grad} and (i) of Lemma \ref{lem:den_property} proves \eqref{eq:vector_tech2}.
Therefore the assertion of Theorem \ref{thm:uniform_quadratic_lme} holds.
\end{Exmp}

\bigskip

\begin{lemma} \label{lem:den_property}
With the notation presented in Example \ref{ex:mixmnorm}, we have
\ben
	\item[(i)] $\sup_{0 < |\bh| < 1} \sup_{\eta\in\cH} |\nabla \ell_\eta(\by+\bh)
	-\nabla \ell_\eta(\by)| / |\bh| = O(|\by|^2)$
	\item[(ii)] $\sup_{0< |\bh| < 1} |\nabla p_{\eta_0}(\by+\bh) -
	\nabla p_{\eta_0}(\by)| / |\bh| 
	= O\big((|\by|+1)^2 e^{m|\by|/\sigma_1^2} p_{\eta_0}(\by)\big)$
\een
as $|\by| \rightarrow \infty$.
\end{lemma}
\begin{proof}
Note first that
\bea
	\sup_{\eta\in\cH} \frac{|\nabla_j p_\eta|}{p_\eta}(\by) &=& O(|\by|) \label{eq:grad}\\
	\sup_{\eta\in\cH} \frac{|\nabla_{jk} p_\eta|}{p_\eta}(\by) &=& O(|\by|^2) \label{eq:hessian}\\
	\sup_{|\bh| < 1} \sup_{\eta\in\cH} \frac{p_\eta(\by+\bh)}{p_\eta(\by)}
	&=& O\big(\exp(m\cdot |\by|/\sigma_1^2)\big) \label{eq:h_ratio}
\eea
as $|\by| \rightarrow \infty$.
Also, for any $\by$ and $\bh$
and
\bea
	\nabla p_{\eta}(\by+\bh) -\nabla p_{\eta}(\by) 
	&=& \int_0^1 \nabla^2 p_{\eta}(\by+t\bh) dt \cdot \bh \label{eq:taylor_grad_p}\\
	\nabla \ell_{\eta}(\by+\bh) -\nabla \ell_{\eta}(\by) 
	&=& \int_0^1 \nabla^2 \ell_{\eta}(\by+t\bh) dt \cdot \bh \label{eq:taylor_grad_l}
\eea
by Taylor's expansion.
Therefore, \eqref{eq:taylor_grad_p} combining with \eqref{eq:hessian} and \eqref{eq:h_ratio} yields (ii).
Also, (i) is satisfied by \eqref{eq:grad}, \eqref{eq:hessian} and the identity
$$
	\nabla_{jk} \ell_{\eta}(\by) = 
	\frac{p_{\eta} \cdot \nabla_{jk}p_{\eta} - \nabla_j p_{\eta} \cdot \nabla_k p_{\eta}}
	{p_{\eta}^2}(\by)
$$
for all $j$ and $k$.
\end{proof}

\bigskip

\section{Examples}

In this section, we apply the general semiparametric BvM theorem
for specific models and priors.
We consider three models which have symmetric errors: the location, linear regression
and random intercept models.
In each model, error densities are modeled by symmetric mixtures of normal densities.
Although the location model is contained in the regression model,
we begin the proof with the location model 
because the essentials of the proofs are similar for the other models.

\subsection{Location model}

Consider the symmetric location model, where
\iid~real-valued observations $X_1, \ldots, X_n$ are modeled as
\be \label{eq:err_den}
	X_i = \theta + \epsilon_i, ~~~ i=1, \ldots, n,
\ee
and $\epsilon_i$ follows a mixture of normal densities $p_\eta(x) = \int \phi_\sigma(x-z) d\eta(z,\sigma)$
for a mixing distribution $\eta$ that is symmetric in the sense $d\eta(z,\sigma) = d\eta(-z,\sigma)$.
The model can be parameterized by the location parameter $\theta \in \Theta$ and
the mixing distribution $\eta \in \cH$, where $\Theta$ is an open subset of $\bbR$ and
$\cH$ is defined as in Example \ref{ex:mixnorm}.
Assume that the true distribution $P_0$ which generates the observations 
is contained in the model, that is $P_0 = P_{\theta_0,\eta_0}$ for some $(\theta_0,\eta_0) \in \Theta\times\cH$.
Let $\score_{\theta, \eta}(x) = \partial\ell_{\theta,\eta}(x)/\partial\theta$
be the ordinary score function
and $I_{\theta,\eta} = P_{\theta,\eta}[\score_{\theta,\eta}^{\;2}]$ be the Fisher information.
Then it is obvious that $\score_{\theta, \eta}(x) = -\ell^\prime_{\theta, \eta}(x)$ 
and $I_{\theta,\eta} > 0$ for all $\theta \in \Theta$ and $\eta\in\cH$.
Let $V_\eta = P_{\theta_0,\eta_0}[\score_{\theta_0, \eta} \score_{\theta_0,\eta_0}]$
and $d_H(\eta_1, \eta_2) = h(P_{\theta_0,\eta_1}, P_{\theta_0, \eta_2})$.
A weak neighborhood of $\eta_0$ is defined by
\be \label{eq:weak_nbr}
	\left\{ \eta \in \cH : \bigcap_{i=1}^m \Big| \int f_i(z, \sigma) d\eta(z,\sigma)
	- \int f_i(z,\sigma) d\eta_0(z,\sigma) \Big| < \epsilon\right\}
\ee
for $\epsilon > 0$ and any finite collection $f_1, \ldots, f_m$ of bounded continuous functions
on $[-M, M] \times [\sigma_1, \sigma_2]$.

\bigskip

\begin{theorem}\label{thm:main_location}
Assume that $P_0$ is contained in the model $\scrP = \{P_{\theta,\eta}: \theta \in \Theta, \eta \in \cH\}$
which is endowed with the product prior $\Pi=\Pi_\Theta \times \Pi_\cH$,
where $\Pi_\Theta$ is thick at $\theta_0$ and $\Pi_\cH(U) > 0$ 
for every weak neighborhood $U$ of $\eta_0$.
Then,
\bean
	\sup_B \left|\Pi(\sqrt{n}(\theta-\theta_0) \in B | X_1, \ldots, X_n) - N_{\Delta_n, I^{-1}_{\theta_0,\eta_0}}(B)
	\right| \rightarrow 0,
\eean
in $P_0$-probability, where 
$$ 
  \Delta_n = \frac{1}{\sqrt{n}} \sum_{i=1}^n I_{\theta_0, \eta_0}^{-1} \score_{\theta_0, \eta_0}(X_i).
$$
\end{theorem}
\begin{proof}
By Theorems \ref{thm:BvM_general}, \ref{thm:ILAN_general}, and \ref{thm:sqrt_n_general},
it is sufficient to show that \eqref{eq:zero_score_condition}--\eqref{eq:conv_perturb_general},
and \eqref{eq:consist} hold for some $\score_{\theta,\eta}^{(n)}$, $(\cH_n)$ and $(V_{n,\eta})$.
First, the metric $d_H$ satisfies \eqref{eq:d_H_def}.
Also, \eqref{eq:h_unif_libschitz} holds by Lemma \ref{lem:Hellinger_ineq} and
$$
	h^2 \Big( N(z_1, \sigma^2), N(z_2, \sigma^2) \Big)
	= 1 - \exp\left(- \frac{|z_1-z_2|^2}{8\sigma^2 n} \right)
	\leq \frac{|z_1-z_2|^2}{8\sigma^2 n} \leq \frac{|z_1-z_2|^2}{8\sigma_1^2 n}.
$$
Condition \eqref{eq:quad_bdd_general} is satisfied by Lemma \ref{lem:quad_bound}.
Now, the assertions of Theorems \ref{thm:consist_perturb_general} and \ref{thm:consist_general} hold
by Lemmas \ref{lem:entropy_bound}, \ref{lem:KL_bound}, \ref{lem:Lipschitz} and
the fact that the median $\hat\theta_n$ of $X_1, \ldots, X_n$ satisfies \eqref{eq:sqrtn_estimator}.
This implies that there exists a sequence $\epsilon_n \downarrow 0$ such that
the sequence $(\cH_n)$ defined by $\cH_n = \{\eta\in\cH: d_H(\eta,\eta_0) < \epsilon_n\}$
satisfies \eqref{eq:conv_perturb_general} and \eqref{eq:consist}.

Let $\score_{\theta,\eta}^{(n)} = \sum_{i=1}^n \score_{\theta,\eta}(X_i)$ be the ordinary score function
and $V_{n,\eta} = V_\eta$.
Then, \eqref{eq:zero_score_condition} holds by Lemma \ref{lem:tht_star},
\eqref{eq:V_conti_condition} holds by Lemma \ref{lem:L2_conv}, \eqref{eq:V_positive} is trivial,
and \eqref{eq:quadratic_condition} holds by Example \ref{ex:mixnorm}.
Lemma \ref{lem:Donsker} directly implies \eqref{eq:donsker_condition}.
If we define an index set $T=\{\eta\in\cH: d_H(\eta,\eta_0) < \epsilon)$
for sufficiently small $\epsilon > 0$ and a semimetric
\be \label{eq:L2_metric}
	d_2(\eta_1, \eta_2) = \Big(P_0 (\score_{\theta_0, \eta_1} - \score_{\theta_0, \eta_2})^2\Big)^{1/2},
\ee
then the Donsker property also implies that $T$ is totally bounded with respect to $d_2$ and
the stochastic process $\bbG_n \score_{\theta_0, \eta}$ indexed by $\eta \in T$
is asymptotically uniformly $d_2$-equicontinuous in probability
\footnote{See problem 2 in page 93 of \citet{van1996weak}.}.
As a result, \eqref{eq:score_conti_condition} holds
because $d_H(\eta,\eta_0) \rightarrow 0$ implies $d_2(\eta,\eta_0) \rightarrow 0$.

For \eqref{eq:lr_apprx_small_condition}, by the Tayler expansion implies
\bean
	&&\ell^{(n)}_{\theta_n(h),\eta}(X^{(n)}) - \ell^{(n)}_{\theta_0,\eta}(X^{(n)})
	- \frac{h}{\sqrt{n}} \score_{\theta_0,\eta}^{(n)}(X^{(n)})
	\\
	&=& \frac{h}{\sqrt{n}} \int_0^1 \sum_{i=1}^n \Big( \score_{\theta_0+th/\sqrt{n},\eta}(X_i) 
	- \score_{\theta_0,\eta}(X_i) \Big) dt
\eean
and by the Jensen's inequality and Fubini's theorem,
the variance of the right hand side is bounded by
\bean
	\frac{h^2}{n} \int_0^1 \sum_{i=1}^n P_0 \Big( \score_{\theta_0+th/\sqrt{n},\eta}(X_i)
	- \score_{\theta_0,\eta}(X_i) \Big)^2 dt = O(h^4/n)
\eean
by (vi) of Lemma \ref{lem:density}.
Therefore, for each $(h, \eta)$ the term in the left hand side of \eqref{eq:lr_apprx_small_condition}
converges in probability to 0, and it converges uniformly by Donsker's theorem.
Finally, write
\bean
	\ell^{(n)}_{\theta_n(h),\eta}(X^{(n)}) - \ell^{(n)}_{\theta_0,\eta}(X^{(n)})
	= \frac{h}{\sqrt{n}} \int_0^1 \sum_{i=1}^n \score_{\theta_0 + th/\sqrt{n}, \eta} (X_i) dt
\eean
and apply the Donsker's theorem to prove \eqref{eq:lr_apprx_large_condition}.
\end{proof}

\bigskip

In Theorem \ref{thm:main_location}, the only nontrivial condition is in the prior $\Pi_\cH$.
However, we note that the condition is very weak compared to the usual conditions
\footnote{Prior positivity for every Kullback-Leibler type neighborhood}.
that are typically required for posterior consistency or convergence rate.
We provide an example which is widely used for nonparametric symmetric density estimation problems.

Recall that the support of a positive measure $\nu$ on a topological space is defined by
the complement of the largest open set of $\nu$-measure zero and denoted by ${\rm supp}(\nu)$.
If $\nu$ follows a Dirichlet process with base measure $\alpha$ over the Euclidean spaces,
then it is well-known \footnote{See, for example, \citet{ghosh2003bayesian}.} that
the support of $\nu$ with respect to the weak topology is given by
\bea
	{\rm supp}(\nu) &=& \{ P: \nu(U) > 0 ~ \textrm{for all weakly open set}~ U \ni P\}\nonumber\\
	&=& \{P: {\rm supp}(P) \subset {\rm supp}(\alpha)\}. \label{eq:DP_support}
\eea
Using this fact, we can easily check that under the Dirichlet mixture priors of normal densities
the BvM theorem holds.

\bigskip

\begin{example}
For the prior $\Pi_\cH$ for $\eta$, consider a symmetrized Dirichlet process prior defined by
$$d\eta(z,\sigma) = \frac{1}{2}\Big( d\tilde\eta(z,\sigma) + d\tilde\eta(-z,\sigma)\Big),$$
where $\tilde\eta$ follows a Dirichlet process with parameter $(\alpha, P)$ and
$P$ is supported on $[0,M]\times[\sigma_1,\sigma_2]$.
Since $f(z,\sigma) = \frac{1}{2}\big(f(z,\sigma) + f(-z,\sigma)\big) + 
\frac{1}{2}\big(f(z,\sigma) - f(-z,\sigma)\big)$ and
$\int \big(f(z,\sigma) - f(-z,\sigma) \big) d\eta(z,\sigma) = 0$
for all $f$ and $\eta$ with $d\eta(z,\sigma) = d\eta(-z,\sigma)$,
it is sufficient to consider weak neighborhoods of type \eqref{eq:weak_nbr} with bounded continuous $f_i$
satisfying $f_i(z,\sigma) = f_i(-z,\sigma)$.
If $f(z,\sigma) = f(-z,\sigma)$ and $d\eta(z,\sigma) = d\eta(-z,\sigma)$ then 
$\int f(z, \sigma) d\eta(z,\sigma) = \int f(z, \sigma) d\eta^+(z,\sigma)$, where
$\eta^+ = 2\times\eta \big|_{(0,M] \times [\sigma_1,\sigma_2]} + \eta(\{0\})\delta_0$ and
$\delta_0$ is the Dirac measure at zero.
Therefore, $U$ contains a weak neighborhood of $\eta_0$ in $\cH$ if and only if
$U^+$ contains a weak neighborhood of $\eta_0^+$ in $\cH^+$,
where $U^+ = \{\eta^+: \eta\in U\}$ and $\cH^+ = \{\eta^+: \eta\in\cH\}$.
We conclude that if the support of $P$ contains the support of $\eta_0$,
then $\Pi_\cH$ satisfies the condition in Theorem \ref{thm:main_location} by \eqref{eq:DP_support}
so the BvM theorem holds. \qed
\end{example}

\bigskip

In the following, we prove some technical lemmas for proving Theorem \ref{thm:main_location}.

\bigskip

\begin{lemma} \label{lem:L2_conv}
With the notation of Theorem \ref{thm:main_location}, we have
\bea
	\max\left\{-P_{\theta_0,\eta_0} \left(\log \frac{p_{\theta_0,\eta}}{p_{\theta_0,\eta_0}}\right), ~
	P_{\theta_0,\eta_0} \left(\log \frac{p_{\theta_0,\eta}}{p_{\theta_0,\eta_0}}\right)^2 \right\} 
	&\rightarrow& ~ 0 \label{eq:L2_conv_ll}\\
	P_{\theta_0,\eta_0}(\score_{\theta_0, \eta} - \score_{\theta_0,\eta_0})^2 =
	P_{\theta_0,\eta_0}(\ell^\prime_{\theta_0, \eta} - \ell^\prime_{\theta_0,\eta_0})^2 
	&\rightarrow& 0 \label{eq:L2_conv_score}
\eea
as $d_H(\eta, \eta_0) \rightarrow 0$.
That is, $\ell_{\theta_0, \eta}$ and $\score_{\theta_0, \eta}$ converge in $L_2$ to 
$\ell_{\theta_0,\eta_0}$ and $\score_{\theta_0,\eta_0}$, respectively,
as $\eta \rightarrow \eta_0$ with respect to $d_H$.
\end{lemma}
\begin{proof}
Without loss of generality, we may assume that $\theta_0=0$.
Then, for \eqref{eq:L2_conv_ll}, by Theorem 5 in \citet{wong1995probability},
it is sufficient to show that
\be 
	\sup_{d_H(\eta, \eta_0) < \epsilon} \int \left(\frac{p_{\theta_0,\eta_0}}{p_{\theta_0, \eta}}
	\right)^\delta dP_{\theta_0,\eta_0} ~<~ \infty \label{eq:integral_bound}
\ee
for some $\delta \in (0,1]$ and $\epsilon > 0$.
Note that
\bean
	p_{\theta_0,\eta_0}(x) &\leq& (2\pi\sigma_1^2)^{-1/2} \\
	\inf_{\eta \in \cH} p_{\theta_0, \eta}(x) &=&
	\inf_{\eta \in \cH} \int \phi_\sigma(x-z) d\eta(z,\sigma) \geq \inf_{z,\sigma} \phi_\sigma(x-z)
	\geq \sigma_1\sigma_2^{-1}\phi_{\sigma_1}(2M)
\eean
for $x \in [0, M]$.
Also, we have
\bean
	p_{\theta_0,\eta_0}(x) &=& \int \phi_\sigma(x-z)d\eta_0(z,\sigma) \leq \sup_{z,\sigma} \phi_\sigma(x-z) 
	\leq \sigma_1^{-1}\sigma_2\phi_{\sigma_2}(x-M)\\
	\inf_{\eta \in \cH} p_{\theta_0, \eta} (x) &=& \inf_{\eta \in \cH} \int \phi_\sigma(x-z)d\eta(z,\sigma)
	\geq \sigma_1\sigma_2^{-1}\phi_{\sigma_1}(x+M)
\eean
for $x \geq M$.
Therefore, for sufficiently small $\delta \in (0,1]$,
\bean
	\sup_{\eta \in \cH} \int \left(\frac{p_{\theta_0,\eta_0}}{p_{\theta_0, \eta}}\right)^\delta dP_{\theta_0,\eta_0}
	&\leq& 2\int_0^M \left(\frac{\sigma_2(2\pi)^{-1/2}}{\sigma_1^2\phi(2M)}\right)^\delta dP_{\theta_0,\eta_0} 
	\\
	&& ~ + ~ 2\left(\frac{\sigma_2}{\sigma_1}\right)^{2(1+\delta)} \int_M^\infty 
	\frac{\big(\phi_{\sigma_2}(x-M)\big)^{1+\delta}}{\big(\phi_{\sigma_1}(x+M)\big)^\delta} dx \\ &<& \infty
\eean
and \eqref{eq:integral_bound} is satisfied for every $\epsilon > 0$.

For \eqref{eq:L2_conv_score}, note that
\bean
	&& \lim_{h \downarrow 0} \sup_{\eta \in \cH} \left| \int \left[
	\frac{\ell_{\theta_0,\eta}(x+h) - \ell_{\theta_0,\eta}(x)}{h} - 
	\ell^\prime_{\theta_0,\eta_0}(x) \right]^2  - 
	\Big(\ell_{\theta_0,\eta}^\prime(x) - \ell^\prime_{\theta_0,\eta_0}(x)\Big)^2 dP_{\theta_0,\eta_0}(x) \right|
	\\
	&=& \lim_{h \downarrow 0} \sup_{\eta \in \cH} \bigg| 
	\int \int_0^1 \Big[\ell^\prime_{\theta_0, \eta}(x+th) - 
	\ell^\prime_{\theta_0, \eta}(x)\Big] dt
	\\
	&& ~~~~~~~~~~~~~~~ \times \left[\frac{\ell_{\theta_0,\eta}(x+h) - \ell_{\theta_0,\eta}(x)}{h} 
	+ \ell^\prime_{\theta_0,\eta}(x) -
	2\ell^\prime_{\theta_0,\eta_0}(x) \right] dP_{\theta_0,\eta_0}(x) \bigg|
	\\
	&\leq& \lim_{h \downarrow 0} \sup_{\eta \in \cH} h \bigg| \int O(x^2) \times
	\left[\frac{\ell_{\theta_0,\eta}(x+h) - \ell_{\theta_0,\eta}(x)}{h} + 
	\ell^\prime_{\theta_0,\eta}(x) - 2\ell^\prime_{\theta_0,\eta_0}(x) \right] dP_{\theta_0,\eta_0}(x) \bigg|
	\\
	&=& 0,
\eean
where the inequality holds by (vi) of Lemma \ref{lem:density}.
This enables the interchange of two limits, by Moore-Osgood theorem, in the following equality:
\bean
	&& \lim_{d_H(\eta, \eta_0) \rightarrow 0} P_{\theta_0,\eta_0}(\ell^\prime_{\theta_0, \eta} 
	- \ell^\prime_{\theta_0,\eta_0})^2
	\\
	&=& \lim_{d_H(\eta, \eta_0) \rightarrow 0} \lim_{h \downarrow 0} \int \left[
	\frac{\ell_{\theta_0,\eta}(x+h) - \ell_{\theta_0,\eta}(x)}{h} - 
	\ell^\prime_{\theta_0,\eta_0}(x) \right]^2 dP_{\theta_0,\eta_0}(x)
	\\
	&=& \lim_{h \downarrow 0} \lim_{d_H(\eta, \eta_0) \rightarrow 0} \int \left[
	\frac{\ell_{\theta_0,\eta}(x+h) - \ell_{\theta_0,\eta}(x)}{h} - 
	\ell^\prime_{\theta_0,\eta_0}(x) \right]^2 dP_{\theta_0,\eta_0}(x)
	\\
	&=& 0,
\eean
where the first and third equalities hold by the dominated convergence theorem and $L_2$ 
convergence of $\ell_{\theta_0, \eta}$ to $\ell_{\theta_0,\eta_0}$.
\end{proof}

\bigskip

\begin{lemma} \label{lem:Donsker}
The class
\bean
	\cF &=& \Big\{\score_{\theta, \eta}: |\theta-\theta_0| < \epsilon, \eta \in \cH \Big\}
\eean
are $P_{\theta_0,\eta_0}$-Donsker for every $\epsilon > 0$.
\end{lemma}
\begin{proof}
Without loss of generality, we may assume that $\theta_0=0$.
By Theorem \ref{thm:bracketing_bound} with $\alpha=V=d=1, r=2$
and a partition $\bbR = \cup_{j=-\infty}^\infty [j-1, j)$,
we can easily check that
$N_{[]}(\delta, \cF, L_2(P_{\theta_0,\eta_0}))$ is bounded by $O(\delta^{-1})$ as $\delta \rightarrow 0$.
Therefore, it is a Donsker class by Theorem \ref{thm:donsker_bracketing}.
\end{proof}

\bigskip

Note that the total variation is bounded by the Kullback-Leibler divergence as
$$d_V^2(P, Q) \leq 2 K_P(Q)$$
by Pinsker's inequality.\footnote{See, for example, \citet{massart2007concentration}.}

\bigskip

\begin{lemma} \label{lem:KL_lower_bound}
For every $\epsilon > 0$ there exist $\delta > 0$ and
a universal constant $C > 0$ (does not depend on $\epsilon$) such that
$$
	\inf_{d_H(\eta, \eta_0) < \delta} \left(-P_{\theta_0,\eta_0} \log \frac{p_{\theta, \eta}}{p_{\theta_0, \eta}}\right)
	~\geq~ C \big(|\theta-\theta_0|^2 \wedge 1\big) - \epsilon.
$$
\end{lemma}
\begin{proof}
Without loss of generality, we assume that $\theta_0 = 0$.
Note that $P_{\theta, \eta}\big([\theta, \infty)\big) = 1/2$.
Since $p_{\theta_0,\eta_0}(\theta_0) > 0$ and $x \mapsto p_{\theta_0,\eta_0}(x)$ is continuous,
there exists $\tau \in (0,1)$ such that $p_{\theta_0,\eta_0}(x) > p_{\theta_0,\eta_0}(\theta_0)/2$ for $|x| < \tau$.
Then, for $\theta > \theta_0$, 
$$
	P_{\theta_0,\eta_0} \big([\theta, \infty)\big) 
	\leq 1/2 - (\theta \wedge \tau) p_{\theta_0,\eta_0}(\theta_0)/2,
$$
and therefore,
$$
	P_{\theta, \eta}\big([\theta, \infty)\big) - P_{\theta_0,\eta_0}\big([\theta, \infty)\big) ~\geq~ 
	(\theta \wedge \tau) p_{\theta_0,\eta_0}(\theta_0)/2 ~\geq~ 
	\tilde C \big(\theta \wedge 1\big),
$$
where $\tilde C = \tau p_{\theta_0,\eta_0}(\theta_0)/2$.
The same argument can be applied for $\theta < 0$, and as a result, we have
$d_V(P_{\theta_0,\eta_0}, P_{\theta, \eta}) ~\geq~ 2\tilde C \big(|\theta| \wedge 1\big)$.
For given $\epsilon > 0$, we can choose $\delta > 0$, by Lemma \ref{lem:L2_conv},
such that $d_H(\eta, \eta_0) < \delta$ implies $K_{P_{\theta_0,\eta_0}}(P_{\theta_0, \eta}) < \epsilon$.
Since 
$$
	-P_{\theta_0,\eta_0} \log \frac{p_{\theta, \eta}}{p_{\theta_0, \eta}} = 
	K_{P_{\theta_0,\eta_0}}(P_{\theta, \eta}) - K_{P_{\theta_0,\eta_0}}(P_{\theta_0, \eta})
$$
and $K_{P_{\theta_0,\eta_0}}(P_{\theta, \eta}) \geq d_V^2 (P_{\theta_0,\eta_0}, P_{\theta, \eta})/2$
by Pinsker's inequality, we have the conclusion with $C = 2\tilde C^2$.
\end{proof}

\bigskip

Let $\theta^*(\eta)$ as a maximizer of the map $\theta \mapsto P_{\theta_0,\eta_0} \ell_{\theta, \eta}$ if it exists.
Example \ref{ex:mixnorm} says that the expectation of log density can be
approximated by a quadratic function near $\theta_0$ for every fixed $\eta$.
Since $V_{\eta_0} = I_{\theta_0,\eta_0}$ is strictly positive, if $\eta$ is sufficiently close to $\eta_0$,
then $\theta_0$ will be a local maximizer of the function $\theta \mapsto P_{\theta_0,\eta_0} \ell_{\theta,\eta}(x)$.
If $\theta$ is not sufficiently close to $\theta_0$, then it can never be 
a maximizer of $\theta \mapsto P_{\theta_0,\eta_0} \ell_{\theta,\eta}(x)$ by Lemma \ref{lem:KL_lower_bound},
so $\theta_0$ is expected to the global maximizer.
That is, even when the nuisance parameter $\eta$ is slightly misspecified,
the Kullback-Leibler divergence of the misspecified model is minimized at $\theta_0$.
Here, the distance between $\eta_1$ and $\eta_2$ in $\cH$ is measured by $d_H$ defined by 
$d_H(\eta_1, \eta_2) = h(p_{\eta_1}, p_{\eta_2})$.
This is summarized in the following lemma.

\bigskip

\begin{lemma} \label{lem:tht_star}
There exists an $\epsilon > 0$ such that 
$\theta_0$ is the unique maximizer 
of the map $\theta \mapsto P_{\theta_0,\eta_0} \ell_{\theta, \eta}$ for all $\eta$ with $d_H(\eta,\eta_0) < \epsilon$.
\end{lemma}
\begin{proof}
Without loss of generality, we may assume that $\theta_0$ = 0.
Since $V_{\eta_0} > 0$ we can choose $\delta \in (0,1)$ such that
$$
	\sup_{\eta \in \cH} \Big| P_{\theta_0,\eta_0} \log \frac{p_{\theta, \eta}}{p_{\theta_0, \eta}} + \frac{\theta^2}{2}  
	V_\eta \Big| < \frac{\theta^2}{8}V_{\eta_0}
$$
for every $|\theta| \leq \delta$ by Example \ref{ex:mixnorm}.
Since $\lim_{d_H(\eta, \eta_0) \rightarrow 0} V_\eta = V_{\eta_0}$ by \eqref{eq:L2_conv_score},
we can also choose an $\epsilon > 0$ such that $|V_\eta - V_{\eta_0}| < V_{\eta_0}/2$
for $d_H(\eta,\eta_0) < \epsilon$.
Therefore, for $d_H(\eta,\eta_0) < \epsilon$ and $|\theta| \leq \delta$,
$$
	P_{\theta_0,\eta_0} \log \frac{p_{\theta, \eta}}{p_{\theta_0, \eta}} < -\frac{\theta^2}{2} V_\eta + \frac{\theta^2}{8}V_{\eta_0}
	< -\frac{\theta^2}{4} V_{\eta_0} + \frac{\theta^2}{8}V_{\eta_0} < 0,
$$
and as a result, we have $\argmax_{|\theta| \leq \delta} P_{\theta_0,\eta_0} \log P_{\theta, \eta} = \theta_0$.
By Lemma \ref{lem:KL_lower_bound}, $\epsilon > 0$ can be chosen sufficiently small so that
$$-P_{\theta_0,\eta_0} \log\frac{p_{\theta,\eta}}{p_{\theta_0,\eta}} \geq C (\theta^2\wedge 1) - C\delta^2$$
for all $d_H(\eta,\eta_0) < \epsilon$, where $C > 0$ is a constant in Lemma \ref{lem:KL_lower_bound}.
Therefore, 
$$\sup_{|\theta| > \delta} P_{\theta_0,\eta_0} \log P_{\theta, \eta} < \log P_{\theta_0, \eta}$$
and this yields $\theta^*(\eta) = \theta_0$.
\end{proof}

\bigskip

\subsection{Linear regression model}
\label{sec:lm}

This section considers the linear regression model
\be \label{eq:reg_model}
	Y_i = \beta^T X_i + \epsilon_i, ~~~~~~ i=1,\ldots,n,
\ee
for independent observations $(X_1, Y_1), \ldots, (X_n, Y_n)$,
where the error distribution follows a mixture of normal densities as in the previous section.
The parameter space for $\beta$ is denoted by $\cB$ which is an open subset of $\bbR^p$.
The parameter space $\cH$ for $\eta$ and the corresponding density $p_\eta$
are defined as in Example \ref{ex:mixnorm}.
The $p$-dimensional covariate vectors $X_i$'s are non-random and their norms are assumed to be
uniformly bounded by a constant $L > 0$.
Additionally, we denote $P_{\beta,\eta,i}$ be the probability measure of $Y_i$ 
in the model \eqref{eq:reg_model} conditional on $X_i$, and
$$p_{\beta,\eta,i} = \int \phi_\sigma(Y_i -\beta^T X_i - z) d\eta(z,\sigma)$$
be its density evaluated at $Y_i$.
The corresponding log density and its derivative evaluated at $Y_i$ are denoted by
$\ell_{\beta,\eta,i}$ and $\ell^\prime_{\beta,\eta,i}$, respectively.
Let $\score_{\beta,\eta,i} = - \ell^\prime_{\beta,\eta,i} \cdot X_i$ 
which is equal to the partial derivative of the map $\beta \mapsto \ell_{\beta,\eta,i}$.
$P_{\beta,\eta}^n$ represents the product measure 
$P_{\beta,\eta,1} \times \cdots \times P_{\beta,\eta,n}$.
Let $\bbX_n = n^{-1} \sum_{i=1}^n X_i X_i^T$ be the design matrix
and $V_{n, \eta} = V_\eta \cdot \bbX_n$, where 
$$
	V_\eta = P_{\beta_0, \eta_0,i} \ell^\prime_{\beta_0,\eta,i}\ell^\prime_{\beta_0,\eta_0,i}
$$
which is the same to the location model.\footnote{Note that the definition of $V_\eta$ does not depend on $i$.}
We assume that the minimum eigenvalue $\rho_{\rm min}(\bbX_n)$ of $\bbX_n$ is bounded away from 0
in the sense $\liminf_{n\rightarrow\infty} \rho_{\rm min}(\bbX_n) > 0$
which is required for the identifiability of $\beta$.
Now, we state the main theorem of this section.

\bigskip

\begin{theorem}\label{thm:main_lm}
Suppose that there is the true parameter $(\beta_0, \eta_0) \in \cB\times\cH$ 
which generates the observation from the model \eqref{eq:reg_model}
with $\liminf_{n\rightarrow\infty}\rho_{\rm min}(\bbX_n) > 0$ and $\sup_i |X_i| \leq L$ for some $L > 0$.
If the model is endowed with the product prior $\Pi=\Pi_\cB \times \Pi_\cH$,
where $\Pi_\cB$ is thick at $\beta_0$ and $\Pi_\cH(U) > 0$ 
for every weak neighborhood $U$ of $\eta_0$, then
\bean
	\sup_B \left|\Pi(\sqrt{n}(\beta-\beta_0) \in B | X_1, \ldots, X_n) 
	- N_{\Delta_n, V^{-1}_{n,\eta_0}}(B) \right| \rightarrow 0,
\eean
in $P_0$-probability, where 
$$ 
  \Delta_n = \frac{1}{\sqrt{n}} \sum_{i=1}^n V_{n, \eta_0}^{-1} \score_{\beta_0, \eta_0,i}.
$$
\end{theorem}
\begin{proof}
The proof is similar to that of Theorem \ref{thm:main_location}.
As in the proof of Theorem \ref{thm:main_location} we will show that 
\eqref{eq:zero_score_condition}--\eqref{eq:conv_perturb_general},
and \eqref{eq:consist} hold with $\theta$ and $X$ replaced by $\beta$ and $Y$, respectively.
The definition of $(\cH_n)$, proofs of \eqref{eq:conv_perturb_general} and \eqref{eq:consist}
are the same to those of Theorem \ref{thm:main_location} replacing the median $\hat\theta_n$
by the least square estimator $\hat\beta_n$.

Let $\score_{\beta,\eta}^{(n)}(Y^{(n)}) = \sum_{i=1}^n \score_{\beta,\eta,i}$, and
$V_{n,\eta}$ be defined as in the beginning of this section.
Then, \eqref{eq:V_conti_condition} holds by Lemma \ref{lem:L2_conv} and 
\eqref{eq:V_positive} holds by the condition on the design matrix.
Since $\sup_i |X_i|$ is bounded, we have, by Example \ref{ex:mixnorm},
$$
	\sup_{i}\sup_{\eta \in \cH} \Big| P_0 \log \frac{p_{\beta, \eta,i}}{p_{\beta_0, \eta,i}} + 
	\frac{V_\eta}{2} (\beta-\beta_0)^T X_i X_i^T (\beta-\beta_0) \Big| = o(|\beta-\beta_0|^2)
$$
which directly implies \eqref{eq:quadratic_condition}.
Also, by \eqref{eq:V_conti_condition}, \eqref{eq:V_positive}, \eqref{eq:quadratic_condition}
and differentiability, $\beta_0$ is a local maximizer of 
$\beta \mapsto \sum_{i=1}^n P_0 \ell_{\beta,\eta,i}$
for large enough $n$ and $\eta$ sufficiently close to $\eta_0$ with respect to $d_H$.
As a result \eqref{eq:zero_score_condition} is satisfied.
The result of Lemma \ref{lem:donsker_lm} implies \eqref{eq:donsker_condition}.

For a given nonzero vector $a \in \bbR^p$, the stochastic process
$$
	\eta \mapsto \frac{(a^T \bbX_n a)^{-1/2}}{\sqrt{n}}  \sum_{i=1}^n a^T \score_{\beta_0,\eta,i}
$$
indexed by $T=\{\eta: d_H(\eta,\eta_0) < \epsilon\}$ for sufficiently small $\epsilon > 0$
is asymptotically tight by Theorem \ref{lem:donsker_lm} and the condition on $\bbX_n$.
Furthermore, it converges marginally to a Gaussian distribution by the Lindberg-Feller's theorem,
so weakly converges to a Gaussian process.
As a result, for the semimetric\footnote{The definition of $d_2$ is the same to \eqref{eq:L2_metric}.}
$$
	d_2(\eta_1, \eta_2) = \Big(P_{\eta_0} (\ell^\prime_{\eta_1} - \ell^\prime_{\eta_2})^2\Big)^{1/2},
$$
the process is asymptotically uniformly $d_2$-equicontinuous in probability
because $T$ is totally bounded\footnote{See the proof of Theorem \ref{thm:main_location}.}
with respect to $d_2$.
Since $d_H(\eta,\eta_0) \rightarrow 0$ implies $d_2(\eta,\eta_0) \rightarrow 0$,
\eqref{eq:score_conti_condition} holds.

For \eqref{eq:lr_apprx_small_condition}, Tayler expansion implies
\bean
	&&\ell^{(n)}_{\beta_n(h),\eta}(Y^{(n)}) - \ell^{(n)}_{\beta_0,\eta}(Y^{(n)})
	- \frac{h^T}{\sqrt{n}} \score_{\beta_0,\eta}^{(n)}(Y^{(n)})
	\\
	&=& \frac{h^T}{\sqrt{n}} \int_0^1 \sum_{i=1}^n \Big( \score_{\beta_0+th/\sqrt{n},\eta,i}
	- \score_{\beta_0,\eta,i} \Big) dt
\eean
and by the Jensen's inequality and Fubini's theorem,
the variance of the right hand side is bounded by
\bean
	\frac{1}{n} \int_0^1 \sum_{i=1}^n \bigg[P_0 \Big( \ell^\prime_{\beta_0+th/\sqrt{n},\eta,i}
	- \ell^\prime_{\beta_0,\eta,i} \Big)^2 \cdot h^T X_iX_i^T h \bigg]dt = O(|h|^4/n)
\eean
by (vi) of Lemma \ref{lem:density}.
Therefore, for each $(h, \eta)$ the term in the left hand side of \eqref{eq:lr_apprx_small_condition}
converges in probability to 0, and it converges uniformly by Lemma \ref{lem:donsker_lm}.
Finally, write
\bean
	\ell^{(n)}_{\beta_n(h),\eta}(Y^{(n)}) - \ell^{(n)}_{\beta_0,\eta}(Y^{(n)})
	= \frac{h^T}{\sqrt{n}} \int_0^1 \sum_{i=1}^n \score_{\beta_0 + th/\sqrt{n}, \eta,i} dt
\eean
and apply Lemma \ref{lem:donsker_lm} to prove \eqref{eq:lr_apprx_large_condition}.
\end{proof}

\bigskip

The following lemma corresponds to the Donsker's theorem for \iid~models.

\bigskip

\begin{lemma} \label{lem:donsker_lm}
If $\sup_i |X_i| \leq L$ for some constant $L > 0$, then for any $a \in \bbR^p$,
there exists $\epsilon > 0$ such that the sequence of stochastic processes
$$
	(\beta,\eta) \mapsto \frac{a^T}{\sqrt{n}} \sum_{i=1}^n \left( \score_{\beta,\eta,i}
	- P_0 \score_{\beta,\eta,i} \right)
$$
is asymptotically tight in $\ell^\infty(B_\epsilon \times \cH)$,
where $B_\epsilon$ is an open ball of $\beta_0$ with radius $\epsilon$.
\end{lemma}
\begin{proof}
For given $a \in \bbR^p$ we will prove the assertion using Theorem \ref{thm:jain_marcus}.
Without loss of generality, we may assume that $\beta_0=0$.
If $\epsilon > 0$ is sufficiently small, there exists a square-integrable $Q_0$ by Lemma \ref{lem:quad_bound}
such that $\sqrt{n} |Z_{ni}(\beta,\eta)| \leq Q_0(Y_i)$,
where $Z_{ni}(\beta,\eta) = a^T\score_{\beta,\eta,i}/\sqrt{n}$,
a stochastic process indexed by $B_\epsilon \times \cH$.
Let $d_2$ be a metric on $\cH$ defined as
$d^2_2(\eta_1,\eta_2) = \int e^y\big(\ell^\prime_{\eta_1}(y) - \ell^\prime_{\eta_2}(y)\big)^2 dP_{\eta_0}(y)$.
Also, let $\rho$ be the product metric $|\cdot|\times d_2$ on $B_\epsilon \times \cH$
defined as $\rho\big((\beta_1,\eta_1),(\beta_2,\eta_2)\big) = \max \{|\beta_1-\beta_2|, d_2(\eta_1,\eta_2)\}$.

Since $\sqrt{n}|Z_{ni}(\beta,\eta)| < Q_0(Y_i)$ for every $(\beta,\eta) \in B_\epsilon\times\cH$
and $Q_0$ is square-integrable, the triangular array of random variables
$\big(\sup_{\beta \in B_\epsilon}\sup_{\eta\in\cH} |Z_{ni}(\beta,\eta)|\big)$
satisfies the Lindberg's condition.
By the triangular inequality,
$$
	|Z_{ni}(\beta_1,\eta_1) - Z_{ni}(\beta_2,\eta_2)| \leq
	|Z_{ni}(\beta_1,\eta_1) - Z_{ni}(\beta_2,\eta_1)|
	+ |Z_{ni}(\beta_2,\eta_1) - Z_{ni}(\beta_2,\eta_2)|,
$$
and the first term of the right-hand-side is bounded by
\bean
	\sup_{\eta\in\cH}|Z_{ni}(\beta_1,\eta) - Z_{ni}(\beta_2,\eta)|
	&\leq& \sup_{\eta \in \cH} \frac{|a|}{\sqrt{n}} 
	\big| \score_{\beta_1,\eta,i} - \score_{\beta_2,\eta,i} \big| \\
	&\leq& \frac{|a| \cdot |X_i|}{\sqrt{n}} \sup_{\eta \in \cH} 
	\big|\ell^\prime_{\beta_1,\eta,i} - \ell^\prime_{\beta_2,\eta,i}\big| \\
	&\leq& \frac{K_1}{\sqrt{n}} \cdot |\beta_1-\beta_2| \cdot (Y_i+K_2)^2
\eean
for some constants $K_1, K_2 > 0$ independent of $i$, where the last inequality holds by (vi) of Lemma \ref{lem:density}.
The expectation of the square of the second term can be bounded by
$$
	P_0 \big(Z_{ni}(\beta_2,\eta_1) - Z_{ni}(\beta_2,\eta_2)\big)^2
	\leq \frac{|a|^2 \cdot |L|^2}{n} 
	P_0\big(\ell^\prime_{\beta_2,\eta_1,i} - \ell^\prime_{\beta_2,\eta_2,i}\big)^2,
$$
and if $\epsilon >0$ is sufficiently small,
\bean
	P_0\big(\ell^\prime_{\beta,\eta_1,i} - \ell^\prime_{\beta,\eta_2,i}\big)^2
	&=& \int \Big(\ell^\prime_{\beta,\eta_1,i}(y) - \ell^\prime_{\beta,\eta_2,i}(y)\Big)^2
	dP_{\beta_0,\eta_0,i}(y) \\
	&=& \int \Big(\ell^\prime_{\beta_0,\eta_1,i}(y) - \ell^\prime_{\beta_0,\eta_2,i}(y)\Big)^2
	dP_{-\beta,\eta_0,i}(y) \\
	&=& \int \Big(\ell^\prime_{\beta_0,\eta_1,i}(y) - \ell^\prime_{\beta_0,\eta_2,i}(y)\Big)^2
	\frac{dP_{\beta,\eta_0,i}}{dP_{\beta_0,\eta_0,i}}(y) dP_{\beta_0,\eta_0,i}(y)\\
	&\leq& K_3 d^2_2(\eta_1,\eta_2)
\eean
for some constant $K_3>0$ independent of $i$ and $\beta$,
where the third equality is due to the symmetricity and
the last inequality holds because
$$
	\frac{dP_{\beta,\eta_0,i}}{dP_{\beta_0,\eta_0,i}}(y)
	= \exp\big( \ell_{\beta,\eta_0,i}(y) - \ell_{\beta_0,\eta_0,i}(y) \big)
	= \exp\big( \ell^\prime_{\eta_0} (y-t)\cdot \beta^T X_i  \big)
$$
for some $t$ between $\beta_0^T X_i$ and $\beta^T X_i$ by the mean value theorem,
and $\ell^\prime_{\eta_0}(y) = O(y)$ as $y \rightarrow \infty$ by (ii) of Lemma \ref{lem:density}.
Since $\sup_i P_0 Y_i^4 < \infty$, there exist a global constant $K$ such that
$$
	\sum_{i=1}^n P_0 \left[\frac{Z_{ni}(\beta_1,\eta_1) - Z_{ni}(\beta_2,\eta_2)}
	{\rho\big((\beta_1,\eta_1),(\beta_2,\eta_2)\big)} \right]^2 \leq K
$$
for every $(\beta_1,\eta_1), (\beta_2,\eta_2) \in B_\epsilon\times\cH$ and $n \geq 1$.

It only remains to prove 
$$ \int_0^\infty \sqrt{\log N(\delta, B_\epsilon\times\cH, \rho)} d\delta < \infty.$$
Note that $\log N(\delta, B_\epsilon\times\cH, \rho) \leq 
\log N(\delta, B_\epsilon, |\cdot|) + \log N(\delta, \cH, d_2)$ and
$N(\delta, B_\epsilon, |\cdot|) = O(\delta^{-p})$ as $\delta \rightarrow 0$.
If we define $\cF = \{ y \mapsto e^{y/2} \ell^\prime_\eta(y) : \eta \in \cH\}$,
then $N(\delta,\cH, d_2)$ is equal to $N(\delta, \cF, L_2(P_{\eta_0}))$
and this is bounded above by the bracketing number $N_{[]}(\delta, \cF, L_2(P_{\eta_0}))$.
The bracketing entropy $\log N_{[]}(\delta, \cF, L_2(P_{\eta_0}))$ is of order $O(\delta^{-1})$ 
as $\delta \rightarrow 0$ by applying Theorem \ref{thm:bracketing_bound} with $\alpha=V=d=1$, 
$r=2$ and a partition $\bbR = \cup_{j=-\infty}^\infty [j-1,j)$.
Therefore, we complete the proof.
\end{proof}

\bigskip

\subsection{Random intercept model}
\label{sec:lme}

For the independent data $\{(\bY_i, \bX_i)\}_{i=1}^n$,
we consider the random intercept model
\be\label{eq:lin_mix_eff_model}
	Y_{ij} = \beta^T X_{ij} + b_i + \epsilon_{ij}, ~~~~~ i=1,\ldots,n, ~ j=1, \ldots, m_i
\ee
where $\bY_i = (Y_{i1}, \ldots, Y_{im_i})^T \in \bbR^{m_i}$.
$\bX_i = (X_{i1}, \ldots, X_{im_i})^T \in \bbR^{m_i\times p}$.
For the given random effect $b_i$, 
the errors $\epsilon_{ij}, j=1, \ldots, m_i$ are conditionally independent and follow $P_F$,
where $P_F$ is the probability measure on $\bbR$
whose Lebesgue density is given by $p_F(y) = \int \phi_\sigma(y-z) dF(z,\sigma)$.
The random effects $b_i$'s are \iid~from distribution $G$.
Since there are two unknown distributions $F$ and $G$,
we need different notations from those used in previous sections.

For $\bx \in \bbR^{m\times p}$ and $\beta \in \bbR^p$, let 
$\score_{\beta, \eta}(\by|\bx) = -\bx^T \cdot  \nabla \ell_{\eta}(\by-\bx\beta)$,
$\score_{\beta, \eta,i} = \score_{\beta, \eta}(\bY_i|\bX_i)$
and $\nabla \ell_{\beta,\eta,i} = \nabla \ell_\eta(\bY_i - \bX_i \beta)$.
$P_{\beta,\eta,i}$ denotes the probability for $\bY_i$ in the model \eqref{eq:lin_mix_eff_model}.
Define the metric $d_H$ on $\cH$ by $d_H(\eta_1, \eta_2) = h(p_{\eta_1}, p_{\eta_2})$.
We assume that there exists the true parameter $(\beta_0, \eta_0) \in \cB\times\cH$
generating the data and let $V_\eta = P_{\eta_0}\big[ \nabla\ell_\eta \cdot \nabla\ell_{\eta_0}^{\;T} \big]$.
Also, define
\be \label{eq:V_n_def_lme}
	V_{n,\eta} = \frac{1}{n} \sum_{i=1}^n P_0
	\big[ \score_{\beta, \eta,i} \score_{\beta, \eta_0,i}^T \big]
	= \frac{1}{n} \sum_{i=1}^n \bX_i^T V_\eta \bX_i.
\ee
Since $V_{\eta_0}$ is positive definite matrix,
so is $V_{n,\eta}$ for large enough $n$ provided $\liminf_{n\rightarrow\infty} \rho_{\rm min}(\bbX_n) > 0$,
where $\bbX_n = n^{-1} \bX_i^T \bX_i$ is the design matrix.

\bigskip

\begin{theorem}\label{thm:main_lme}
Suppose that there is the true parameter $(\beta_0, \eta_0) \in \cB\times\cH$ 
which generates the observation from the model \eqref{eq:lin_mix_eff_model}
with $\liminf_{n\rightarrow\infty}\rho_{\rm min}(\bbX_n) > 0$ and $\sup_i \|\bX_i\| \leq L$ for some $L > 0$.
If the model is endowed with the product prior $\Pi=\Pi_\cB \times \Pi_\cH$,
where $\Pi_\cB$ is thick at $\beta_0$ and $\Pi_\cH(U) > 0$ 
for every weak neighborhood $U$ of $\eta_0$, then
\bean
	\sup_B \left|\Pi(\sqrt{n}(\beta-\beta_0) \in B | X_1, \ldots, X_n) 
	- N_{\Delta_n, V^{-1}_{n,\eta_0}}(B) \right| \rightarrow 0,
\eean
in $P_0$-probability, where 
$$ 
  \Delta_n = \frac{1}{\sqrt{n}} \sum_{i=1}^n V_{n, \eta_0}^{-1} \score_{\beta_0, \eta_0,i}.
$$
\end{theorem}
\begin{proof}
The proof is similar to that of Theorems \ref{thm:main_location} and \ref{thm:main_lm}.
We will show that \eqref{eq:zero_score_condition}--\eqref{eq:conv_perturb_general},
and \eqref{eq:consist} hold with $\theta$ and $X$ replaced by $\beta$ and $Y$, respectively.
The proofs of \eqref{eq:conv_perturb_general} and \eqref{eq:consist} are slightly different
from previous two models because of the presence of random effects.
Note first that by Lemma \ref{lem:Hellinger_ineq},
for $\btheta = (\theta_1, \ldots, \theta_m)^T$
and $\btheta_0 = (\theta_{0;1}, \ldots, \theta_{0;m})^T$,
\bea
	\sup_{\eta\in\cH} h^2\big(P_{\btheta, \eta}, P_{\btheta_0,\eta}\big)
	&\leq& \sup_{F\in\cF} \sup_{b} h^2\Big(P_{\theta_1+b,F}\times\cdots\times P_{\theta_m+b,F},
	P_{\theta_{0;1}+b,F} \times \cdots \times P_{\theta_{0;m}+b,F}\Big)
	\nonumber \\
	&=& \sup_{F\in\cF} \left\{2 - 2\cdot \prod_{j=1}^m
	\left[1-\frac{1}{2} h^2 \big(P_{\theta_j,F},P_{\theta_{0;j},F} \big)\right]\right\}
	\nonumber \\
	&\leq& 2 - 2\cdot \Big(1-O(|\btheta-\btheta_0|)\Big)^m 
	\nonumber \\
	&\leq& m \times O(|\btheta-\btheta_0|^2)
\eea
as $|\btheta - \btheta_0| \rightarrow 0$.
Therefore, \eqref{eq:h_unif_libschitz} is satisfied by the boundedness of covariate.
Condition \eqref{eq:d_H_def} is trivial, and 
\eqref{eq:quad_bdd_general} holds by \eqref{eq:grad}.
Note also that the least square estimator $\hat\beta_n$ satisfies \eqref{eq:sqrtn_estimator}.
To apply Theorems \ref{thm:consist_perturb_general} and \ref{thm:consist_general}
it is sufficient to show that $N(\delta, \cH, d_H) < \infty$ and $\Pi_\cH (K(\delta)) > 0$
for every $\delta > 0$ by Lemma \ref{lem:Lipschitz}.
Both facts hold by Lemmas \ref{lem:entropy_bound} and \ref{lem:KL_bound}, respectively,
treating random effects also as mixture components.
We conclude that there exists a sequence $\epsilon_n \downarrow 0$ such that
the sequence $(\cH_n)$ defined by $\cH_n = \{\eta\in\cH: d_H(\eta,\eta_0) < \epsilon_n\}$
satisfies \eqref{eq:conv_perturb_general} and \eqref{eq:consist}.

Let $\score_{\beta,\eta}^{(n)}(Y^{(n)}) = \sum_{i=1}^n \score_{\beta,\eta,i}$, and
$V_{n,\eta}$ be defined as \eqref{eq:V_n_def_lme}.
Then, \eqref{eq:V_conti_condition} holds by Lemma \ref{lem:V_conti_lme} and 
\eqref{eq:V_positive} holds by the condition on $\bbX_n$.
Since $\sup_i \|\bX_i\|$ is bounded, we have, by Example \ref{ex:mixmnorm},
\be \label{eq:uql_reg_lme}
	\sup_{i}\sup_{\eta \in \cH} \Big| P_0 \log \frac{p_{\beta, \eta,i}}{p_{\beta_0, \eta,i}} + 
	\frac{1}{2} (\beta-\beta_0)^T \bX_i^T V_\eta \bX_i (\beta-\beta_0) \Big| = o(|\beta-\beta_0|^2)
\ee
as $\beta\rightarrow\beta_0$
which directly implies \eqref{eq:quadratic_condition}.
Also, by \eqref{eq:V_conti_condition}, \eqref{eq:V_positive} and \eqref{eq:quadratic_condition},
$\beta_0$ is a local maximizer of $\beta \mapsto \sum_{i=1}^n P_0 \ell_{\beta,\eta,i}$
for large enough $n$ and $\eta$ sufficiently close to $\eta_0$ with respect to $d_H$.
As a result \eqref{eq:zero_score_condition} is satisfied.
The result of Theorem \ref{lem:donsker_lme} implies \eqref{eq:donsker_condition}.

To prove \eqref{eq:score_conti_condition}, note that, for a given vector $a \in \bbR^p$,
the stochastic process 
$$
	\eta \mapsto \frac{a^T}{\sqrt{n}}  \sum_{i=1}^n \score_{\beta_0,\eta,i}
$$
indexed by $T=\{\eta: d_H(\eta,\eta_0) < \epsilon\}$ for sufficiently small $\epsilon > 0$,
is asymptotically uniformly $d$-equicontinuous in probability, 
where $d$ is the semimetric defined by \eqref{eq:d_def_lme}.
Since $d_H(\eta,\eta_0) \rightarrow 0$ implies $d(\eta,\eta_0) \rightarrow 0$ by Lemma \ref{lem:V_conti_lme}, 
we have \eqref{eq:score_conti_condition}.

For \eqref{eq:lr_apprx_small_condition}, by the Tayler expansion
\bean
	&&\ell^{(n)}_{\beta_n(h),\eta}(Y^{(n)}) - \ell^{(n)}_{\beta_0,\eta}(Y^{(n)})
	- \frac{h^T}{\sqrt{n}} \score_{\beta_0,\eta}^{(n)}(Y^{(n)})
	\\
	&=& \frac{h^T}{\sqrt{n}} \int_0^1 \sum_{i=1}^n \Big( \score_{\beta_0+th/\sqrt{n},\eta,i}
	- \score_{\beta_0,\eta,i} \Big) dt
\eean
and by the Jensen's inequality and Fubini's theorem,
the variance of the right hand side is bounded by $O(|h|^4/n)$.
Therefore, for each $(h, \eta)$ the term in the left hand side of \eqref{eq:lr_apprx_small_condition}
converges in probability to 0, and it converges uniformly by Lemma \ref{lem:donsker_lm}.
Finally, write
\bean
	\ell^{(n)}_{\beta_n(h),\eta}(Y^{(n)}) - \ell^{(n)}_{\beta_0,\eta}(Y^{(n)})
	= \frac{h^T}{\sqrt{n}} \int_0^1 \sum_{i=1}^n \score_{\beta_0 + th/\sqrt{n}, \eta,i} dt
\eean
and apply Lemma \ref{lem:donsker_lm}, then \eqref{eq:lr_apprx_large_condition} holds.
\end{proof}

\bigskip

\bigskip

\begin{lemma} \label{lem:V_conti_lme}
As $d_H(\eta,\eta_0) \rightarrow 0$,
$\ell_\eta \rightarrow \ell_{\eta_0}$ and $\nabla \ell_\eta \rightarrow \nabla \ell_{\eta_0}$
in $L_2(P_{\eta_0})$.
Furthermore, $d(\eta, \eta_0) \rightarrow 0$ as $d_H(\eta,\eta_0) \rightarrow 0$,
where $d$ is the semimetric defined by \eqref{eq:d_def_lme}.
\end{lemma}
\begin{proof}
First, $\ell_\eta \rightarrow \ell_{\eta_0}$ in $L_2(P_{\eta_0})$ as $d_H(\eta, \eta_0) \rightarrow 0$
by Theorem 5 of \citet{wong1995probability}.
Note that
\bean
	\nabla_j \ell_{\eta}(\by) = \lim_{h \rightarrow 0} \frac{\ell_{\eta}(\by+h\bfe_j) - \ell_\eta(\by)}{h},
\eean
where $\bfe_j$ is the $j$th unit vector.
Also,
\bea
	&& \sup_{\eta\in\cH} \left|
	\int \left(\frac{\ell_\eta(\by+h\bfe_j) - \ell_\eta(\by)}{h} 
	- \nabla_j \ell_\eta(\by) \right)^2 dP_{\eta_0}(\by) \right|
	\nonumber\\
	&& \leq \int \sup_{\eta\in\cH} \sup_{t\in[0,1]} \big| \nabla_j
	\ell_{\eta}(\by+th\bfe_j) - \nabla_j \ell_{\eta}(\by) \big|^2 dP_{\eta_0}(\by)
	\nonumber \\
	&& \rightarrow 0 \label{eq:unif_conv_grad_lme}
\eea
as $h\rightarrow 0$.
Therefore, by the dominated convergence theorem and Moore-Osgood theorem,
\bean
	&&\lim_{\eta \rightarrow \eta_0}
	\int \Big( \nabla_j \ell_\eta(\by) - \nabla_j \ell_{\eta_0}(\by) \Big)^2 dP_{\eta_0}(\by)
	\\
	&&= \lim_{\eta \rightarrow \eta_0} \lim_{h\rightarrow 0}
	\int \left(\frac{\ell_{\eta}(\by+h\bfe_j) - \ell_{\eta}(y)}{h} 
	- \nabla_j \ell_{\eta_0}(\by) \right)^2 dP_{\eta_0}(\by)
	\\
	&& = \lim_{h\rightarrow 0} \lim_{\eta \rightarrow \eta_0}
	\int \left(\frac{\ell_{\eta}(\by+h\bfe_j) - \ell_{\eta}(\by)}{h} 
	- \nabla_j \ell_{\eta_0}(\by) \right)^2 dP_{\eta_0}(\by) 
	 \\
	&& = 0,
\eean
where the convergence of the limit of $\eta$ is taken by the metric $d_H$.
Therefore, $\nabla \ell_\eta \rightarrow \nabla \ell_{\eta_0}$ in $L_2(P_{\eta_0})$
as $d_H(\eta,\eta_0) \rightarrow 0$.

Finally, note that the uniform convergence \eqref{eq:unif_conv_grad_lme} still holds
when the integrand is multiplied by $e^{|\by|}$.
Therefore, the proof above can be applied for the convergence under the semimetric $d$.
\end{proof}

\bigskip

Note that for every $L >0$, there exist $K_1, K_2 > 0$ and an open neighborhood $U$ of $\beta_0$ such that
\be \label{eq:libschitz_bdd}
	\sup_{|\bx|\leq L} \sup_{\eta\in\cH} 
	\big| \score_{\beta_1, \eta}(\by|\bx) - \score_{\beta_2, \eta}(\by|\bx) \big|
	\leq K_1\cdot|\beta_1 - \beta_2| \cdot (|\by|+K_2)^2
\ee
for every $\beta_1, \beta_2 \in U$.
The following Lemma corresponds to the Donsker's theorem for \iid~models and
Lemma \ref{lem:donsker_lm} for non-\iid~regression model.

\bigskip

\begin{lemma} \label{lem:donsker_lme}
If $\sup_i \|\bX_i\| < L$ for some constant $L > 0$, then 
for any $a \in \bbR^p$, there exists $\epsilon > 0$ such that the sequence of stochastic processes
$$
	(\beta,\eta) \mapsto \frac{a^T}{\sqrt{n}} \sum_{i=1}^n \left( \score_{\beta,\eta,i}
	- P_0 \score_{\beta,\eta,i} \right)
$$
is asymptotically tight in $\ell^\infty(B_\epsilon \times \cH)$,
where $B_\epsilon$ is the open ball of $\beta_0$ with radius $\epsilon$.
\end{lemma}
\begin{proof}
We will prove the assertion using Theorem \ref{thm:jain_marcus}.
Without loss of generality, we may assume that $\beta_0=0$.
Since $\sup_i |X_i|$ is bounded, if $\epsilon > 0$ is sufficiently small,
then there exists a function $Q$
such that $\sqrt{n} |Z_{ni}(\beta,\eta)| \leq Q(Y_i)$ 
for every $(\beta,\eta) \in B_\epsilon\times \cH$,
and $\int Q^2(y) dP_{\eta_0}(y) < \infty$,
where $Z_{ni}(\beta,\eta) = a^T\score_{\beta,\eta,i}/\sqrt{n}$,
a stochastic process indexed by $B_\epsilon \times \cH$.
Let $d$ be a metric on $\cH$ defined as
\be \label{eq:d_def_lme}
	d^2(\eta_1,\eta_2) = \int e^{|\by|}\big|\nabla\ell_{\eta_1}(\by)
	- \nabla\ell_{\eta_2}(\by)\big|^2 dP_{\eta_0}(\by).
\ee
Also, let $\rho$ be the product metric $|\cdot|\times d$ on $B_\epsilon \times \cH$
defined as 
$$
	\rho\big((\beta_1,\eta_1),(\beta_2,\eta_2)\big) = \max \{|\beta_1-\beta_2|, d(\eta_1,\eta_2)\}.
$$

Since $\sqrt{n}|Z_{ni}(\beta,\eta)| < Q(Y_i)$ for every $(\beta,\eta) \in B_\epsilon\times\cH$
and $Q$ is square-integrable, the triangular array of random variables
$\big(\sup_{\beta \in B_\epsilon}\sup_{\eta\in\cH} |Z_{ni}(\beta,\eta)|\big)$
satisfies the Lindberg's condition.
By the triangular inequality,
$$
	|Z_{ni}(\beta_1,\eta_1) - Z_{ni}(\beta_2,\eta_2)| \leq
	|Z_{ni}(\beta_1,\eta_1) - Z_{ni}(\beta_2,\eta_1)|
	+ |Z_{ni}(\beta_2,\eta_1) - Z_{ni}(\beta_2,\eta_2)|,
$$
and the first term of the right-hand-side is bounded by
\bean
	\sup_{\eta\in\cH}|Z_{ni}(\beta_1,\eta) - Z_{ni}(\beta_2,\eta)|
	&\leq& \frac{K_1}{\sqrt{n}} \cdot |\beta_1-\beta_2| \cdot (|Y_i|+K_2)^2
\eean
for some constants $K_1, K_2 > 0$ and $\beta_1,\beta_2 \in B_\epsilon$
by \eqref{eq:libschitz_bdd} provided $\epsilon$ is sufficiently small.
The expectation of the square of the second term can be bounded by
\bean
	&&P_0 \big(Z_{ni}(\beta_2,\eta_1) - Z_{ni}(\beta_2,\eta_2)\big)^2 \\
	&&\leq \frac{|a|^2}{n}\cdot
	P_0\Big|\big(\nabla\ell_{\beta_2,\eta_1,i} - \nabla\ell_{\beta_2,\eta_2,i}\big)^T
	\bX_i \bX_i^T \big(\nabla\ell_{\beta_2,\eta_1,i} - \nabla\ell_{\beta_2,\eta_2,i}\big)\Big|\\
	&&\leq \frac{K_3}{n} \cdot P_0\big|\nabla\ell_{\beta_2,\eta_1,i} - \nabla\ell_{\beta_2,\eta_2,i}\big|^2,
\eean
for some $K_3 > 0$.
If $\epsilon >0$ is sufficiently small,
\bean
	P_0\big|\nabla\ell_{\beta,\eta_1,i} - \nabla\ell_{\beta,\eta_2,i}\big|^2
	&=& \int \Big|\nabla\ell_{\beta,\eta_1,i}(\by) - \nabla\ell_{\beta,\eta_2,i}(\by)\Big|^2
	dP_{\eta_0}(\by) \\
	&=& \int \Big|\nabla\ell_{\eta_1}(\by) - \nabla\ell_{\eta_2}(\by)\Big|^2
	dP_{-\beta,\eta_0,i}(\by) \\
	&=& \int \Big|\nabla\ell_{\eta_1}(\by) - \nabla\ell_{\eta_2}(\by)\Big|^2
	\frac{dP_{-\beta,\eta_0,i}}{dP_{\eta_0}}(\by) dP_{\eta_0}(\by)\\
	&\leq& K_4 \cdot d^2_2(\eta_1,\eta_2)
\eean
where the last inequality holds because
$$
	\frac{dP_{\beta,\eta_0,i}}{dP_{\eta_0}}(\by)
	= \exp\big( \ell_{\eta_0}(\by - \bX_i\beta) - \ell_{\eta_0}(\by) \big)
	\leq K_4 \cdot \exp(|\by|)
$$
for some $K_4 > 0$ and $\beta$ sufficiently close to $\beta_0=0$.
Since $\sup_i P_0 |\bY_i|^4 < \infty$, there exist $\epsilon > 0$ and a constant $K$ such that
$$
	\sum_{i=1}^n P_0 \left[\frac{Z_{ni}(\beta_1,\eta_1) - Z_{ni}(\beta_2,\eta_2)}
	{\rho\big((\beta_1,\eta_1),(\beta_2,\eta_2)\big)} \right]^2 \leq K
$$
for every $(\beta_1,\eta_1), (\beta_2,\eta_2) \in B_\epsilon\times\cH$ and $n \geq 1$.

It only remains to prove 
$$ \int_0^\infty \sqrt{\log N(\delta, B_\epsilon\times\cH, \rho)} d\delta < \infty.$$
Note that $\log N(\delta, B_\epsilon\times\cH, \rho) \leq 
\log N(\delta, B_\epsilon, |\cdot|) + \log N(\delta, \cH, d)$ and
$N(\delta, B_\epsilon, |\cdot|) = O(\delta^{-p})$ as $\delta \rightarrow 0$.
If we define $\cF = \{ \by \mapsto e^{|\by|/2} \nabla\ell_\eta(\by) : \eta \in \cH\}$,
then $N(\delta,\cH, d)$ is equal to $N(\delta, \cF, d_2)$,
where $d^2_2(f_1, f_2) = \int |f_1(\by) - f_2(\by)|^2 dP_{\eta_0}(\by)$
for $f_1, f_2 \in \cF$.
Let $\cF_j$ be the set of all $j$th coordinate functions of $f \in \cF$ for $j=1, \ldots, m$.
Then, we have
$$
	\log N(\delta, \cF, d_2) 
	\leq \sum_{j=1}^m \log N(\delta/\sqrt{m}, \cF_j, L_2(P_{\eta_0}))
	\leq \sum_{j=1}^m \log N_{[]}(\delta/\sqrt{m}, \cF_j, L_2(P_{\eta_0}))
$$
for every $\delta > 0$.
The bracketing entropy $\log N_{[]}(\delta, \cF_j, L_2(P_{\eta_0}))$ is of order $O(\delta^{-m})$ 
as $\delta \rightarrow 0$ by applying Theorem \ref{thm:bracketing_bound} with 
$\alpha=1$, $V=m$, $d=m$, $r=2$ and a partition 
$$
	\bbR^m = \cup_{j=1}^\infty \Big\{\by \in \bbR^m: |\by| \in [j-1,j)\Big\}.
$$
Therefore, the proof is complete.
\end{proof}

\bigskip

\bigskip

\chapter{Numerical studies}
\label{chap:numerical}

\section{Gibbs sampler algorithm}
\label{sec:gibbs}

Consider the data $(\bY_1, \bX_1), \ldots, (\bY_n, \bX_n)$ generated from the random intercept model
\be \label{eq:lme}
	Y_{ij} = \beta^T X_{ij} + b_i + \epsilon_{ij},
\ee
for $i=1, \ldots, n$ and $j=1, \ldots, m_i$.
In Section \ref{sec:lme}, we proved the BvM theorem 
when the error distribution is modeled as a location-scale mixture of normal densities.
In this section, we only consider location mixtures of normal densities and
the scale parameter of the normal distribution is endowed with a prior.
Then the density of the error distribution can be written by
$\int \phi_\sigma(y-z) dF(z)$ for some $\sigma > 0$ and mixing distribution $F$.
Also, we assume that the distribution of random effects is $N(0, \sigma_b^2)$.
Therefore, the unknown parameters are $(\beta, \sigma, F, \sigma_b)$ and we endow
$N(0, \tau_0^2 I_p)$, $IG(\alpha_0, \lambda_0)$,
$DP_S(\alpha_1, N(0,\tau_1^2))$ and $IG(\alpha_1, \lambda_1)$ prior
for $\beta, \sigma^2, F$ and $\sigma_b$, respectively,
where $IG(\alpha, \lambda)$ is the inverse gamma distribution which has
$$
	x \mapsto \frac{\lambda^\alpha}{\Gamma(\alpha)} x^{-\alpha-1} \exp\Big( -\frac{\lambda}{x}\Big)
$$
as the density, $DP_S$ is the symmetrized Dirichlet process
defined in Section \ref{sec:sdp}, and $I_p$ is the $p\times p$ identity matrix.

We introduce a Gibbs sampler algorithm based on Algorithm \ref{alg:a2} in Section \ref{sec:sdp}.
For this we denote the latent class variable $c_{ij}$ associated with observation $(X_{ij}, Y_{ij})$
and the corresponding location parameter $z_{ij}$ and sign indicator $s_{ij}$ as in Section \ref{sec:sdp}.
Also, let $\vartheta_c$ be the location parameter for the class $c$, that is $z_{ij} = s_{ij} \vartheta_{c_{ij}}$.
Note that for given $\beta, X_{ij}, b_i, \sigma, \sigma_b$ and $z_{ij}$, the response variable $Y_{ij}$ follows
a normal distribution with mean $\beta^T X_{ij} + b_i + z_{ij}$ and variance $\sigma^2$.
The conditional posterior distribution of unobservable quantities are given below.
The conditional posterior distributions (i)--(iv) can be easily calculated by conjugacy.
For notational convenience, the observations are abbreviated in each conditional probability.
The boldface $\bz$ represents all $z_{ij}$ for $i=1, \ldots, n$ and $j=1, \ldots, m_i$
and it is denoted by $\bz_{-ij}$ when $z_{ji}$ is excluded.
Also, ${\bf b}$ and ${\bf b}_{-i}$ are similarly defined.

\bigskip

\ben
	\item[(i)] Generating $\beta$ for given $\sigma, \bz, {\bf b}$
\een
The conditional posterior distribution of $\beta$ is given by
$$
	\beta \;|\; \sigma, \bz, {\bf b} \sim N(\mu_n, \Sigma_n),
$$
where 
\bean
	\mu_n &=& \tau_0^2 \Big(\sigma^2 I_p + \tau_0^2 \sum_{i,j} X_{ij} X_{ij}^T \Big)^{-1}
	\sum_{i,j} \big(Y_{ij} - z_{ij} - b_i\big) X_{ij}
	\\
	\Sigma_n &=& \tau_0^2 \sigma^2 \Big( \sigma^2 I_p + \tau_0^2 \sum_{i,j} X_{ij} X_{ij}^T \Big)^{-1}.
\eean

\bigskip

\ben
	\item[(ii)] Generating $\sigma$ for given $\beta, \bz, {\bf b}$
\een
By the conjugacy, the conditional posterior distribution of  $\sigma^2$ is given by
$$
	\sigma^2 \;|\; \beta, \bz, {\bf b} \sim IG \bigg( \alpha_0 + \frac{1}{2} \sum_i m_i, ~
	\lambda_0 + \frac{1}{2} \sum_{i,j} \big(
	Y_{ij} - \beta^TX_{ij} - z_{ij} - b_i \big)^2 \bigg).
$$

\bigskip

\ben
	\item[(iii)] Generating $b_i$ for given $\beta, \sigma, \sigma_b, \bz$
\een
The conditional posterior distribution of $b_i$ is also normal given by
\bean
	b_i \;|\; \beta, \sigma, \sigma_b, \bz ~\sim~
	N\bigg(\frac{\sigma_b^2}{\sigma^2 + m_i\sigma_b^2} \sum_{j=1}^{m_i} \big(Y_{ij} - \beta^T X_{ij} - z_{ij}\big), ~
	\frac{\sigma^2 \sigma_b^2}{\sigma^2 + m_i\sigma_b^2}\bigg)
\eean
for each $i=1, \ldots, n$.

\bigskip

\ben
	\item[(iv)] Generating $\sigma_b \;|\; {\bf b}$
\een
The conditional posterior distribution of $\sigma_b$ can be similarly calculated by
$$
	\sigma_b^2 \;|\; {\bf b} \sim IG\bigg(\alpha_1 + \frac{n}{2}, ~
	\lambda_1 + \frac{1}{2} \sum_{i=1}^n b_i^2 \bigg).
$$

\bigskip

\ben
	\item[(v)] Generating $z_{ij} \;|\; \beta, b_i, \sigma, \bz_{-ij}$
\een
Instead of generating $z_{ij}$ directly, we sample $c_{ij}, s_{ij}$ and $\vartheta_c$ iteratively
for $i=1,\ldots, n$, $j=1, \ldots, m_i$ and $c=1, 2, \ldots$.
Then, $z_{ij}$ can be determined by $s_{ij} \vartheta_{c_{ij}}$.
Let $e_{ij} = Y_{ij} - \beta^T X_{ij} - b_i$ and $n_{-ij, c}^+$ be the number of observation with
$(i^\prime,j^\prime) \neq (i,j)$, $c_{i^\prime j^\prime}=c$ and $s_{i^\prime j^\prime}=1$.
Similarly we can define $n_{-ij, c}^-$ by replacing $s_{i^\prime j^\prime}=1$ to $s_{i^\prime j^\prime}=-1$.
The conditional distribution of $c_{ij}$ is given by
\bean
	&& p(c_{ij}=c \;|\; \sigma, \bfe, \bfs, \bfc_{-ij}) \\ 
	&&\propto \left\{ \begin{array}{cl}
	n_{-ij, c}^+ \phi_\sigma(e_{ij} - \vartheta_{c_{ij}})
	+ n_{-ij, c}^- \phi_\sigma(e_{ij} + \vartheta_{c_{ij}})
	& \textrm{if $c_{ij} = c_{i^\prime j^\prime}$} \textrm{~for some~} (i^\prime, j^\prime) \neq (i,j)
	\\
	2\alpha \phi_\sigma (e_{ij} - \vartheta_{\rm new}) & 
	\textrm{if $c_{ij} \neq c_{i^\prime j^\prime}$} \textrm{~for all~} (i^\prime, j^\prime) \neq (i,j)
	\end{array}\right.
\eean
where $\vartheta_{\rm new}$ is a random variable from $N(0, \tau_1^2)$.
If a new class is generated, then draw a random variable from
\bean
	N\bigg(\frac{\tau_1^2}{\tau_1^2 + \sigma^2}s_{ij} e_{ij}, ~ \frac{\tau_1^2 \sigma^2}{\tau_1^2 + \sigma^2} \bigg)
\eean
and set $\vartheta_{c_{ij}}$ to this value.
Next, the conditional distribution of $s_{ij}=1$ and $s_{ij}=-1$ are proportional to
$\phi_\sigma(e_{ij} - \vartheta_{c_{ij}})$ and $\phi_\sigma(e_{ij} + \vartheta_{c_{ij}})$, respectively.
Finally, the conditional posterior distribution of $\vartheta_{c}$ is given by
$$
	\vartheta_c \;|\; \sigma, \bfe, \bfs, \bfc ~\sim~
	N\bigg( \frac{\tau_1^2}{n_c\tau_1^2 + \sigma^2} \sum_{c_{ij}=c} s_{ij}e_{ij}, ~
	\frac{\tau_1^2 \sigma^2}{n_c\tau_1^2 + \sigma^2} \bigg)
$$
where $n_c$ is the number of observations with $c_{ij}=c$.

\bigskip

The whole Gibbs sampler algorithm repeat (i)--(v) until the generated Markov chain converges.
The algorithm converges after few dozens of iterations.
Algorithm 3 of \citet{neal2000markov} also can be used to construct a Gibbs sampler algorithm
with the help of a partially collapsed Gibbs sampler algorithm introduced by \citet{van2008partially}.
This algorithm integrates $z_{ij}$ out when it generates the latent class $c_{ij}$,
so the convergence speed can be improved.

\section{Simulation}

In numerical experiments, a dataset is generated from model \eqref{eq:lme} with various error distributions.
Then, the regression parameters $\beta$ are estimated by various methods
including both frequentist's and Bayesian point estimators.
We repeat this procedure $N$ times and
the performance of each method is evaluated by the mean squared error 
$N^{-1} \sum_{k=1}^N |\hat\beta^{(k)}_n - \beta_0|^2$,
where $\hat\beta^{(k)}_n$ is a point estimator in $k$th repetition.
We compare the performance of 5 estimators (two frequentist's ones F1--F2
and three Bayesians B1--B3) with 9 error distributions E1--E9.
In all experiments, we use 2 covariates which follow independent Bernoulli distributions
with success probability 1/2, and the true parameter $\beta_0$ is set to be $(-1,1)^T$.
With regard to the error distribution, Student's t-distributions with 
1, 2, 4, 8, 16 degrees of freedom for E1--E5, the standard normal distribution for E6,
uniform(-3,3) distribution for E7, and mixtures of normal densities for E8 and E9.
More specifically, mixture densities are of the form
\bean
	p(x) = \sum_{k=1}^K \pi_k\Big(\phi_1 (x-z_k) + \phi_1 (x+z_k)\Big)
\eean
with $K=4, (z_1, z_2, z_3, z_4)=(0, 1.5, 2.5, 3.5), (\pi_1, \pi_2, \pi_3, \pi_4) = (0.1, 0.2, 0.15, 0.05)$ for E8,
and $K=4, (z_1, z_2, z_3, z_4)=(0, 1, 2, 4), (\pi_1, \pi_2, \pi_3, \pi_4) = (0.05, 0.15, 0.1, 0.2)$ for E9.
These two densities are depicted in Figure \ref{fig:density}.

For the estimators, F1 is the least square estimator, F2
\footnote{The maximum likelihood estimate can be calculated by solving the 
generalized estimating equation (\citet{liang1986longitudinal}).}
is the maximum likelihood estimator based on normal random effects and normal errors,
B1 and B2 are posterior means under the assumption of normal errors 
without and with normal random effects, respectively.
Note that B1 and B2 are Bayesian correspondences of F1 and F2.
B3 is the posterior mean under the assumption of normal random effects and location mixtures of
normal densities as error density.

\begin{figure}
\begin{center}
\includegraphics[width=100 mm, height=50 mm]{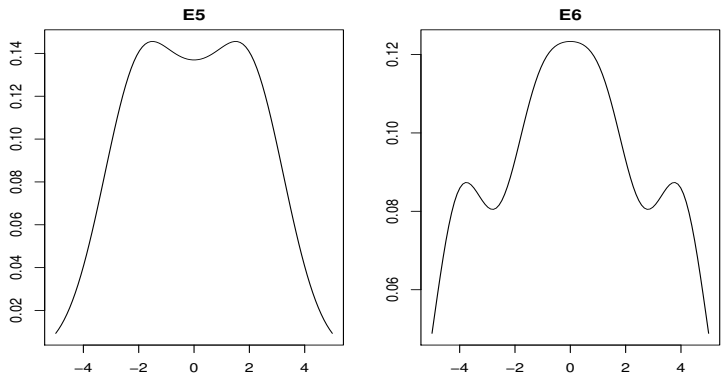}
\end{center}
\caption{Density plots of error distribution in E8 and E9 \label{fig:density}}
\end{figure}

For each experiment, $N=300$ datasets, with $n=20$ and $m_i=5$ for each $i$, are simulated.
The mean squared errors and relative efficiencies comparing with E9 
are summarized in Table \ref{tab:sim_res}.
As we can see, the results for F1 and F2 are similar to their Bayesian counterparts B1 and B2.
Since the Student's t-distribution converges to the standard normal distribution as the degree of freedom increases,
F2 and B2 performs slightly better than B3 in E4--E6.
For the case of Cauchy error E1, B3 also does not work well like other estimators,
because we only considered the location mixture of normal densities.
In other cases, B3 dominates the other estimators as expected.

\begin{table}
\footnotesize
\caption{Mean squared error (and relative efficiency) of each methods F1--F2 and B1--B3 
		 among $N=300$ repetitions for each experiment E1--E9 \label{tab:sim_res}}
\begin{center}
\begin{tabular}{c|rrrrr}
	\hline
    \multicolumn{2}{r}{F1}   & F2    & B1    & B2    & B3 \\ \hline
    E1    & 244.16 & 244.62 & 238.77 & 243.48 & 235.27 \\
          & (1.04)  & (1.04)  & (1.01)  & (1.03)  & (1.00) \\
    E2    & 0.27  & 0.27  & 0.27  & 0.26  & 0.09 \\
          & (3.07)  & (3.06)  & (3.03)  & (2.99)  & (1.00) \\
    E3    & 0.06  & 0.05  & 0.06  & 0.05  & 0.05 \\
          & (1.35)  & (1.12)  & (1.34)  & (1.11)  & (1.00) \\
    E4    & 0.04  & 0.03  & 0.04  & 0.03  & 0.03 \\
          & (1.35)  & (1.00)  & (1.35)  & (1.00)  & (1.00) \\
    E5    & 0.04  & 0.03  & 0.04  & 0.03  & 0.03 \\
          & (1.50)  & (0.98)  & (1.50)  & (0.98)  & (1.00) \\
    E6    & 0.04  & 0.03  & 0.04  & 0.03  & 0.03 \\
          & (1.50)  & (0.98)  & (1.50)  & (0.98)  & (1.00) \\
    E7    & 0.08  & 0.07  & 0.08  & 0.07  & 0.05 \\
          & (1.62)  & (1.40)  & (1.62)  & (1.39)  & (1.00) \\
    E8    & 0.14  & 0.13  & 0.14  & 0.12  & 0.11 \\
          & (1.28)  & (1.18)  & (1.27)  & (1.16)  & (1.00) \\
    E9    & 0.20  & 0.19  & 0.20  & 0.19  & 0.17 \\
          & (1.17)  & (1.13)  & (1.17)  & (1.12)  & (1.00) \\
	\hline
\end{tabular}
\end{center}
\end{table}

\bigskip

\section{Analysis of orthodontic distance growth data}
\label{sec:realdata}

In this section, we analyze the orthodontic distance growth data 
considered previously by \citet{pinheiro2001efficient} and \citet{song2007maximum}.
These data were originally reported in an orthodontic study by \citet{potthoff1964generalized}.
The measurements of the response variable is the distance (in millimeters) from the pituitary gland
to the pterygomaxillary fissure taken repeatedly at 8, 10, 12 and 14 years of age on a sample of 27 children,
comprised of 16 boys and 11 girls.
It is known from the previous studies that there are two outliers in these data.
For analyzing these data, we consider the following linear mixed effects model
\be \label{eq:real_model}
	Y_{ij} = \beta_0 + \beta_1 X_i + \beta_2 T_j + \beta_3 X_i T_j + 
	b_1 + b_2 T_j + \epsilon_{ij},
\ee
where $Y_{ij}$ is the orthodontic distance for the $i$th subject at age $T_j$,
and $X_i$ represents the sex of $i$th subject, coded $X_i=0$ for boys and $X_i=1$ for girls.
The error $\epsilon_{ij}$ follows a symmetric density and
there are two independent random effects $b_1$ and $b_2$ which are assumed to follow
the normal distributions $N(0, \sigma_{b1}^2)$ and $N(0, \sigma_{b2}^2)$, respectively.

To compare the results from the model \eqref{eq:real_model},
we consider four submodels M1--M4: 
$\sigma_{b1}=\sigma_{b2}=0$ and $\epsilon_{ij} \sim N(0, \sigma^2)$ for M1,
$\sigma_{b2}=0$ and $\epsilon_{ij} \sim N(0, \sigma^2)$ for M2,
$\sigma_{b2}=0$ and $\epsilon_{ij}$ follows an unknown symmetric density for M3,
$\epsilon_{ij} \sim N(0, \sigma^2)$ for M4.
The saturated model \eqref{eq:real_model} is denoted by M5.
For the priors of error distributions of M3 and M5, Dirichlet processes mixtures of normal densities,
explained in Section \ref{sec:gibbs}, are used.

The posterior mean and standard deviation of each parameter is summarized in Table \ref{tab:real_res}.
The estimated regression coefficients from the models M3 and M5 are different from the others
because their results do not heavily depend on outliers.
Using the Student's $t$-distribution, \citet{song2007maximum} found the maximum likelihood estimators
from four different models of type \eqref{eq:real_model}:
normal-normal, $t$-normal, $t$-$t$ and normal-$t$ for the distributions of errors and random effects.
The estimated regression coefficient from the model M5 is closest to the result from the third model, 
$t$-$t$, although we only considered the normal random effects.

\begin{table}
\small
\caption{The results of orthodontic distance growth data analysis under five different methods \label{tab:real_res}}
\begin{center}
\begin{tabular}{c|rr|rr|rr|rr|rr}
	\hline
	& \multicolumn{2}{c|}{M1} & \multicolumn{2}{c|}{M2} & \multicolumn{2}{c|}{M3} & \multicolumn{2}{c|}{M4} 
	& \multicolumn{2}{c}{M5}
	\\
	& \multicolumn{1}{c}{Mean} & \multicolumn{1}{c|}{Sd} & \multicolumn{1}{c}{Mean} & \multicolumn{1}{c|}{Sd} 
	& \multicolumn{1}{c}{Mean} & \multicolumn{1}{c|}{Sd} & \multicolumn{1}{c}{Mean} & \multicolumn{1}{c|}{Sd}
	& \multicolumn{1}{c}{Mean} & \multicolumn{1}{c}{Sd}
	\\
	\hline
    $\beta_0$ & 16.031 & 1.412 & 16.184 & 0.986 & 16.651 & 1.021 & 16.208 & 0.904 & 16.621 & 0.958
    \\
    $\beta_1$ & 1.267 & 2.156 & 1.181 & 1.544 & 0.700 & 1.504 & 1.136 & 1.416 & 0.748 & 1.386
    \\
    $\beta_2$ & 0.812 & 0.126 & 0.795 & 0.078 & 0.756 & 0.082 & 0.803 & 0.117 & 0.761 & 0.115
    \\
    $\beta_3$ & -0.326 & 0.192 & -0.314 & 0.122 & -0.278 & 0.118 & -0.319 & 0.175 & -0.272 & 0.170
    \\
    $\sigma$ & 2.253 & 0.156 & 1.395 & 0.111 & 1.163 & 0.203 & 1.376 & 0.115 & 1.162 & 0.203
    \\
    $\sigma_{b_1}$ & \multicolumn{1}{c}{-} & \multicolumn{1}{c|}{-} & 1.795 & 0.294 & 1.820 & 0.292 & 1.147 & 0.483 & 1.067 & 0.406 
	\\
    $\sigma_{b_2}$ & \multicolumn{1}{c}{-} & \multicolumn{1}{c|}{-} & \multicolumn{1}{c}{-} &
    \multicolumn{1}{c|}{-} & \multicolumn{1}{c}{-} & \multicolumn{1}{c|}{-} & 0.335 & 0.049 & 0.332 & 0.048
    \\
    \hline
\end{tabular}
\end{center}
\end{table}

\chapter{Conclusion}
\label{chap:conclusion}

We have shown that the semiparametric Bernstein-von Mises theorem holds
in the location, the linear regression and random intercept models
if the unknown symmetric error distribution is endowed with a Dirichlet process mixture of normal densities.
Our results can be applied to more general models such as linear mixed effects models which have symmetric errors.
The only non-trivial requirement is the consistency of the posterior distribution.

As an extension of our results, we are interested in two Bayesian problems.
The first one is the regression with unknown error distribution when the number of covariates diverges.
In the linear regression problems with increasing regressors,
\citet{bontemps2011bernstein, johnstone2010high, ghosal1999asymptotic}
proved the asymptotic normality of the posterior distributions, but they considered only Gaussian error distributions.
Frequentists also assume the Gaussian error when they consider high-dimensional data
because, otherwise, it is very difficult to find an efficient algorithm and nice asymptotic properties.
Since Bayesian computation is often more convenient than frequentist's one,
we believe that Bayesian method will be a promising tools to analyze high-dimensional data in the near future.

The second problem we consider is to prove the Bernstein-von Mises theorem in semiparametric mixture models
in which there is loss of information.
The generalized linear mixed effects models and the frailty model are important examples.
Before working this paper, our original interest was to prove the Bernstein-von Mises theorem,
or at least the consistency of the posterior distribution, in frailty models.
There are general semiparametric Bernstein-von Mises theorems, 
but it is difficult to apply them for mixture models
because they require the change of parameters.
Since a collection of probability measures is not closed under the subtraction,
it is not easy to apply general approach.
In our main results, we considered only adaptive models, so this change of parameters is not needed.

\appendix

\chapter{Miscellanies}

\section{Posterior consistency under independent observations}
\label{sec:consistency}

In this section, we consider posterior consistency 
under the independent observation $X^{(n)} = (X_1, \ldots, X_n)^T$
which follows
$$
	P_{\theta,\eta}^{(n)} \big( X^{(n)} \in A_1\times\cdots\times A_n \big) 
	= \prod_{i=1}^n P_{\theta,\eta,i}\big(X_i \in A_i \big)
$$
for any product set $A_1\times\cdots\times A_n$.
The final goal of this subsection is to prove Therems \ref{thm:consist_perturb_general} and \ref{thm:consist_general}
which can be used as tools for proving \eqref{eq:conv_perturb_general} and \eqref{eq:consist}.
When $X_1, \ldots, X_n$ are identically distributed,
it is well-known (\citet{ghosal2000convergence}) that 
the posterior convergence rate depends on the Hellinger metric entropy
and the prior concentration rate to Kullback-Leibler type neighborhoods
of the true parameter $(\theta_0, \eta_0)$.
This general result can be extended to non-\iid~cases (\citet{ghosal2007convergence})
and misspecified models (\citet{kleijn2006misspecification}).
The convergence rate of conditional posterior of $\eta$ is also well established
in \citet{bickel2012semiparametric} under a slight misspecification of $\theta$.
All of these results, however, cannot be directly applied to our examples
because we should consider both non-\iid~observation and misspecification.
The aim of this subsection is to prove the consistency of joint posterior $(\theta,\eta)$,
\eqref{eq:consist}, and the consistency of conditional posterior of $\eta$
under $\sqrt{n}$-perturbation of $\theta$, condition \eqref{eq:conv_perturb_general},
when the observation is independent and a Hellinger type metric is used.

Assume that for every $\eta_1, \eta_2 \in \cH$, the Hellinger distance
$h(P_{\theta_0, \eta_1,i}, P_{\theta_0, \eta_2,i})$ does not depend on $i$,
so we can define a metric $d_H$ on $\cH$ by
\be \label{eq:d_H_def} 
	d_H(\eta_1, \eta_2) = h(P_{\theta_0, \eta_1,i}, P_{\theta_0, \eta_2,i}).
\ee
Assume also that
\be \label{eq:h_unif_libschitz}
	\sup_i \sup_{\eta\in\cH} h(P_{\theta,\eta,i}, P_{\theta_0,\eta,i}) = O(|\theta-\theta_0|)
\ee
as $\theta \rightarrow \theta_0$.
When the model $\theta \mapsto P_{\theta,\eta,i}$ is smooth for every $\eta$ and $i$,
\eqref{eq:h_unif_libschitz} is not difficult to prove.
Let
$$
	K_n(\epsilon, M) = \bigcap_{i=1}^\infty
	\left\{ \eta\in\cH: P_{0} \bigg(\sup_{|h| \leq M} 
	-\log \frac{p_{\theta_n(h),\eta,i}}{p_{\theta_0,\eta_0,i}}\bigg) \leq \epsilon^2,
	P_{0} \bigg(\sup_{|h| \leq M} 
	-\log \frac{p_{\theta_n(h),\eta,i}}{p_{\theta_0,\eta_0,i}}\bigg)^2 \leq \epsilon^2
	 \right\}
$$
and $K(\epsilon) = K_n(\epsilon,0)$.
Finally, assume that there exists a neighborhood $U$ of $\theta_0$
and maps $Q_i$ such that
$\sup_i \int Q_i^2(x) P_{\theta_0,\eta_0,i}(x) < \infty$ and
\be\label{eq:quad_bdd_general}
	\sup_{\eta\in\cH}|\ell_{\theta_1, \eta, i}(x) - \ell_{\theta_2, \eta, i}(x)| 
	\leq |\theta_1-\theta_2| \cdot Q_i(x)
\ee
for every $i$ and $\theta_1,\theta_2\in U$.
If \eqref{eq:quad_bdd_general} holds, then there is a universal constant $L > 0$
such that for sufficiently small $\epsilon > 0$,
\bea 
	\big\{\theta: |\theta-\theta_0| < \epsilon^2\big\} \times
	K(\epsilon) &\subset& \bigcap_{i=1}^\infty
	\bigg\{ (\theta,\eta): -P_{0} \Big(
	\log \frac{p_{\theta,\eta,i}}{p_{\theta_0,\eta_0,i}}\Big) \leq L\epsilon^2,
	\nonumber \\
	&&~~~~~~~~~~~~~~~~~~ P_{0} \Big(
	\log \frac{p_{\theta,\eta,i}}{p_{\theta_0,\eta_0,i}}\Big)^2 \leq L\epsilon^2
	 \bigg\} \label{eq:KL_subset}
\eea
as in the proof of Lemma \ref{lem:Lipschitz}.
The following lemma is a generalization of Lemma 8.1 in \citet{ghosal2000convergence}
and Lemma 3.4 in \citet{bickel2012semiparametric}.

\bigskip

\begin{lemma} \label{lem:lower_bdd_general}
Let $(h_n)$ be stochastic and bounded by some $M > 0$.
Then,
\be \label{eq:lower_bdd_general}
	P_0^n \left( \int_\cH \prod_{i=1}^n 
	\frac{p_{\theta_n(h_n),\eta,i}}{p_{\theta_0,\eta_0,i}} d\Pi_\cH(\eta)
	< e^{-(1+C) n\epsilon^2} \Pi_\cH(K_n(\epsilon,M)) \right)
	\leq \frac{1}{C^2 n\epsilon^2}
\ee
for all $C >0, \epsilon > 0$ and $n \geq 1$.
\end{lemma}
\begin{proof}
Let $\tilde\Pi_\cH$ be the probability measure obtained by restricting $\Pi_\cH$ to
$K_n(\epsilon, M)$ and next renormalizing.
Then, by Jensen's inequality, the left hand side of \eqref{eq:lower_bdd_general} can be bounded by
\bean
	&&P_0^n \left( \int \sum_{i=1}^n \log\frac{p_{\theta_n(h_n),\eta,i}}{p_{\theta_0,\eta_0,i}} 
	d\tilde\Pi_\cH(\eta) < -(1+C)n\epsilon^2\right)
	\\
	&\leq& P_0^n \left( \int \sum_{i=1}^n \left(\sup_{|h| \leq M} - 
	\log\frac{p_{\theta_n(h),\eta,i}}{p_{\theta_0,\eta_0,i}} \right)
	d\tilde\Pi_\cH(\eta) > (1+C)n\epsilon^2\right)
	\\
	&=& P_0^n \Bigg( \frac{1}{\sqrt{n}} \sum_{i=1}^n \int
	\bigg(\sup_{|h| \leq M} - \log\frac{p_{\theta_n(h),\eta,i}}{p_{\theta_0,\eta_0,i}} \bigg)^o
	 d\tilde\Pi_\cH(\eta)
	 \\
	&& ~~~~~~~~~~~~~~~~~ 
	> (1+C)\sqrt{n}\epsilon^2 - \frac{1}{\sqrt{n}} \sum_{i=1}^n \int P_0^n
	\bigg( \sup_{|h| \leq M} -\log\frac{p_{\theta_n(h),\eta,i}}{p_{\theta_0,\eta_0,i}}\bigg)
	d\tilde\Pi_\cH(\eta) \Bigg)
	\\
	&\leq& P_0^n \Bigg( \frac{1}{\sqrt{n}} \sum_{i=1}^n \int
	\bigg(\sup_{|h| \leq M} - \log\frac{p_{\theta_n(h),\eta,i}}{p_{\theta_0,\eta_0,i}} \bigg)^o
	 d\tilde\Pi_\cH(\eta)
	> C\sqrt{n}\epsilon^2\Bigg)
	\\
	&\leq& \frac{\sum_{i=1}^n \int P_0^n \big[
	\sup_{|h| \leq M}-\log\big(p_{\theta_n(h),\eta,i}/p_{\theta_0,\eta_0,i}\big)\big]^2 d\tilde\Pi_\cH(\eta)}
	{C^2 n^2\epsilon^4}
	\\
	&\leq& \frac{1}{C^2 n\epsilon^2},
\eean
where the third inequality holds by Markov's and Jensen's inequalities.
\end{proof}

\bigskip

Typically, a certain type of consistent tests is required for posterior consistency,
and in \iid~cases, the Hellinger metric entropy bound of a given model ensures the existence of such tests.
Lemma 3.2 of \citet{bickel2012semiparametric} is an extension of this result to semiparametric models
when the finite dimensional parameter $\theta$ is misspecified.
Lemma \ref{lem:exp_test_general} generalize this to non-\iid~models.

\bigskip

\begin{lemma} \label{lem:exp_test_general}
Suppose that \eqref{eq:d_H_def} is well-defined and \eqref{eq:h_unif_libschitz} holds.
Also, assume that $N(\delta,\cH,d_H) < \infty$ for every $\delta > 0$.
Then, for every $\epsilon > 0$ and $(M_n)$ with $M_n\rightarrow \infty$ and $M_n/\sqrt{n}\rightarrow 0$,
there exist a sequence of tests $(\varphi_n)$ 
and a universal constant $D > 0$ (does not depend on $\epsilon$) such that
$$
	P_0^n \varphi_n \leq e^{-Dn\epsilon^2} ~~~~~~~ \sup_{|h| \leq M_n}
	\sup_{d_H(\eta,\eta_0) > \epsilon} P^n_{\theta_n(h),\eta} (1-\varphi_n) \leq e^{-Dn\epsilon^2}
$$
for large enough $n$.
\end{lemma}
\begin{proof}
Let $\epsilon > 0$ and $(M_n)$, $M_n\rightarrow\infty$ and $M_n/\sqrt{n}\rightarrow 0$, be given.
Then, 
$$
	\sup_i \sup_{|h| \leq M_n}\sup_{\eta\in\cH} 
	h\big(P_{\theta_n(h),\eta,i}, P_{\theta_0,\eta,i}\big)
	= o(1)
$$
by \eqref{eq:h_unif_libschitz}.
Choose $\eta_1\in\cH$ such that $d_H(\eta_1,\eta_0) > \epsilon$.
Then,
\be\label{eq:subset}
	\bigg\{(h,\eta): |h| \leq M_n, ~ d_H(\eta,\eta_1) < \frac{\epsilon}{3} \bigg\}
	\subset
	\bigcap_{i=1}^\infty \bigg\{(h,\eta): h\big(P_{\theta_n(h),\eta,i}, P_{\theta_0,\eta_1,i}\big)
	< \frac{2\epsilon}{3}\bigg\}
\ee
for large enough $n$.
Also,
$$\inf_i h\big(P_{\theta_n(h),\eta,i}, P_{\theta_0,\eta_0,i}\big) > \frac{\epsilon}{3}$$
for every $(h,\eta)$ that is contained in the right hand side of \eqref{eq:subset}.
Note that each set in the right hand side of \eqref{eq:subset} is convex.
Therefore, by the general result known from \citet{birge1984sur}
and \citet{le1986asymptotic} (Lemma 4 on page 478)
and the inequality $1 - x \leq e^{-x}$,
there exists a sequence of tests $\tilde\varphi_n$ such that 
$$
	P_0^n \tilde\varphi_n \leq e^{-n\epsilon^2/18} ~~~~~~~~
	\sup_{|h| \leq M_n}\sup_{d_H(\eta,\eta_1) < \epsilon/3}
	P^n_{\theta_n(h),\eta} (1-\tilde\varphi_n) \leq e^{-n\epsilon^2/18}
$$
for large enough $n$.
Since $N(\epsilon/3, \cH, d_H) < \infty$,
the assertion holds by taking maximum of such tests.
\end{proof}

\bigskip

In the following two theorems, assume that $\Pi_\cH$ does not depend on $n$.

\bigskip

\begin{theorem} \label{thm:consist_perturb_general}
Suppose that \eqref{eq:d_H_def} is well-defined and \eqref{eq:h_unif_libschitz}, \eqref{eq:quad_bdd_general} hold.
Furthermore, assume that $N(\delta,\cH,d_H) < \infty$
and $\Pi_\cH(K(\delta)) > 0$ for every $\delta > 0$ and $M > 0$.
Then, for every $\epsilon > 0$ and bounded stochastic $(h_n)$
$$
	\Pi\big(d_H(\eta,\eta_0) > \epsilon \big| \theta=\theta_n(h_n), X_1, \ldots, X_n\big) 
	\rightarrow 0
$$
in $P_0^n$-probability.
\end{theorem}
\begin{proof}
Let $\epsilon > 0$ be given and $(h_n)$ be a stochastic sequence bounded by $M > 0$.
Let $A_n$ be the set in the left hand side of \eqref{eq:lower_bdd_general}.
By Lemma \ref{lem:exp_test_general}, there exists a test sequence $(\varphi_n)$ and $D > 0$ such that
$$
	P_0^n \varphi_n \rightarrow 0, ~~~~~~~ \sup_{|h| \leq M}
	\sup_{d_H(\eta,\eta_0) > K\epsilon} P^n_{\theta_n(h), \eta} (1-\varphi_n) \leq e^{-DK^2n\epsilon^2}
$$
for every $K, \epsilon > 0$ and large enough $n$.
Then,
\bean
	&&P_0^n \left[\Pi(d_H(\eta,\eta_0) > K\epsilon | \theta=\theta_n(h_n), X_1, \ldots, X_n)\right] \\
	&=& P_0^n \left[ \int_{\{d_H(\eta,\eta_0) > K\epsilon\}} \prod_{i=1}^n 
	\frac{p_{\theta_n(h_n),\eta,i}}{p_{\theta_0,\eta_0,i}} d\Pi_\cH(\eta) \Big/
	\int_\cH \prod_{i=1}^n \frac{p_{\theta_n(h_n),\eta,i}}{p_{\theta_0,\eta_0,i}} d\Pi_\cH(\eta) \right]\\
	&\leq& \frac{e^{(1+C)n\epsilon^2}}{\Pi_\cH(K_n(\epsilon,M))}\cdot P_0^n \left[
	\int_{\{d_H(\eta,\eta_0)> K\epsilon\}} \prod_{i=1}^n 
	\frac{p_{\theta_n(h_n),\eta,i}}{p_{\theta_0,\eta_0,i}} d\Pi_\cH(\eta) \cdot 
	1_{A_n^c}\cdot(1-\varphi_n)\right] + o(1)\\
	&=& \frac{e^{(1+C)n\epsilon^2}}{\Pi_\cH(K_n(\epsilon,M))}
	\cdot \sup_{|h| \leq M} \sup_{d(\eta,\eta_0) > K\epsilon}P_{\theta_n(h_n),\eta}^n (1-\varphi_n) + o(1)\\
	&\leq& \frac{e^{(1+C)n\epsilon^2 -DK^2n\epsilon^2}}{\Pi_\cH(K_n(\epsilon,M))} + o(1)
\eean
for every $C > 0$.
If we choose $C, K > 0$ satisfying $1+C < DK^2$, then
the last term converges to 0 because $\Pi_\cH\big( K_n(\epsilon,M)\big) \geq \Pi_\cH(\epsilon/L) > 0$
for some $L > 0$ and large enough $n$.
Since $\epsilon > 0$ can be chosen arbitrarily small, we have the desired result.
\end{proof}

\bigskip

\begin{theorem} \label{thm:consist_general}
Suppose that \eqref{eq:d_H_def} is well-defined,
\eqref{eq:h_unif_libschitz} and \eqref{eq:quad_bdd_general} hold, and
$N(\delta,\cH,d_H) < \infty$, $\Pi_\cH(K(\delta)) > 0$ for every $\delta > 0$.
Furthermore, assume that $\Pi_\Theta$ is thick at $\theta_0$ and
there exists an estimator $\hat\theta_n$ for $\theta$ satisfying 
\be \label{eq:sqrtn_estimator}
	\sup_{\eta\in\cH}P_{\theta,\eta}^n\big(\sqrt{n}|\hat\theta_n - \theta| > M_n \big) \rightarrow 0
\ee
for every $M_n \rightarrow \infty$.
Then, for every $\epsilon > 0$,
$$
	\Pi \big( |\theta-\theta_0| < \epsilon, d_H(\eta,\eta_0) < \epsilon | X_1, \ldots, X_n\big)
	\rightarrow 1
$$
in $P_0^n$-probability.
\end{theorem}
\begin{proof}
For given $(M_n)$, $M_n\rightarrow \infty$ and $M_n/\sqrt{n}\rightarrow 0$, 
there exists a sequence of tests $(\tilde\varphi_n)$ satisfying
the assertion of Lemma \ref{lem:exp_test_general}.
Therefore, by combining Theorem 2.2 in \citet{wu2008posterior}
and \eqref{eq:KL_subset},
it is sufficient to show that 
there exists a sequence of tests $(\varphi_n)$ satisfying
$$
	P_0^n \varphi_n \rightarrow 0 ~~~~~~~ \sup_{|h| > M_n}
	\sup_{\eta \in \cH} P^n_{\theta_n(h), \eta} (1-\varphi_n) \rightarrow 0
$$
because the existence of uniformly consistent tests implies 
the existence of exponentially consistent tests 
(see \citet{lecam1973convergence} or Lemma 7.2 of \citet{ghosal2000convergence}).
Since $\hat\theta_n$ is a $n^{-1/2}$-consistent estimator for $\theta$,
if we define $\varphi_n = 1_{\{\sqrt{n}|\hat\theta_n-\theta_0| > M_n/2\}}$,
then, $P_0^n \varphi_n \rightarrow 0$ and
\bean
	\sup_{|h| > M_n}\sup_{\eta \in \cH} P^n_{\theta_n(h), \eta} (1-\varphi_n)
	&\leq& \sup_{|h| > M_n}\sup_{\eta \in \cH} P^n_{\theta_n(h), \eta}
	\left( \sqrt{n}|\theta_n(h)-\theta_0| - \sqrt{n}|\hat\theta_n-\theta_n(h)| \leq M_n/2 \right)\\
	&\leq& \sup_{|h| > M_n}\sup_{\eta \in \cH} P^n_{\theta_n(h), \eta}
	\left(\sqrt{n}|\hat\theta_n-\theta_n(h)| \geq M_n/2\right)\\
	&\rightarrow& 0
\eean
as $n \rightarrow \infty$.
This completes the proof.
\end{proof}

\bigskip

\section{Semiparametric mixtures}
\label{sec:smixture}

In this section, we prove some technical lemmas for semiparametric mixtures.
Let $\Theta$ be an arbitrary parameter space and
$\cH$ is the set of all probability measures
whose supports are contained in a compact subset $[-M,M]^d$ of $\bbR^d$.
For a given family of kernel densities 
$$\left\{x \mapsto p_\theta(x|z): \theta\in\Theta, z \in[-M,M]^d\right\}$$
and $(\theta,\eta) \in \Theta\times\cH$,
let $p_{\theta, \eta}(x) = \int p_\theta(x|z) d\eta(z)$ be the density
of the probability measure $P_{\theta,\eta}$ with respect to the Lebesgue measure $\mu$ on $\bbR^k$.
For given $\theta_0 \in \Theta$ and $\eta_0\in\cH$, let
$$
	K(\epsilon) = \left\{\eta\in\cH: P_{\theta_0,\eta_0}
	\bigg(-\log\frac{p_{\theta_0,\eta}}{p_{\theta_0,\eta_0}}\bigg) \leq \epsilon^2, ~
	\bigg(-\log\frac{p_{\theta_0,\eta}}{p_{\theta_0,\eta_0}}\bigg)^2 \leq \epsilon^2
	\right\}
$$
and define a metric $d_H$ on $\cH$ by
$d_H(\eta_1,\eta_2) = h(P_{\theta_0,\eta_1}, P_{\theta_0,\eta_2})$.
A prior on $\cH$ is denoted by $\Pi_\cH$.

\bigskip

\begin{lemma} \label{lem:entropy_bound}
Assume that:
\ben
\item[(i)] $\{ P_{\theta_0, \eta} : \eta \in \cH\}$ is uniformly tight.
\item[(ii)] For any compact $K \subset \bbR^k$, $\{z\mapsto p_{\theta_0}(x|z): x \in K\}$
			is an equicontinuous family of functions from $[-M,M]^d$ to $\bbR$.
\een
Then, $N(\epsilon, \cH, d_H) < \infty$ for all $\epsilon > 0$.
\end{lemma}

\begin{proof}
For a given $\epsilon > 0$, we have a compact set $K$ satisfying 
$\sup_\eta  P_{\theta_0, \eta}(K^c) < \epsilon$, by (i).
Since equicontinuouity on a compact set implies uniform equicontinuity,
we have a finite partition $\{B_l\}_{l=1}^L$ of $[-M,M]^d$ such that $z, \tilde z \in B_l$
for some $l$ implies $\sup_{x \in K} |p_{\theta_0}(x|z) - p_{\theta_0}(x|\tilde z)| < \epsilon/\mu(K)$.
Pick $z_l \in B_l$ for each $l=1,\ldots,L$ and choose a large integer $N$ such that $1/N \leq \epsilon$.
Let 
$$
	\cH_d = \left\{\eta\in\cH: \eta(\cdot) = \sum_{l=1}^L q_l \delta_{z_l}(\cdot),
	~~q_l = \frac{j_l}{NL}~\text{for some}~j_l\in \bbZ^+\right\},
$$
where $\delta_z(\cdot)$ is the Dirac measure at $z$ and $\bbZ^+$ is the set of nonnegative integers.
Since $|\cH_d| < \infty$ we can define, for any $\eta \in \cH$,
$\eta_d = \argmin_{\zeta \in \cH_d} \sum_{l=1}^L |\eta(B_l) - \zeta(B_l)|$.
Then, it is not difficult to show that 
$\max_l |\eta(B_l) - \eta_d(B_l)| \leq (NL)^{-1} \leq \epsilon/L$.
Since $d_H (\eta_1, \eta_2) = d_H(P_{\theta_0, \eta_1}, P_{\theta_0, \eta_2})
\leq \sqrt{d_V(P_{\theta_0, \eta_1}, P_{\theta_0, \eta_2})}$ and
\bean
	d_V(P_{\theta_0, \eta}, P_{\theta_0, \eta_d}) &=& 2\epsilon + \int_K \left| \int p_{\theta_0}(x|z) d\eta(z) -
	\int p_{\theta_0}(x|z) d\eta_d(z) \right| d\mu(x)\\
	&\leq& 2\epsilon + \int_K \sum_{l=1}^L \Big[ p_{\theta_0}(x|z_l) \cdot \left|\eta(B_l) - \eta_d(B_l)\right| \\
	&& ~~~~~ + ~ \int_{B_l} \left|p_{\theta_0} (x|z) - p_{\theta_0} (x|z_l)\right| ~ d(\eta + \eta_d)(z) \Big] d\mu(x) \\
	&\leq& 5 \epsilon
\eean
we get the desired result because $\epsilon > 0$ is arbitrary.
\end{proof}

\bigskip

\begin{lemma} \label{lem:KL_bound}
Assume that:
\ben
\item[(i)] $\{ P_{\theta_0, \eta} : \eta \in \cH\}$ is uniformly tight.
\item[(ii)] For any compact set $K \subset \bbR^k$, 
			$\{x\mapsto p_{\theta_0}(x|z) : z \in [-M,M]^d\}$ is an equicontinuous family of functions 
			from $K$ to $\bbR$.
\item[(iii)] For all $x$, $z \mapsto p_{\theta_0}(x|z)$ is bounded and continuous.
\item[(iv)] $\Pi_\cH(U) > 0$ for all weak neighborhood $U$ of $\eta_0$
\item[(v)] $\sup_{h(\eta,\eta_0)<\delta}\int_{\{p_{\theta_0, \eta_0}/p_{\theta_0, \eta} \geq e^{1/\gamma}\}} 
	  (p_{\theta_0, \eta_0}/p_{\theta_0, \eta})^\gamma d P_{\theta_0, \eta_0} < \infty$
	  for some $\gamma \in (0,1]$ and $\delta > 0$.
\een
Then, $\Pi_\cH\bigl( K(\epsilon) \bigr) > 0$ for every $\epsilon > 0$.
\end{lemma}

\begin{proof}
For a given $\epsilon > 0$, we have a compact set $K$ satisfying
$\sup_\eta  P_{\theta_0, \eta}(K^c) < \epsilon$, by (i).
Using (ii), choose $\delta > 0$ so that $x_1, x_2 \in K$ and $|x_1 - x_2| < \delta$
implies $\sup_z | p_{\theta_0}(x_1|z) - p_{\theta_0}(x_2|z)| < \epsilon/\mu(K)$.
Let $U_1, \ldots, U_m$ be a partition of $K$ with $diam(U_i) < \delta$.
Pick $x_i \in U_i$ for each $i$ and let $\mu_i = \mu(U_i)$ and $f_i(z) = \mu_i ~p_{\theta_0}(x_i|z)$.
Then,
\bean
	d_V (p_{\theta_0, \eta}, p_{\theta_0, \eta_0}) &\leq&
	 \int_K \left|\int p_{\theta_0}(x|z) d\eta(z) - \int p_{\theta_0}(x|z) d\eta_0(z) \right| d\mu(x) + 2\epsilon\\
	&=& ~ \sum_{i=1}^m \int_{U_i} \left|\int p_{\theta_0}(x|z)d\eta(z) - \int p_{\theta_0}(x|z)d\eta_0(z)\right| d\mu(x) + 2\epsilon\\
	&\leq& ~ \sum_{i=1}^m \left|\int f_i(z) d\eta(z) - \int f_i(z) d\eta_0(z)\right| + 4\epsilon
\eean
Therefore,
$$\left\{F: d_V (p_{\theta_0, \eta}, p_{\theta_0, \eta_d}) < 5\epsilon \right\} \supset
\left\{ F : \bigcap_{i=1}^m \Big| \int f_i(z) d\eta(z) - \int f_i(z) d\eta_0(z) \Big| < \frac{\epsilon}{m}\right\}$$
and (iii) and (iv) implies every total variation neighborhood of $\eta_0$ has positive prior mass.
Since the total variation and Hellinger metrics are topologically equivalent,
(v) and Theorem 5 in \citet{wong1995probability} yield the desired result.
\end{proof}

\bigskip

\begin{lemma} \label{lem:Hellinger_ineq}
For any $\theta \in \Theta$,
\[
  \sup_{\eta\in\cH} h (P_{\theta,\eta},P_{\theta_0,\eta}) \leq 
  \sup_{z\in[-M,M]^d} h\bigl(P_{\theta}(\cdot|z), P_{\theta_0}(\cdot|z)\bigr),
\]
where $P_\theta(\cdot|z)$ is the probability measure with density $x\mapsto p_\theta(x|z)$.
\end{lemma}
\begin{proof}
By the Cauchy-Scwartz ineqaulity, we have
\bean
	h^2(P_{\theta, \eta}, P_{\theta_0, \eta}) 
	&=& 2\left(1 - \int \left[\int p_{\theta} (x|z) d\eta(z) 
	\int p_{\theta_0} (x|z) d\eta(z) \right]^{1/2} d\mu(x)\right) \\
	&\leq& 2\left(1 - \int \int \left[p_{\theta} (x|z)p_{\theta_0} (x|z)\right]^{1/2} d\mu(x) d\eta(z) \right)\\
	&=& \int 2\left[ 1- \int \left[p_{\theta} (x|z)p_{\theta_0} (x|z)\right]^{1/2} d\mu(x) \right] d\eta(z)\\
	&=& \int h^2\big(P_{\theta}(\cdot|z), P_{\theta_0}(\cdot|z)\big) d\eta(z) \\
	&\leq& \sup_z h^2\big(P_{\theta}(\cdot|z), P_{\theta_0}(\cdot|z)\big)
\eean
which is the desired result.
\end{proof}

\bigskip

Now, assume that $\Theta$ is an open subset of Euclidean space
and let
$$
	K_n(\epsilon,M) = \left\{\eta\in\cH: P_{\theta_0,\eta_0}
	\left(\sup_{|h|\leq M}-\log\frac{p_{\theta_n(h),\eta}}{p_{\theta_0,\eta_0}}\right) \leq \epsilon^2, ~
	\left(\sup_{|h|\leq M}-\log\frac{p_{\theta_n(h),\eta}}{p_{\theta_0,\eta_0}}\right)^2 \leq \epsilon^2
	\right\}
$$
for every $\epsilon > 0$ and $M \geq 0$.

\bigskip

\begin{lemma} \label{lem:Lipschitz}
Assume that there exist a function $x\mapsto Q(x)$ with $\int Q^2(x) dP_{\theta_0,\eta_0}(x)< \infty$
and an open neighborhood $U$ of $\theta_0$ such that
\be \label{eq:kernel_libschitz}
	\sup_z \left| \log \frac{p_{\theta_1}}{p_{\theta_2}}(x|z) \right| \leq Q(x)\cdot|\theta_1 - \theta_2|
\ee
for all $\theta_1, \theta_2 \in U$.
Then, there is a universal constant $L > 0$ such that 
for all $M > 0$ and $\epsilon > 0$,
$K(\epsilon) \subset K_n(L\epsilon,M)$ for large enough $n$.
\end{lemma}
\begin{proof}
Let $M > 0$ be given.
By condition \eqref{eq:kernel_libschitz},
\be \nonumber
	\left|\log\frac{p_{\theta_1,\eta}}{p_{\theta_2,\eta}}(x)\right| 
	\leq \sup_z\left|\log\frac{p_{\theta_1}}{p_{\theta_2}}(x|z)\right|
	\leq Q(x)\cdot|\theta_1-\theta_2|
\ee
for $\theta_1, \theta_2 \in U$.
For $|h| \leq M$, we have
\bean
    \sup_{|h|\leq M}-\log\frac{p_{\theta_n(h),\eta}}{p_{\theta_0,\eta_0}}(x)
    &\leq& -\log\frac{p_{\theta_0,\eta}}{p_{\theta_0,\eta_0}}(x) +
    \sup_{|h|\leq M}\left|\log\frac{p_{\theta_n(h),\eta}}{p_{\theta_0,\eta}}(x)\right| \\
    &\leq& -\log\frac{p_{\theta_0,\eta}}{p_{\theta_0,\eta_0}}(x) + \frac{M}{\sqrt{n}}\cdot Q(x)
\eean
and
\bean
	\left(\sup_{|h|\leq M}-\log\frac{p_{\theta_n(h),\eta}}{p_{\theta_0,\eta_0}}(x)\right)^2 
	&\leq& 2\left(\log\frac{p_{\theta_0,\eta}}{p_{\theta_0,\eta_0}}(x)\right)^2 +
	2\left(\sup_{|h|\leq M}\left|\log\frac{p_{\theta_n(h),\eta}}{p_{\theta_0,\eta}}(x)\right|\right)^2 \\
	&\leq& 2\left(\log\frac{p_{\theta_0,\eta}}{p_{\theta_0,\eta_0}}(x)\right)^2 +
	\frac{2M^2}{n}\cdot Q^2(x).
\eean
This completes the proof.
\end{proof}

\bigskip

\section{Symmetrized Dirichlet processes}
\label{sec:sdp}

In this section, we address the properties of symmetrized Dirichlet processes.
Let $(\Omega, \cF, \nu)$ be a probability space and $\cR$ be the Borel $\sigma$-field on $\bbR$.
Let $M(\bbR)$ be the set of all probability measures on $(\bbR, \cR)$,
equipped with a metric induced by the weak convergence of probability measures,
and $\cM$ be the Borel $\sigma$-field of $M(\bbR)$ with respect to this metric.

For given $\alpha > 0$ and a probability measure $P_0$ on $(\bbR, \cR)$, the law of a measurable map
$P: (\Omega, \cF) \rightarrow (M(\bbR), \cM)$ is called the \emph{Dirichlet process} with parameter $(\alpha, P_0)$,
denoted by $DP(\alpha, P_0)$, if
$$\big(P(A_1), \ldots, P(A_k)\big) \sim \cD\big(\alpha P(A_1), \ldots, \alpha P(A_k)\big)$$
for every finite partition $A_1, \ldots, A_k \in \cR$ of $\bbR$,
where $\cD$ denotes the Dirichlet distribution.
Then, the \emph{symmetrized Dirichlet process}, denoted by $DP_S(\alpha, P_0)$,
is defined by the law of
$$\frac{1}{2} (P + P^-),$$
where $P^-(A) = P(-A)$ for all $A \in \cR$.

\bigskip

\begin{lemma} \label{lem:dir_sym}
If two probability measures $P_1$ and $P_2$ on $(\bbR, \cR)$ satisfies $P_1 + P_1^- = P_2 + P_2^-$,
then for $\alpha > 0$, $DP_S(\alpha, P_1)$ and $DP_S(\alpha, P_2)$ have the same distribution.
\end{lemma}
\begin{proof}
By the construction of \citet{sethuraman1994constructive},
if $p_i \stackrel{\iid}{\sim} beta(1, \alpha)$ and
$\theta_i \stackrel{\iid}{\sim} P_0$, then the law of
$$\frac{1}{2} \sum_{i=1}^\infty p_i \big(\delta_{\theta_i} + \delta_{-\theta_i}\big)$$
is equal to $DP_S(\alpha, P_0)$ for any probability measure $P_0$.
Therefore, it is sufficient to show that
$X_1 \sim P_1$ and $X_2 \sim P_2$
implies $|X_1| \stackrel{d}{=} |X_2|$.
This follows from the condition $P_1 + P_1^- = P_2 + P_2^-$.
\end{proof}

\bigskip

\begin{lemma} \label{lem:dir_s_conjugacy}
Assume that $P$ is endowed with a $DP_S(\alpha, P_0)$ prior, and for given $P$, $\theta_1, \ldots, \theta_n$ 
are independent and identically distributed by $P$.
Then, the posterior distribution of $P$ given $\theta_1, \ldots, \theta_n$ is
$$DP_S\bigg(\alpha + n, \frac{\alpha P_0 + \sum_{i=1}^n \delta_{\theta_i}}{\alpha+n}\bigg),$$
that is, $DP_S$ is conjugate.
\end{lemma}
\begin{proof}
It is sufficient by conjugacy to prove the assertion in the case of $n=1$.
Let $\tilde P \sim DP(\alpha, P_0)$ and $S$ be an independent binary random variable
with $Pr(S=1) = Pr(S=0) = 1/2$.
For given $\tilde P$ and $S$, the conditional distribution of $\theta$ is
$\tilde P$ or $\tilde P^-$ according as $S=1$ or $S=0$.
It is sufficient to show that for given $\theta$ the conditional distribution of $(\tilde P + \tilde P^-)/2$
is equal to $$DP_S\bigg(\alpha + 1, \frac{\alpha P_0 + \delta_\theta}{\alpha+1}\bigg).$$

Note first that $\tilde P^-$ follows $DP(\alpha, P^-)$.
Therefore, conditional on $\theta$ and $S$, the law of $\tilde P$ is
\bea
	DP\bigg(\alpha+1, \frac{\alpha P_0 + \delta_\theta}{\alpha+1}\bigg) ~~~~~ \textrm{if $S=1$}
	\label{eq:dp_cond1}\\
	DP\bigg(\alpha+1, \frac{\alpha P_0^- + \delta_\theta}{\alpha+1}\bigg)
	~~~~~ \textrm{if $S=0$} \label{eq:dp_cond2}
\eea
by the conjugacy of the Dirichlet process.
Therefore, conditional on $\theta$, the law of $\tilde P$ is \eqref{eq:dp_cond1}
with probability $P(S=1|\theta)$,
or the law of $\tilde P^-$ is \eqref{eq:dp_cond2} with probability $P(S=0|\theta)$.
In both cases, the law of $(\tilde P + \tilde P^-)/2$ is
$$DP_S\bigg(\alpha+1, \frac{\alpha P_0+\delta_\theta}{\alpha+1}\bigg)$$
by Lemma \ref{lem:dir_sym}.
\end{proof}

\bigskip

\begin{lemma}
Assume that $P$ is endowed with a $DP_S(\alpha, P_0)$ prior, and for given $P$, $\theta_1, \theta_2, \ldots$ 
are independent and identically distributed by $P$.
Then, conditional on $\theta_1, \ldots, \theta_n$, the predictive distribution of $\theta_{n+1}$ is given by
\be \label{eq:sdp_polya}
	\frac{\alpha}{2(\alpha+n)} \big(P_0+P_0^-\big)
	+ \frac{1}{2(\alpha+n)} \sum_{i=1}^n \big(\delta_{\theta_i} + \delta_{-\theta_i}\big)
\ee
for every $n \geq 0$.
(When $n=0$, the predictive distribution is the marginal distribution of $\theta_1$ and
the summation is defined by zero.)
\end{lemma}
\begin{proof}
The marginal distribution of $\theta_1$ follows by the Sethuraman's construction of the Dirichlet process.
The remainder follows by Lemma \ref{lem:dir_s_conjugacy}.
\end{proof}

\bigskip

Now, consider the observation $X_1, \ldots, X_n$ 
which is generated from a symmetrized Dirichlet process mixture model, that is,
\bea
	P &\sim& DP_S(\alpha, P_0)
	\nonumber \\
	\theta_1, \ldots, \theta_n | P &\stackrel{\iid}{\sim}& P
	\nonumber \\
	X_i | \theta_i &\sim& f_{\theta_i} \label{eq:sdp_mixture}
\eea
for a class of densities $\{f_\theta: \theta \in \Theta\}$.
This is a direct extension of a Dirichlet process mixture model
which is popularly used in nonparametric Bayesian data analysis.
There are many interesting Markov chain Monte Carlo algorithms to infer Dirichlet process mixture models
and they can be naturally extended to symmetrized Dirichlet process models.
We refer to \citet{neal2000markov} for a nice review on these algorithms.
We only consider conjugate algorithms, where conjugacy means that 
$P_0$ in \eqref{eq:sdp_mixture} is a conjugate prior for model $\{f_\theta: \theta \in \Theta\}$.

The first algorithm, which samples $\theta_1, \ldots, \theta_n$ iteratively,
is an extension of algorithms proposed by \citet{escobar1994estimating} and \citet{escobar1995bayesian}.
We can derive the conditional posterior of $\theta_i$ from the marginal distribution \eqref{eq:sdp_polya} by
\be \label{eq:theta_gen}
	\theta_i | \btheta_{-i}, X_i ~ \propto ~ r_i H_i 
	+ \sum_{j\neq i} \Big( f_{\theta_j}(X_i) \cdot \delta_{\theta_j} 
	+ f_{-\theta_j}(X_i) \cdot \delta_{-\theta_j} \Big),
\ee
where $\btheta_{-i} = (\theta_1, \ldots, \theta_{i-1}, \theta_{i+1}, \ldots, \theta_n)$
and
$$
	r_i = \alpha \cdot \int f_\theta(X_i) \;d\big(P_0 + P_0^-\big)(\theta).
$$
Here, $H_i$ is the posterior distribution for $\theta$ based on the prior $(P_0 + P_0^-)/2$
and the single observation $X_i$, with likelihood $f_\theta$.
The algorithm is summarized in Algorithm \ref{alg:a1}.
This algorithm is simple and intuitive but the convergence to the stationary distribution
may be rather slow, so inefficient as noted in \citet{neal2000markov}.

\bigskip

\alglanguage{pseudocode}
\begin{algorithm}
\caption{An Gibbs sampler algorithm generating $\theta_1, \ldots, \theta_n$ directly \label{alg:a1}}
\begin{algorithmic}
\State Initialize $\theta_1, \ldots, \theta_n$
\Repeat
\For {i=1, \ldots, n}
\State Sample $\theta_i$ from \eqref{eq:theta_gen}
\EndFor
\Until{convergence}
\end{algorithmic}
\end{algorithm}

\bigskip

The next algorithm is an extension of algorithms proposed 
in \citet{bush1996semiparametric} and \citet{west1993hierarchical}.
If $X_1, \ldots, X_n$ are observations from the symmetrized Dirichlet process mixture,
the generative model \eqref{eq:sdp_mixture} can be written by
\bean
	p(c_i = c | c_1, \ldots, c_{i-1}) &=& \left\{ \begin{array}{cl}
	\alpha / (\alpha + i -1) & \textrm{if $c \neq c_j$ for all $j=1, \ldots, i-1$} \\
	1 / (\alpha + i -1) & \textrm{if $c = c_j$ for some $j=1, \ldots, i-1$} \end{array} \right.
	\\
	\vartheta_1, \vartheta_2, \cdots &\stackrel{\iid}{\sim}& (P_0 + P_0^-)/2
	\\
	p(s_i=1) = p(s_i=-1) &=& 1/2 ~~~ \textrm{for $i=1, \ldots, n$}
	\\
	X_i | \vartheta_{c_i}, s_i &\sim& f_{s_i \vartheta_{c_i}},
\eean
where $c_i$ indicates which latent class is associated with observation $X_i$.
For each class, $c$, the parameters $\vartheta_c$ determine the distribution of observations from that class.
To build a Gibbs sampler algorithm we include sign indicators $s_i$ for each observation
in the generative model.
The conditional distribution of $c_i$ is given by
\bea
	P(c_i = c | \bfc_{-i}, \bm{\vartheta}, s_i, X_i) 
	\propto \left\{ \begin{array}{cl}
	\frac{\alpha}{2} \int f_\theta(X_i) d(P_0+P_0^-)(\theta) 
	& ~~ \textrm{if $c \neq c_j$ for all $j\neq i$}\\
	n_{-i,c} f_{s_i \vartheta_c}(X_i) & ~~ \textrm{if $c = c_j$ for some $j\neq i$},
	\end{array}\right. \label{eq:c_gen}
\eea
where $\bfc_{-i} = (c_1, \ldots, c_{i-1}, c_{i+1}, \ldots, c_n)$ and 
$n_{-i,c}$ is the number of $j$'s with $j\neq i$ and $c_i=c$.
If generated $c_i$ is different from $c_j$ for all $j\neq i$, draw a value for $\vartheta_{c_i}$
from $H_i$, where $H_i$ is the posterior distribution for $\theta$ based on
the prior $(P_0 + P_0^-)/2$ and the single observation $X_i$ with likelihood
$f_{s_i \vartheta_c}(X_i)$.
Next, the conditional distribution of $\vartheta_c$ and $s_i$ are given by
\bea
	dP(\vartheta_c| \bfc, {\bf s}, X_1, \ldots, X_n) &\propto&
	\prod_{c_i=c} f_{s_i \vartheta_c}(X_i) \cdot d(P_0 + P_0^-)(\vartheta_c)
	\label{eq:vartheta_gen} \\
	p(s_i |c_i, s_i, \vartheta_{c_i}, X_i) &\propto& f_{s_i \vartheta_{c_i}}(X_i), \label{eq:sgn_gen}
\eea
respectively.
The algorithm is summarized in Algorithm \ref{alg:a2}.

\alglanguage{pseudocode}
\begin{algorithm}
\caption{An alternative Gibbs sampler algorithm \label{alg:a2}}
\begin{algorithmic}
\State Initialize $c_1, \ldots, c_n, s_1, \ldots, s_n$ and $\vartheta_1, \vartheta_2, \ldots$
\Repeat
\For {i=1, \ldots, n}
\State Sample $c_i$ from \eqref{eq:c_gen}, and if $c_i\neq c_j$ for all $j \neq i$, then sample $\vartheta_{c_i}$ from $H_i$
\State Sample $s_i$ from \eqref{eq:sgn_gen}
\EndFor
\For {c=1, 2, \ldots}
\State Sample $\vartheta_c$ from \eqref{eq:vartheta_gen}
\EndFor
\Until{convergence}
\end{algorithmic}
\end{algorithm}

\newpage
\addcontentsline{toc}{chapter}{Bibliography}
\bibliographystyle{apalike}
\bibliography{Chae_thesis_arxiv}


\end{document}